%% file: Submitted_Version_Semifree_Hamiltonian_II.tex
\documentclass[a4paper, 10pt]{amsart}
\usepackage[top=80pt,bottom=65pt,left=70pt,right=70pt]{geometry}
\usepackage{graphicx, eepic}
\usepackage{times}

\normalfont
\usepackage{bbm}
\usepackage{xcolor}
\usepackage{verbatim}
\usepackage{kotex}
\usepackage{mathrsfs}
\usepackage{amscd}
\usepackage{hhline}
\usepackage{float}
\usepackage[all]{xy}
\usepackage[subnum]{cases}
\ifpdf
\usepackage{tikz}
\fi
\usetikzlibrary{arrows,matrix,positioning}

\usepackage{makecell}
\usepackage{adjustbox}

\usepackage{amsthm,amsmath,amsfonts,amssymb}
\usepackage{hyperref}
\hypersetup{colorlinks}
\hypersetup{
    bookmarks=true,         
    unicode=false,          
    pdftoolbar=true,        
    pdfmenubar=true,        
    pdffitwindow=false,     
    pdfstartview={FitH},    
    pdftitle={My title},    
    pdfauthor={Author},     
    pdfsubject={Subject},   
    pdfcreator={Creator},   
    pdfproducer={Producer}, 
    pdfkeywords={keyword1} {key2} {key3}, 
    pdfnewwindow=true,      
    colorlinks=true,       
    linkcolor=blue,          
    citecolor=red,        
    filecolor=violet,      
    urlcolor=violet       
}
\usepackage{multirow, url}
\usepackage{color}
\usepackage[all]{xy}
\theoremstyle{plain}
\newtheorem{theorem}{Theorem}[section]

\newtheorem{proposition}[theorem]{Proposition}
\newtheorem{lemma}[theorem]{Lemma}

\newtheorem{corollary}[theorem]{Corollary}

\theoremstyle{definition}

\newtheorem{conjecture}[theorem]{Conjecture}
\newtheorem{definition}[theorem]{Definition}

\newtheorem{example}[theorem]{Example}

\theoremstyle{remark}
\newtheorem{remark}[theorem]{Remark}

\newtheorem{notation}[theorem]{Notation}

\newcommand{\C}{\mathbb{C}}

\newcommand{\Q}{\mathbb{Q}}
\newcommand{\R}{\mathbb{R}}

\newcommand{\Z}{\mathbb{Z}}
\newcommand{\p}{\mathbb{C}P}

\newcommand{\mcal}{\mathcal}

\def\mcal{\mathcal}
\def\frak{\mathfrak}

\newcommand{\ds}{\displaystyle}
\newcommand{\vs}{\vspace}
\newcommand{\hs}{\hspace}

\numberwithin{equation}{section} \numberwithin{table}{section}

\begin{document}                                                                          

\title{Classification of six dimensional monotone symplectic manifolds admitting semifree circle actions II}
\author{Yunhyung Cho}
\address{Department of Mathematics Education, Sungkyunkwan University, Seoul, Republic of Korea. }
\email{yunhyung@skku.edu}

\begin{abstract}
	Let $(M,\omega_M)$ be a six dimensional closed monotone symplectic manifold admitting an effective semifree Hamiltonian $S^1$-action.
	We show that if the maximal and the minimal fixed component are both two dimensional, then $(M,\omega_M)$ is $S^1$-equivariantly symplectomorphic to some 
	K\"{a}hler Fano manifold $(X, \omega_X, J)$ equipped with a certain holomorphic Hamiltonian $S^1$-action. We also give a complete list of all such Fano manifolds together with 
	an explicit description of the corresponding $S^1$-actions. 
\end{abstract}
\maketitle
\setcounter{tocdepth}{1} 
\tableofcontents

\section{Introduction}
\label{secIntroduction}

A {\em Fano variety} is a smooth projective variety whose anti-canonical bundle is ample. It is proved by Koll\'{a}r-Miyaoka-Mori \cite{KMM} that there are finitely many deformation classes 
of Fano varieties in each dimension. The complete classification has been known up to dimension three by Iskovskih and Mori-Mukai \cite{I1, I2, MM}, and up to dimension five 
for toric Fano case by Batyrev and Kreuzer-Nill \cite{Ba, KN}.

A {\em monotone} symplectic manifold is a symplectic analogue of a Fano variety in the sense that $\langle c_1(TM), [\Sigma] \rangle > 0$ for every symplectic surface $\Sigma$. 
In a low dimensional case, the monotonicity of $\omega$ implies that $(M,\omega$ is symplectomorphic to some K\"{a}hler manifold.
Especially in dimension four, it was proved by Ohta-Ono \cite{OO2} that any closed monotone symplectic four
manifold is diffeomorphic to a del Pezzo surface (and hence Fano by the uniqueness of a symplectic structure on a rational surface proved by McDuff \cite{McD3}).
On the other hand, it turned out by Fine-Panov \cite{FP} that a monotone symplectic manifold need not be K\"{a}hler in general.
More precisely, they constructed a twelve dimensional closed monotone symplectic manifold having the fundamental group which is not a K\"{a}hler group. 
We notice that the existence of a closed monotone symplectic non-K\"{a}hler manifold is still unknown in dimension 6, 8, or 10. 

In a series of papers, the author deals with the following conjecture.

\begin{conjecture}\cite[Conjecture 1.1]{LinP}\cite[Conjecture 1.4]{FP2}\label{conjecture_main}
	Let $(M,\omega)$ be a six dimensional closed monotone symplectic manifold equipped with an effective Hamiltonian circle action.  Then $(M,\omega)$ is $S^1$-equivariantly
	 symplectomorphic to some K\"{a}hler manifold $(X,\omega_X, J)$ with a certain holomorphic Hamiltonian $S^1$-action.
\end{conjecture}

In the previous work \cite{Cho}, the author proved that Conjecture \ref{conjecture_main} holds under the assumptions that the action is {\em semifree}\footnote{An $S^1$-action is called {\em semifree}
if it is free outside the fixed point set.} 
and at least one of extremal fixed components is an {\em isolated point}. 
Indeed, there are 18 types of such manifolds and their algebro-geometric descriptions (in the sense of Mori-Mukai \cite{MM}) as well as their fixed point 
data are illustrated in \cite[Section 6,7,8]{Cho} and \cite[Table 9.1]{Cho}, respectively. For the complete classification of semifree $S^1$-actions, 
it remains to deal with the case where every extremal fixed component is non-isolated. 

In this paper, we prove the following.
\begin{theorem}\label{theorem_main}
		Let $(M,\omega)$ be a six-dimensional closed monotone symplectic manifold equipped with a semifree Hamiltonian 
		circle action. Suppose that the maximal and the minimal fixed component are both two-dimensional. 
		Then $(M,\omega)$ is $S^1$-equivariantly symplectomorphic to some K\"{a}hler manifold with a certain holomorphic Hamiltonian circle action. 
		In fact, there are 21 types of such manifolds up to $S^1$-equivariant symplectomorphism.
\end{theorem}

\subsection{Summary of the classification}
\label{ssecSummaryOfClassification}

	Figure \ref{figure_summary} (except for {\bf (II-2.2)} and {\bf (III-3)}) 
	illustrates all possible moment map images of a six-dimensional closed monotone symplectic manifold with a Hamiltonian torus action which induces
	a semifree circle action with two dimensional extremal fixed components. 
	
	\begin{figure}[H]
		\scalebox{0.6}{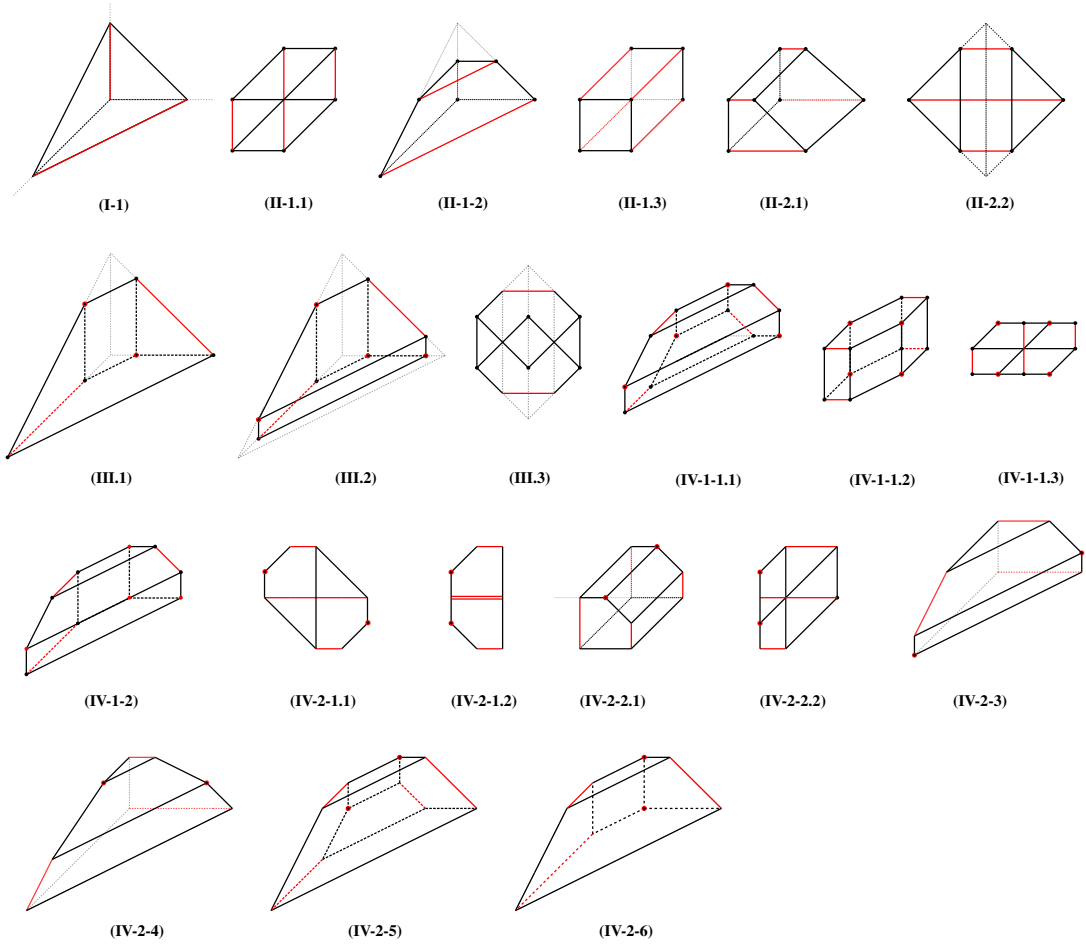}
		\caption{\label{figure_summary} Semifree $S^1$-Fano 3-folds with two dimensional extremal fixed components}
	\end{figure}	

	\noindent
	Note that 
	\begin{itemize}
		\item red edges are images of fixed spheres,
		\item red dots are images of isolated fixed points.	
	\end{itemize}
	Two exceptional cases {\bf (II-2.2)} and {\bf (III-3)} are {\em conceptual images} each of which depicts 
	a blowing up of a complexity-one and complexity zero (toric) variety along some 
	$S^1$-invariant sphere, respectively. 
	Later, one can see that 
	\begin{enumerate}
		\item {\bf (I-1), (II-1.2), (II-1.3), (II-2.1), (III.1), (III.2), (IV-1-1.1), (IV-1-1.2), (IV-1-2), (IV-2-2.1), (IV-2-3), (IV-2-4), (IV-2-5), (IV-2-6)} are toric, 
		\item {\bf (II-1.1), (IV-1-1.3), (IV-2-1.1), (IV-2-1.2), (IV-2-2.2)} are of complexity one, and 
		\item {\bf (II-2.2), (III-3)} are of complexity two
	\end{enumerate}
	where the {\em complexity} of a variety $X$ is by definition half the minimal codimension of a (possibly improper) toric subvariety of $X$.
	More detailed description of the manifolds and the actions thereon can be found in Section \ref{secCaseIMathrmCritMathringHEmptyset}, 
	\ref{secCaseIIMathrmCritMathringH},
	\ref{secCaseIIIMathrmCritMathringH11}, \ref{secCaseIVMathrmCritMathringH11}.
	See also Table \ref{table_list}.

\subsection{Outline of Proof of Theorem \ref{theorem_main}}
\label{ssecOutlineOfProofOfTheoremRefTheoremmain}

The strategy of the proof of Theorem \ref{theorem_main} is essentially the same as one used in \cite{Cho}. 
The main difference from \cite{Cho} is that the normal bundle of an extremal fixed component could be {\em arbitrary}, while the normal bundle of an isolated extremum
is always trivial and isomorphic to $\C^3$. By careful analysis of the geometry of reduced spaces, we may overcome the difficulty and obtain a full list of 
(topological) fixed point data as given in 
Table \ref{table_list}, which leads to {\em symplectical rigidity}\footnote{See Section \ref{secFixedPointData} for the definition.} 
of reduced spaces. This fact enables us to utilize the following theorem. 

\begin{theorem}\cite[Theorem 1.5]{G}\label{theorem_Gonzalez}
		Let $(M,\omega)$ be a six-dimensional closed semifree Hamiltonian $S^1$-manifold. 
		Suppose that every reduced space is symplectically rigid.
		Then $(M,\omega)$ is determined by its fixed point data up to $S^1$-equivariant symplectomorphism.
\end{theorem}

Here, by the fixed point data of $(M,\omega)$ we mean a collection of a
{\em symplectic reduction\footnote{A reduced space at a critical level is not a smooth manifold nor an orbifold in general. However, if $\dim M = 6$ and the action is semifree, then a symplectic reduction at any (critical) level is a smooth manifold with the induced symplectic form. See 
Proposition \ref{proposition_GS}.} at each critical level} together with an {\em information of 
critical submanifolds} (or equivalently fixed components) as embedded symplectic submanifolds of reduced spaces. (See Definition \ref{definition_fixed_point_data} or \cite[Definition 1.2]{G}.)

We divide the proof of Theorem \ref{theorem_main} into three steps : 
\begin{itemize}
	\item ({\bf Step 1}) Classify all {\em topological fixed point data}\footnote{A {\em topological fixed point data}, or TFD for short,
	is a topological analogue of a fixed point data in the sense that it records ``homology classes'', not embeddings themselves, of fixed components in reduced spaces.}. 
	In this process, we obtain a complete list of all topological fixed point data as described in Table \ref{table_list}. Then it follows that every reduced space is diffeomorphic to one of the following 
	manifolds : $\p^1 \times \p^1$ or $X_k$ : $k$-times blow-up of $\p^2$ for $1 \leq k \leq 4$ where those spaces are known to be symplectically rigid (see Section \ref{secMainTheorem}).
	\item ({\bf Step 2}) Show that each topological fixed point data determines a unique fixed point data. 
	\item ({\bf Step 3}) For each topological fixed point data given in Table \ref{table_list}, there exists a corresponding smooth Fano variety
	with a holomorphic semifree Hamiltonian $S^1$-action.
\end{itemize}
Then the proof of Theorem \ref{theorem_main} immediately follows by Gonzalez's theorem \ref{theorem_Gonzalez}.

This paper is organized as follows. In Section \ref{secBackground}, we build up our notations and introduce theorems about Hamiltonian $S^1$-actions that will be crucially used in
the rest of the paper.
In Section \ref{secFixedPointData}, we give a rigorous definition of (topological) fixed point data and explain the idea of the Gonzalez's Theorem \cite[Theorem 1.5]{G}. 
Then, from Section \ref{secCaseIMathrmCritMathringHEmptyset} to Section \ref{secCaseIVMathrmCritMathringH11}, we classify all topological fixed point data as well as provide examples of Fano varieties with 
specific holomorphic $\C^*$-actions for each fixed point data in Table \ref{table_list}. Finally in Section \ref{secMainTheorem}, we complete the proof of Theorem \ref{theorem_main}.

\subsection*{Acknowledgements} 
The author would like to thank Dmitri Panov for bringing the paper \cite{Z} to my attention.
The author would also like to thank Jinhyung Park for helpful comments. 
This work is supported by the National Research Foundation of Korea(NRF) grant funded by the Korea government(MSIP; Ministry of Science, ICT \& Future Planning) (NRF-2017R1C1B5018168).

\section{Background}
\label{secBackground}

	In this section, we establish our notations and review some basic properties of a semifree Hamiltonian circle action.
	We also refer to \cite[Section 2,3,4]{Cho} (and the references therein) in which basic materials (including the ABBV localization and the Duistermaat-Heckman theorem)
	are provided in more detail. 
	
	Let $S^1 \subset \C^*$ be the unit circle group with the Lie algebra $\frak{t}$ and its dual Lie algebra $\frak{t}^*$. 
	An $S^1$-action on a symplectic manifold $(M,\omega)$ is called {\em Hamiltonian} if there exists a smooth function $H : M \rightarrow [a,b] \subset \R$ (called a {\em moment map}) such that 
	\[
		\omega(\underline{X}, \cdot) = dH(\cdot)
	\]
	for every $X \in \frak{t}$ where $\underline{X}$ denotes the fundamental vector field on $M$ generated by $X$. Note that $H$ is a perfect Morse-Bott function and has 
	many good properties (e.g., local normal form), see \cite[Chapter 4]{Au} or \cite[Section 2]{Cho}. 
	
	\begin{notation}
		We use the following notations.
	\begin{itemize}
		\item $M^{S^1}$ : fixed point set (which coincides with the critical point set of $H$).
		\item $\mathrm{Crit}~H$ : set of critical values of $H$. 
		\item $\mathrm{Crit}~ \mathring{H}$ : set of non-extremal critical values of $H$.
		\item $Z_{\min} := H^{-1}(a)$, $Z_{\max} := H^{-1}(b)$ : minimal and maximal fixed component.
		\item $M_t := H^{-1}(t) / S^1$ : reduced space at level $t \in [a,b]$
		\item $\omega_t$ : reduced symplectic form on $M_t$.
		\item $P_c^{\pm}$  : principal $S^1$-bundle $\pi_{c \pm \epsilon} : H^{-1}(c \pm \epsilon) \rightarrow M_{c \pm \epsilon}$ where $\epsilon > 0$ is sufficiently small.
		\item $e(P_c^{\pm}) \in H^2(M_{c \pm \epsilon}; \Q)$ : the Euler class of $P_c^{\pm}$.
		\item $Z_c$ : fixed point set lying on the level set $H^{-1}(c)$. That is, $Z_c = M^{S^1} \cap H^{-1}(c)$.
		\item $\R[\lambda]$ : cohomology ring of $H^*(BS^1;\R)$, where $-\lambda$ is the Euler class of the universal Hopf bundle $ES^1 \rightarrow BS^1.$
	\end{itemize}
	\end{notation}

	From now on, we assume that the $S^1$-action on $(M,\omega)$ is semifree. 

\subsection{Topology of reduced spaces}
\label{ssecTopologyOfReducedSpaces}

	In this section, we briefly review how the topology of a reduced space changes when a level set of $H$ passes through a critical level.
	Note that the `{\em semifree}' condition implies that a reduced space $M_t$ is a smooth manifold for every regular value $t$ of $H$. 
	
	\begin{proposition}\cite{McD2}\cite{GS}\label{proposition_GS}
		Let $(M,\omega)$ be a closed semifree Hamiltonian $S^1$-manifold with a moment map $H : M \rightarrow \R$ and $c \in \R$ be a critical value of $H$. 
		If $Z_c := H^{-1}(c) \cap M^{S^1}$ consists of index-two (co-index two, resp.) fixed components, then $M_c = H^{-1}(c) / S^1$ is smooth and is diffeomorphic to $M_{c-\epsilon}$
		($M_{c+\epsilon}$, resp.). Also, $M_{c+\epsilon}$ is the blow-up (blow-down, resp.) of $M_c$ along $Z_c$.	
	\end{proposition}
	
	If $M$ is of dimension six, then the condition of Proposition \ref{proposition_GS} is automatically satisfied so that a reduced space is smooth for every (possibly critical) value of $H$.
	In fact, Guillemin-Sternberg \cite{GS} states Proposition \ref{proposition_GS} in full generality (i.e., without index assumptions), namely reduced spaces are in {\em birational equivalence}.
	See also the paragraph below \cite[Proposition 4.1]{Cho} for the brief survey on this topic.
	They also describe how the reduced symplectic form $\omega_{c+\epsilon}$ can be obtained from $\omega_{c-\epsilon}$. 
	Recall that the Duistermaat-Heckman's theorem \cite{DH} says that 
	\begin{equation}\label{equation_DH}
		[\omega_r] - [\omega_s] = (s-r)e, \quad r,s \in I
	\end{equation}
	where $I$ is an interval consisting of regular values of $H$ and $e \in H^2(M_r; \Z)$ denotes the Euler class of the principal $S^1$-bundle 
	$\pi_r : H^{-1}(r) \rightarrow M_r$. 
	
	\begin{lemma}\cite[Theorem 13.2]{GS}\label{lemma_Euler_class}
		Suppose that $Z_c = M^{S^1} \cap H^{-1}(c)$ consists of fixed components $Z_1, \cdots, Z_k$ each of which is of index two.   
		Let $e^{\pm}$ be the Euler classes of principal $S^1$-bundles $\pi_{c \pm \epsilon} : H^{-1}(c \pm \epsilon) \rightarrow M_{c \pm \epsilon}$.
		Then 
		\[
			e^+ = \phi^*(e^-) + E \in H^2(M_{c+\epsilon}; \Z)
		\] where $\phi : M_{c+\epsilon} \rightarrow M_{c-\epsilon}$ is the blow-down map and $E$ denotes the Poincar\'{e} dual of the exceptional divisor of $\phi$.
	\end{lemma}
		
	It is worth mentioning that if $Z_c$ in Lemma \ref{lemma_Euler_class} is of codimension four in $M$, i.e., $Z_c$ is of co-dimension two in $M_{c-\epsilon}$, then the 
	blow-up of $M_{c-\epsilon}$ is itself and the exceptional divisor becomes $Z_c$ so that we obtain the following.
	
	\begin{corollary}\label{corollary_Euler_class_zero_level}
		Under the same assumption with Lemma \ref{lemma_Euler_class}, if $Z_c$ is of co-dimension four in $M$, then the topology of a reduced does change, i.e., 
		$M_{c-\epsilon} \cong M_{c+\epsilon}$. Moreover, we have 
		\[
			e^+ = e^- + \mathrm{PD}(Z_c) \in H^2(M_{c+\epsilon}; \Z).
		\]
	\end{corollary}
	See also \cite[Lemma 5]{McD1} for the case of $\dim M = 6$.
			
\subsection{Equivariant cohomology}
\label{ssecEquivariantCohomology}	

	The {\em equivariant cohomology} of $M$ is defined by
	\[
		H^*_{S^1}(M) := H^*(ES^1 \times_{S^1} M)
	\]
	It admits a natural $H^*(BS^1)$-module structure induced by the projection map $\pi$ :
	\begin{equation}\label{equation_Mbundle}
		\begin{array}{ccc}
			M \times_{S^1} ES^1 & \stackrel{f} \hookleftarrow & M \\[0.3em]
			\pi \downarrow          &                             &   \\[0.3em]
			BS^1                   &                             &
		\end{array}
	\end{equation}
	where $f$ is an inclusion of $M$ as a fiber. Then $H^*(BS^1)$-module structure on $H^*_{S^1}(M)$ is given by the map $\pi^*$ such that 
	\[
		y \cdot \alpha = \pi^*(y)\cup \alpha 
	\] for $y \in H^*(BS^1)$ and $\alpha \in H^*_{S^1}(M)$. 
	One remarkable fact on the equivariant cohomology of a Hamiltonian $S^1$-manifold is that it is {\em equivariantly formal}.

	\begin{theorem}\label{theorem_equivariant_formality}\cite{Ki}
		Let $(M,\omega)$ be a closed symplectic manifold equipped with a Hamiltonian circle action. Then $M$ is equivariatly formal, that is,
		$H^*_{S^1}(M)$ is a free $H^*(BS^1)$-module so that $$H^*_{S^1}(M) \cong H^*(M) \otimes H^*(BS^1).$$
		Equivalently, the map $f^*$ is surjective with kernel $x \cdot H^*_{S^1}(M)$ where $\cdot$ means the scalar multiplication of $H^*(BS^1)$-module structure on $H^*_{S^1}(M)$.
	\end{theorem}
	
\subsection{Localization theorem}
\label{ssecLocalizationTheorem} 

	Thanks to the equivariant formality, for any homogeneous element $\alpha \in H^k_{S^1}(M)$, we may express $\alpha$ as
	\begin{equation}\label{equation_expression}
		 \alpha = \alpha_k \otimes 1 + \alpha_{k-2} \otimes \lambda + \alpha_{k-4} \otimes \lambda^2 + \cdots 
	\end{equation}
	where $\alpha_i \in H^i(M)$ for each $i = k, k-2, \cdots$. We then obtain $f^*(\alpha) = \alpha_k$ where $f$ is given in \eqref{equation_Mbundle}.

	\begin{definition}
		An \textit{integration along the fiber $M$} is an $H^*(BS^1)$-module homomorphism $\int_M : H^*_{S^1}(M) \rightarrow H^*(BS^1)$ defined by
		\[
			\int_M \alpha = \langle \alpha_k, [M] \rangle \cdot 1 + \langle \alpha_{k-2}, [M] \rangle \cdot x + \cdots 
		\]
		for every $ \alpha = \alpha_k \otimes 1 + \alpha_{k-2} \otimes \lambda + \alpha_{k-4} \otimes \lambda^2 + \cdots \in H^k_{S^1}(M).$ 
		Here, $[M] \in H_{2n}(M; \Z)$ denotes the fundamental homology class of $M$.
	\end{definition}

	Now, let $M^{S^1}$ be the fixed point set of the $S^1$-action on $M$ and let $F \subset M^{S^1}$ be a fixed component. Then the inclusion $i_F : F \hookrightarrow M$ 
	induces a ring homomorphism $$i_F^* : H^*_{S^1}(M) \rightarrow H^*_{S^1}(F) \cong H^*(F) \otimes H^*(BS^1).$$
	For any $\alpha \in H^*_{S^1}(M)$, we call the image $i_F^*(\alpha)$ \textit{the restriction of $\alpha$ to $F$} and denote by 
	\[
		\alpha|_F := i_F^*(\alpha).
	\] 
	Then we may compute $\int_M \alpha$ concretely by using the following theorem due to Atiyah-Bott \cite{AB} and Berline-Vergne \cite{BV}.

	\begin{theorem}[ABBV localization]\label{theorem_localization}
		For any $ \alpha \in H^*_{S^1}(M)$, we have
		\[
			\int_M \alpha = \sum_{F \subset M^{S^1}} \int_F \frac{\alpha|_F}{e^{S^1}(F)}
		\]
		where $e^{S^1}(F)$ is the equivariant Euler class of the normal bundle $\nu_F$ of $F$ in $M$. That is, $e^{S^1}(F)$ is the Euler class of the bundle 
		\[
			\nu_F \times_{S^1} ES^1 \rightarrow F \times BS^1.
		\] induced from the projection $\nu_F \times ES^1 \rightarrow F \times ES^1$.
	\end{theorem}

\subsection{Monotone symplectic manifolds}
\label{ssecMonotoneSymplecticManifolds}

	Now, we assume that $\omega$ is monotone and normalized, i.e., $c_1(TM) = [\omega]$. 
	\begin{definition}\label{definition_balanced}
		We call a moment map $H : M \rightarrow \R$ {\em balanced} if it satisfies
			\[
				H(Z) = -\Sigma(Z), \quad \quad \Sigma(Z) = ~\text{sum of weights of the $S^1$-action at $Z$}
			\]
		for every fixed component $Z \subset M^{S^1}$. 
	\end{definition}
	Note that there exists a unique balanced moment map. See \cite[Proposition 4.4]{Cho}. The following lemma is immediate from Definition \ref{definition_balanced}.
	
	\begin{lemma}\label{lemma_possible_critical_values}\cite[Lemma 5.9]{Cho} 
	Let $(M,\omega)$ be a six-dimensional closed monotone semifree $S^1$-manifold with the balanced moment map $H$.
	Then all possible critical values of $H$ are $\pm 3, \pm 2, \pm 1$, and $0$. Moreover, any connected component $Z$ of $M^{S^1}$ satisfies one of the followings : 
	\begin{table}[H]
		\begin{tabular}{|c|c|c|c|}
		\hline
		    $H(Z)$ & $\dim Z$ & $\mathrm{ind}(Z)$ & $\mathrm{Remark}$ \\ \hline 
		    $3$ &  $0$ & $6$ & $Z = Z_{\max} = \mathrm{point}$ \\ \hline
		    $2$ &  $2$ & $4$ & $Z = Z_{\max} \cong S^2$ \\ \hline
		    $1$ &  $4$ & $2$ & $Z = Z_{\max}$ \\ \hline
		    $1$ &  $0$ & $4$ & $Z = \mathrm{pt}$ \\ \hline
		    $0$ &  $2$ & $2$ & \\ \hline
		    $-1$ &  $0$ & $2$ & $Z = \mathrm{pt}$ \\ \hline
		    $-1$ &  $4$ & $0$ & $Z = Z_{\min}$ \\ \hline
		    $-2$ &  $2$ & $0$ & $Z = Z_{\min} \cong S^2$ \\ \hline
		    $-3$ &  $0$ & $0$ & $Z = Z_{\min} = \mathrm{point}$ \\ \hline
		\end{tabular}
		\vs{0.2cm}
		\caption{\label{table_fixed} List of possible fixed components}
		\end{table}

	\end{lemma}

	Another important fact on the balanced moment map is that the monotonicity property of the reduced symplectic form $\omega_0$ at level zero is inherited from $\omega$.
	
	\begin{proposition}\label{proposition_monotonicity_preserved_under_reduction}\cite[Proposition 4.8, Remark 4.9]{Cho}
		Let $(M,\omega)$ be a semifree Hamiltonian $S^1$-manifold with $c_1(TM) = [\omega]$ and $H$ be the balanced moment map.
		If the symplectic reduction is defined at level zero, then $(M_0, \omega_0)$ is a monotone symplectic manifold with $[\omega_0] = c_1(TM_0)$
	\end{proposition}
		
	By Proposition \ref{proposition_monotonicity_preserved_under_reduction} and Ohta-Ono's classification \cite{OO2} of closed monotone symplectic four manifolds, 
	we obtain the following.
	
	\begin{corollary}\label{corollary_monotone_reduced_space}
		Let $(M,\omega)$ be a six-dimensional closed monotone semifree $S^1$-manifold with the balanced moment map $H$. Then $M_0$ is diffeomorphic to a del Pezzo surface, i.e., 
		\[
			M_0 \cong \p^2, \p^1 \times \p^1, ~\text{or}~X_k \quad (k \leq 8)
		\]
		where $X_k$ denotes the $k$-points blow-up of $\p^2$.
	\end{corollary}
		
\section{Fixed point data}
\label{secFixedPointData}

	In \cite{Li2, Li3}, Li introduced the notion of a {\em fixed point data} (or FD shortly) for some particular semifree Hamiltonian $S^1$-manifold and Gonzalez \cite{G} defined it in more general context.
	Also, the author \cite{Cho} defined a {\em topological fixed point data} (or TFD for short). 
	In this section, we briefly overview the notions {\em FD} and {TFD} of a closed semifree Hamiltonian $S^1$-manifold 
	and explain how the fixed poitn data determines a manifold up to $S^1$-equivariant symplectomorphism.
	We also refer to \cite[Section 5]{Cho} for more detail. 
	
\subsection{Slices and Gluing}
\label{ssecSlicesAndGluing}
	Any closed Hamiltonian $S^1$-manifold can be decomposed into {\em slices} and, conversely, a family of slices with certain compatible conditions 
	determines a closed Hamiltonian $S^1$-manifold . More precisely, let $(M,\omega)$ be a Hamiltonian $S^1$-manifold with a moment map $H : M \rightarrow I \subset \R$. 
	Assume that the critical values of $H$ are given by 
	\[
		\min H = c_1 < \cdots < c_k = \max H.
	\]	
	Then, $M$ can be decomposed into a union of Hamiltonian $S^1$-manifolds $\{ (N_j, \omega_j) \}_{1 \leq j \leq 2k-1}$ with boundaries :
	\[
		N^{2j-1} = H^{-1}(\underbrace{[c_j - \epsilon, c_j + \epsilon]}_{ =: I_{2j-1}}), \quad N^{2j} = H^{-1}(\underbrace{[c_j + \epsilon, c_{j+1} - \epsilon]}_{=: I_{2j}})
	\]
	where $\epsilon > 0$ is chosen to be sufficiently small so that $I_{2j-1}$ contains exactly one critical value $c_j$ of $H$ for each $j$.
	 We call those $N^{2j-1}$'s and $N^{2j}$'s {\em critical and regular slices}, respectively.
	 
	 \begin{definition}\label{definition_regular_slice}\cite{G} \cite[Definition 5.1, 5.2]{Cho} 
	 	\begin{enumerate}
			\item A {\em regular slice} $(N,\sigma,K, I)$ is a free Hamiltonian $S^1$-manifold $(N, \sigma)$ with boundary and 
				$K : N \rightarrow I$ is a surjective proper moment map where $I = [a,b]$ is a closed interval. 

			\item A {\em critical slice} $(N, \sigma, K, I)$ is a semifree Hamiltonian $S^1$-manifold $(N, \sigma)$ with boundary together with a surjective proper moment map 
				$K : N \rightarrow I = [a,b]$ such that there exists exactly one critical value $c \in [a,b]$ satisfying one the followings : 
				\begin{itemize}
					\item (interior slice) $c \in (a,b)$, 
					\item (maximal slice) $c = b$ and $K^{-1}(c)$ is a critical submanifold, 
					\item (minimal slice) $c = a$ and $K^{-1}(c)$ is a critical submanifold. 
				\end{itemize}
			\item An interior critical slice is called {\em simple} if every fixed component in $K^{-1}(c)$ has the same Morse-Bott index. 
		\end{enumerate}
	\end{definition}
	
	Two slices $(N_1,\sigma_1,K_1, I_1)$ and $(N_2,\sigma_2,K_2, I_2)$ are said to be 
	{\em isomorphic} if there exists an $S^1$-equivariant symplectomorphism $\phi : (N_1, \sigma_1) \rightarrow (N_2, \sigma_2)$ satisfying 
	\[
		\xymatrix{N_1 \ar[r]^{\phi} \ar[d]_{K_1} & N_2 \ar[d]^{K_2} \\ I_1 \ar[r]^{ + k} & I_2}
	\]
	where $+k$ denotes the translation map as the addition of some constant $k \in \R$. The following lemma tells us when two slices can be glued along their boundaries.

	\begin{lemma}\cite[Lemma 13]{Li3}\cite[Lemma 1.2]{McD2}\label{lemma_gluing}
		Two slices $(N_1, \sigma_1, K_1, [a,b])$ and $(N_2, \sigma_2, K_2, [b,c])$ can be glued along $K_i^{-1}(b)$ if there exists a diffeomorphism 
		\[
			\phi : (N_1)_b \rightarrow (N_2)_b, \quad \quad (N_i)_b := K_i^{-1}(b) / S^1
		\]
		such that 
		\begin{itemize}
			\item $\phi^* (\sigma_2)_b = (\sigma_1)_b$, and 
			\item $\phi^* (e_2)_b = (e_1)_b$ 
		\end{itemize}
		where $(\sigma_i)_b$ and $(e_i)_b$ denote the reduced symplectic form on $(N_i)_b$ and the Euler class of the principal $S^1$-bundle 
		$K_i^{-1}(b) \rightarrow (N_i)_b$, respectively. 
	\end{lemma}
	
	Now, suppose that $\frak{S} = \{ (N_i, \sigma_i, K_i, [a_i,b_i]) \}$ be a finite family of slices with gluing data 
	\[
		\Phi := \{\phi_i : (N_i)_{b_i} \rightarrow (N_{i+1})_{a_{i+1}}\}
	\]
	satisfying the conditions in Lemma \ref{lemma_gluing}. Then $(\frak{S}, \Phi)$ determines a closed Hamiltonian $S^1$-manifold denoted by $M(\frak{S}, \Phi)$. 
	Note that $M(\frak{S}, \Phi)$ may not be $S^1$-equivariantly symplectomorphic (nor even diffeomorphic) to $M(\frak{S}, \Phi')$ for a different choice of gluing data $\Phi'$.
	
\subsection{Fixed point data}
\label{ssecFixedPointData}	
	Now, consider a six-dimensional closed monotone symplectic manifold $(M,\omega)$
	equipped with an effective semifree Hamiltonian $S^1$-action. 
	We further assume that $c_1(TM) = [\omega]$ so that there exists a (unique) balanced moment map $H : M \rightarrow \R$ for the action defined in Definition \ref{definition_balanced}.
	
	\begin{definition}\cite[Definition 1.2]{G}\label{definition_fixed_point_data} 
		A {\em fixed point data} (or {\em FD} shortly) of $(M,\omega, H)$, denoted by $\frak{F}(M, \omega, H)$, is a collection 
		\[
			 \frak{F} (M, \omega, H) := \left\{(M_{c}, \omega_c, Z_c^1, Z_c^2, \cdots,  Z_c^{k_c}, e(P_{c}^{\pm})) ~|~c \in \mathrm{Crit} ~H \right\}
		\]
		which consists of the information below.
		\begin{itemize}
			\item $(M_c, \omega_c)$\footnote{$M_c$ is smooth manifold under  the assumption that the action is semifree and the dimension of $M$ is six.
				See Proposition \ref{proposition_GS}.} is the symplectic reduction at level $c$.
			\item $k_c$ is the number of fixed components on the level $c$. 
			\item Each $Z_c^i$ is a connected fixed component and hence a symplectic submanifold of $(M_c, \omega_c)$ via the embedding
				\[
					Z_c^i \hookrightarrow H^{-1}(c) \rightarrow H^{-1}(c) / S^1 = M_c.
				\]
				(This information contains a normal bundle of $Z_c^i$ in $M_c$.)
			\item The Euler class $e(P_c^{\pm})$ of principal $S^1$-bundles $H^{-1}(c \pm \epsilon) \rightarrow M_{c \pm \epsilon}$.
		\end{itemize}		
	\end{definition}
	
	\begin{definition}\cite[Definition 2.13]{McD2}\cite[Definition 1.4]{G}\label{definition_rigid} A manifold $B$ is said to be {\em symplectically rigid} if 
	\begin{itemize}
		\item (uniqueness) any two cohomologous symplectic forms are diffeomorphic, 
		\item (deformation implies isotopy) every path $\omega_t$ ($t \in [0,1]$) of symplectic forms such that $[\omega_0] = [\omega_1]$ can be homotoped through families of symplectic forms 
		with the fixed endpoints $\omega_0$ and $\omega_1$ to an isotopy, that is, a path $\omega_t'$ such that $[\omega_t']$ is constant in $H^2(B)$. 
		\item For every symplectic form $\omega$ on $B$, the group $\text{Symp}(B,\omega)$ of symplectomorphisms that act trivially on $H_*(B;\Z)$ is path-connected.
	\end{itemize}
	\end{definition}
	
	As we have seen in Section \ref{ssecSlicesAndGluing}, the $S^1$-equivariant symplectomorphism class of a Hamiltonian $S^1$-manifold $M(\frak{S}, \Phi)$ constructed from a given 
	family $\frak{S}$ of slices depends on the choice of a gluing data $\Phi$. The following theorem due to Gonzalez states that $M(\frak{S}, \Phi)$ only depends on the fixed point data
	of the action on $M(\frak{S}, \Phi)$ if every reduced space is symplectically rigid.
	
	\begin{theorem}\cite[Theorem 1.5]{G}\label{theorem_Gonzalez_5}
		Let $(M,\omega)$ be a six-dimensional closed semifree Hamiltonian $S^1$-manifold such that every critical level is simple\footnote{
		A critical level is called {\em simple} if every fixed component in the level set has a common Morse-Bott index.}.
		Suppose further that every reduced space is symplectically rigid.
		Then $(M,\omega)$ is determined by its fixed point data up to $S^1$-equivariant symplectomorphism.
	\end{theorem}
	
	\begin{remark}\label{remark_Gonzalez_5}
		 Note that Theorem \ref{theorem_Gonzalez_5} is a six-dimensional version of the original statement of Theorem \cite[Theorem 1.5]{G}
		 so that we may drop ``(co)-index two'' condition in his original statement because every non-extremal fixed component has index two or co-index two in a six-dimensional case.
		 In addition, if $\omega$ is monotone, then the condition ``simpleness'' is automatically satisfied by Lemma \ref{lemma_possible_critical_values}.
	 \end{remark}

	For proving Theorem \ref{theorem_main}, 
	we need to 
	\begin{itemize}
		\item classify all possible fixed point data $\frak{F}$, 
		\item show the existence of the corresponding Hamiltonian $S^1$-manifold having the fixed point data $\frak{F}$, 
		\item show that every reduced space is symplectically rigid.
	\end{itemize} 
	However, the classification of fixed point data is extremely difficult as it 
	involves the classification problem of all symplectic embeddings of each fixed component of $Z_c$ into a reduced space $(M_c, \omega_c)$. 
	Thus, instead of a fixed point data, we introduce the notion ``{\em topological fixed point data}'', which is a topological analogue of a fixed point data, as follows. 
				
	\begin{definition}\label{definition_topological_fixed_point_data}\cite[Definition 5.7]{Cho}
		Let $(M,\omega)$ be a six-dimensional closed semifree Hamiltonian $S^1$-manifold equipped with a moment map $H : M \rightarrow I$ such that all critical level sets are simple.
		A {\em topological fixed point data} (or {\em TFD} for short) of $(M,\omega, H)$, denoted by $\frak{F}_{\text{top}}(M, \omega, H)$, is defined as a collection 
		\[
			 \frak{F}_{\text{top}}(M, \omega, H) := \left\{(M_{c}, [\omega_c], \mathrm{PD}(Z_c^1), \mathrm{PD}(Z_c^2), \cdots, \mathrm{PD}(Z_c^{k_c}), e(P_c^{\pm}) ) ~|~c \in
			  \mathrm{Crit} ~H \right\}
		\]
		where 
		\begin{itemize}
			\item $(M_c, \omega_c)$ is the reduced symplectic manifold at level $c$, 
			\item $k_c$ is the number of fixed components at level $c$, 
			\item each $Z_c^i$ is a connected fixed component lying on the level $c$ and $\mathrm{PD}(Z_c^i) \in H^*(M_c)$ 
			denotes the Poincar\'{e} dual class of the image of the embedding
				\[
					Z_c^i \hookrightarrow H^{-1}(c) \rightarrow H^{-1}(c) / S^1 = M_c.
				\]
			\item the Euler class $e(P_c^{\pm})$ of principal $S^1$-bundles $H^{-1}(c \pm \epsilon) \rightarrow M_{e \pm \epsilon}$.
		\end{itemize}		
	\end{definition}
	
	The classification of TFD is relatively much more easier than the classification of FD as we will see later.
	Indeed, we will classify all possible 
	TFD for a semifree Hamiltonian circle action on a six-dimensional monotone symplectic manifold. (See Table \ref{table_list} for the full list of TFD.)
	
	On the other hand, there is one more critical issue. In general, it is not obvious whether a TFD determines a FD uniquely. 
	Namely, for two candidates $Z_c^1$ and $Z_c^2$ of a fixed component
	in $(M_c,\omega_c)$ representing a same homology class, it is not guaranteed the existence of a symplectomorphism (nor a diffeomorphism)
	\[
		\psi : (M_c, \omega_c) \rightarrow (M_c, \omega_c), \quad \quad \psi(Z_c^1) = \psi(Z_c^2). 
	\]
	In Section \ref{secMainTheorem}, we will show that each TFD determines FD uniquely in our situation, and therefore TFD becomes a complete invariant for a semifree Hamiltonian
	circle action on a six-dimensional closed monotone symplectic manifold.

\section{Reduced spaces near the extremum}
\label{secReducedSpacesNearTheExtremum}

	This section is devoted to collect
	some information of a reduced space near an extremum such as a cohomology ring structure and the symplectic area. These materials would be 
	used in the rest of the paper. 
	
	Let $(M,\omega)$ be a six-dimensional closed monotone semifree Hamiltonian $S^1$-manifold with the balanced moment map $H$. 
	We assume that all extremal fixed components are two-dimensional, i.e., $H(Z_{\min}) = -2$ and $H(Z_{\max}) = 2$. 
	Thanks to Li's theorem \cite[Theorem 0.1]{Li1}, 
	we have 
	\[
		\pi_1(Z_{\max}) \cong \pi_1(Z_{\min}) \cong \pi_1(M_0), 
	\]
	which implies that $Z_{\max} \cong Z_{\min} \cong S^2$ as in Lemma \ref{lemma_possible_critical_values} (since $M_0$ is simply connected.)

	Observe that the only possible non-extremal critical values are $\{\pm1, 0\}$
	and each non-extremal fixed component $Z$ is either 
	\[
		\begin{cases}
			\text{$Z$ = pt} \hspace{1cm} \text{if $H(Z) = \pm 1$, \quad or} \\
			\text{$\dim Z = 2$} \quad \text{if $H(Z) = 0$.}
		\end{cases}
	\]
	by Lemma  \ref{lemma_possible_critical_values}. 
	Moreover, since the moment map $H$ is a perfect Morse-Bott function, we may easily deduce that
	\[
		|Z_1| = |Z_{-1}|
	\]
	by the Poincar\'{e} duality. 
			
	We follow Li's notations in \cite{Li2} and \cite{Li3}. 
	For a sufficiently small $\epsilon > 0$, the level set $H^{-1}(-2+\epsilon)$ becomes an $S^3$-bundle over $Z_{\min}$ with the induced fiberwise free $S^1$-action.
	(This can be shown using the {\em equivariant Darboux theorem} and the explicit formula of the moment map, see \cite[Theorem 2.1, Section 4.1]{Cho}.) 
	Thus the reduced space $M_{-2 + \epsilon}$ near $Z_{\min}$ is an $S^2$-bundle over $S^2$ and hence diffeomorphic to 
	either $S^2 \times S^2$ or a Hirzebruch surface which we denote by $E_{S^2}$. 
	
	When $M_{-2 + \epsilon} \cong S^2 \times S^2$, regarded as a trivial $S^2$-bundle over $Z_{\min} \cong S^2$, let $x$ and $y$ in $H^2(M_{-2 + \epsilon} ;\Z)$ be the dual classes of the fiber $S^2$ and the base $Z_{\min}$, respectively. Then
	\[
		\langle xy, [M_{-2 + \epsilon}] \rangle = 1, \quad \langle x^2, [M_{-2 + \epsilon}] \rangle = \langle y^2, [M_{-2 + \epsilon}] \rangle = 0.
	\]
	Similarly, when $M_{-2 + \epsilon} \cong E_{S^2}$ regarded as a non-trivial $S^2$-bundle over $Z_{\min}$, let $x$ and $y$ be the dual of the fiber $S^2$ and the base respectively
	so  that
	\[
		\langle xy, [M_{-2 + \epsilon}] \rangle = 1, \quad \langle x^2, [M_{-2 + \epsilon}] \rangle = 0, \quad \langle y^2, [M_{-2 + \epsilon}] \rangle = -1.
	\]
	In this notation, we have $c_1(T(S^2 \times S^2)) = 2x + 2y$ and $c_1(TE_{S^2}) = 3x+2y$, respectively. 
	
	The following lemma describes the relation between the Euler class of a level set (as a principal $S^1$-bundle) near the extremal fixed components 
	$Z_{\min}$ and $Z_{\max}$ of the action and the first Chern numbers of the normal bundles 
	of them.

\begin{lemma}\cite[Lemma 6, 7]{Li2}\label{lemma_Euler_extremum}
	Let $b_{\min}$ (respectively $b_{\max}$) be the first Chern number of the normal bundle of $Z_{\min}$ (respectively $Z_{\max}$) in $M$. 
	Also, we let $x$ and $y$ be the dual classes of the fiber and the base of the bundle $M_{-2+\epsilon} \rightarrow Z_{\min}$ (respectively $M_{2 - \epsilon} \rightarrow Z_{\max}$).
	Then $M_{-2 + \epsilon}$ (respectively $M_{2 - \epsilon}$) is a trivial $S^2$-bundle if and only if $b_{\min} = 2k$ (respectively $b_{\max} = 2k$), and it is diffeomorphic to $E_{S^2}$
	if and only if $b_{\min} = 2k+1$ (respectively $b_{\max} = 2k+1$) for some $k \in \Z$. In either case, we have 
	\[
		e(P_{-2}^+) = kx - y \quad \quad \left(\text{respectively} \hs{0.1cm} e(P_2^-) = -kx + y \right)
	\]
	where $e(P_t^{\pm})$ denote the Euler class of the principal $S^1$-bundle $\pi_{t \pm \epsilon} :  P_t^{\pm} = H^{-1}(t \pm \epsilon) \rightarrow 
	M_{t \pm \epsilon}$.
	In particular, we have
	\[
		\langle e(P_{-2}^+)^2, [M_{-2 + \epsilon}] \rangle = -b_{\min} \quad \quad \left( \text{respectively} \hs{0.2cm} \langle e(P_2^-)^2, [M_{2 - \epsilon}] \rangle = -b_{\max} \right).
	\]	
\end{lemma}

The monotonicity\footnote{A symplectic form $\omega$ is {\em monotone} if $c_1(TM) = \lambda [\omega] \in H^2(M; \R)$ for some $\lambda \in \R_{>0}$.} of $\omega$ implies the following.

\begin{corollary}\label{corollary_volume_extremum}
	Let $(M,\omega)$ be a six-dimensional closed semifree Hamiltonian $S^1$-manifold. Suppose that $c_1(TM) = [\omega]$. 
	If the minimal fixed component $Z_{\min}$ (respectively $Z_{\max}$) is diffeomorphic to $S^2$, then 
	\[
		b_{\min} \geq -1 \quad \text{(respectively ~$b_{max} \geq -1$)}.
	\]
\end{corollary}

\begin{proof}
	Note that the symplectic volume of $Z_{\min}$ (respectively $Z_{\max}$) is given by 
	\[
		\int_{Z_{\min}} \omega = 2 + b_{\min} \quad \quad \left( \text{respectively} \hs{0.1cm} \int_{Z_{\max}} \omega = 2 + b_{\max} \right)
	\]
	which follows from the fact that the restriction of the tangent bundle $\left. TM \right|_{Z_\bullet}$ splits into the sum of the tangent bundle and the normal bundle 
	of $Z_{\bullet}$ where $\bullet = \min$ or $\max$. Then the proof is straightforward by the positivity of symplectic area and the fact that $\omega$ is integral.
\end{proof}

\begin{remark}\label{remark_bminbmax}
	If we take the new Hamiltonian $S^1$-action ``$*$'' on $M$ by 
	\[	
		t * p := t^{-1} \cdot p, \quad \quad p \in M,
	\]
	then the balanced moment map becomes $-H$ so that the maximal (resp. minimal) fixed component becomes the minimal (resp. maximal)
	one. Therefore, we only need to classify TFD under the assumption that 
	\begin{equation}\label{equation_assumption}
		b_{\min} \leq b_{\max}.
	\end{equation}
	Then any case with ``$b_{\min} > b_{\max}$'' can be recovered from one in our classification by taking a ``reversed'' $S^1$-action.
\end{remark}

 The following lemma due to McDuff will be useful in the rest sections.

	\begin{lemma}\label{lemma_list_exceptional}\cite[Section 2]{McD2}
		Let $X_k$ be the $k$-times simultaneous symplectic blow-up of $\p^2$ with the exceptional divisors $C_1, \cdots, C_k$.
		We denote by $E_i := \mathrm{PD}(C_i)  \in H^2(M_0; \Z)$ the dual classes, called {\em exceptional classes}.
		Then all possible exceptional classes are listed as follows (modulo permutations of indices) : 
			\[
				\begin{array}{l}
					E_1, u - E_{12},  \quad 2u - E_{12345}, \quad 3u - 2E_1 - E_{234567}, \quad 4u - 2E_{123} - E_{45678}  \\ \vs{0.1cm}
					5u - 2E_{123456}  - E_{78},  \quad 6u - 3E_1 - 2E_{2345678}  \\
				\end{array}
			\]
		Here, $u$ is the positive generator of $H^2(\p^2; \Z)$ and $E_{j \cdots n} := \sum_{i=j}^n E_i$. Furthermore, elements involving $E_i$ do not appear in $X_k$ with $k < i$. 
	\end{lemma} 

We divide the classification process into four cases: $\mathrm{Crit} ~\mathring{H} = \emptyset, \{0\}, \{-1, 1\},$ and $\{-1,0,1\}$. \\

\section{Case I : $\mathrm{Crit} ~\mathring{H} = \emptyset$}
\label{secCaseIMathrmCritMathringHEmptyset}

	In this section, we classify all TFD in the case where $\mathrm{Crit} \mathring{H} = \emptyset$. 
	Also, for each TFD, we give the corresponding example of a Fano variety with an explicit holomorphic $S^1$-action on it. 
	Note that $M_{-2 + \epsilon} \cong M_0 \cong M_{2 - \epsilon}$.
	
\begin{theorem}\label{theorem_I_1}
	Let $(M,\omega)$ be a six-dimensional closed monotone semifree Hamiltonian $S^1$-manifold with $c_1(TM) = [\omega]$.
	Suppose that $\mathrm{Crit} H = \{ 2, -2\}$. Then the only possible 
	topological fixed point data is given by 	
		\begin{table}[H]
			\begin{tabular}{|c|c|c|c|c|c|c|c|}
				\hline
				    & $(M_0, [\omega_0])$ & $e(P_{-2}^+)$ &$Z_{-2}$  & $Z_2$ & $b_2$ & $c_1^3$ \\ \hline \hline
				    {\bf (I-1)} & $(S^2 \times S^2, 2x + 2y)$ & $x-y$  &$S^2$ &  $S^2$ & $1$ & $64$\\ \hline    
			\end{tabular}
			\vs{0.5cm}			
			\caption{\label{table_I_1} Topological fixed point data for $\mathrm{Crit} H = \{-2, 2\}$}			
		\end{table}
			\vs{-0.7cm}
\end{theorem}
	
\begin{proof}
	We first assume that $M_0 \cong S^2 \times S^2$ (so that $b_{\min} = 2k$ for some $k \in \Z$ and $e(P_{-2}^+) = kx - y$ by Lemma \ref{lemma_Euler_extremum}.)
	Then, Corollary \ref{corollary_volume_extremum} implies that $b_{\min} = 2k \geq -1$, i.e., $k \geq 0$. Using the monotonicity of the reduced space (Proposition
	\ref{proposition_monotonicity_preserved_under_reduction}) and the Duistermaat-Heckman theorem \eqref{equation_DH}, we obtain
	\[
		[\omega_t] = 2x + 2y - t(kx - y) = (2 - kt)x + (2 + t)y, \quad \quad t \in (-2,2).
	\]
	As $\lim_{t \rightarrow 2} \int_{M_t} [\omega_t]^2 = 0$, we get $k=1$ and so $b_{\min} = 2$. Moreover there is a natural identification 
	$H^{-1}(-2 + \epsilon) \cong H^{-1}(2 - \epsilon)$ by a Morse flow of $H$ so that we obtain $e(P_2^-) = e(P_{-2}^+)$ and 
	\[
		\langle e(P_{-2}^+)^2, [M_{-2 + \epsilon}] \rangle  = \langle e(P_{-2}^+)^2, [M_{2 - \epsilon}] \rangle = -2.
	\]	
	Therefore $b_{\max} = 2$ by Lemma \ref{lemma_Euler_extremum}.
	
	Let $u$ be the positive generator of $H^2(Z_{\min};\Z) = H^2(Z_{\max};\Z)$ so that $u^2 = 0$.
	The first Chern number can be obtained by applying the 
	localization theorem \ref{theorem_localization} : 
	\[
		\begin{array}{ccl}\vs{0.1cm}
			\ds \int_M c_1^{S^1}(TM)^3 & = &  \ds  
							\int_{Z_{\min}} \frac{\left(c_1^{S^1}(TM)|_{Z_{\min}}\right)^3}{e_{Z_{\min}}^{S^1}} + 
							\int_{Z_{\max}} \frac{\left(c_1^{S^1}(TM)|_{Z_{\max}}\right)^3}{e_{Z_{\max}}^{S^1}} \\ \vs{0.1cm}
							& = &  \ds  
							\int_{Z_{\min}} \frac{\left( (2+b_{\min}) u + 2\lambda \right)^3}{b_{\min} u\lambda + \lambda^2} + 
							\int_{Z_{\max}} \frac{\left( (2+b_{\max}) u - 2\lambda \right)^3}{-b_{\max} u\lambda + \lambda^2} \\ \vs{0.1cm}
							& = &  \ds  
							\int_{Z_{\min}} (\lambda - 2u)(48u\lambda^2 + 8\lambda^3) + 
							\int_{Z_{\max}} (\lambda + 2u)(48u\lambda^2 - 8\lambda^3) = 32 + 32 = 64. 
		\end{array}
	\]
	See Table \ref{table_I_1}: {\bf (I-1).}
	
	It remains to consider the case where $M_0 \cong E_{S^2}$.
	In this case, we have $b_{\mathrm{min}} = 2k + 1$ for some $k \in \Z$ by Lemma \ref{lemma_Euler_extremum}. Similar to the previous case, we have 
	\[
		[\omega_t] = (3x + 2y) - t (kx - y) = (3 - kt)x + (2 + t)y, \quad t \in (-2, 2).
	\]
	Again, since $\lim_{t \rightarrow 2} \int_{M_t} [\omega_t]^2 = 8(3-2k) - 16 = 8 - 16k = 0$, we have $k = \frac{1}{2}$ which contradicts that $k \in \Z$. Consequently, 
	$M_0$ cannot be diffeomorphic to $E_{S^2}$. This completes the proof.
\end{proof}	
	
\begin{example}[Fano variety of type {\bf (I-1)}]\cite[17th in Section 12.2]{IP}\label{example_I_1} 
	Let $X = \p^3$ with the symplectic form $4 \omega_{\mathrm{FS}}$ (so that $c_1(TX) = [4\omega_{\mathrm{FS}}]$) where $\omega_{\mathrm{FS}}$ denotes 
	the normalized Fubini-Study form such that $\int_X \omega_{\mathrm{FS}} = 1$. Consider the Hamiltonian $S^1$-action on $(X, 4\omega_{\mathrm{FS}})$
	given by 
	\[
		t \cdot [z_0, z_1, z_2, z_3] = [tz_0, tz_1, z_2, z_3], \quad \quad t \in S^1
	\]
	where the balanced moment map for the action is given by
	\[
		H([z_0, z_1, z_2, z_3]) = \frac{4|z_0|^2 + 4|z_1|^2}{|z_0|^2 + |z_1|^2 + |z_2|^2 + |z_3|^2} - 2.
	\]
	Then the fixed point set (whose image is red lines in Figure \ref{figure_II_1}) is given by $\{ Z_{-2} \cong Z_2 \cong S^2 \}$ and this coincides with the one given in Theorem
	 \ref{theorem_I_1}.
	(See also \cite[Table 1-(4)]{Li2}.)
	
	\begin{figure}[H]
		\scalebox{1}{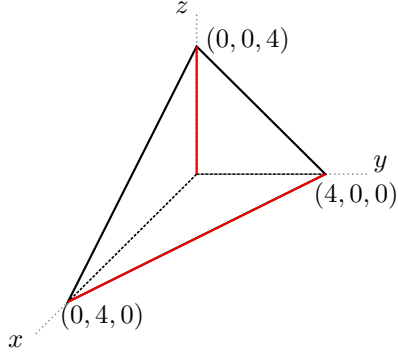}
		\caption{\label{figure_I_1} Toric moment map on $\p^3$}
	\end{figure}
\end{example}

\section{Case II : $\mathrm{Crit} ~\mathring{H} = \{0\}$}
\label{secCaseIIMathrmCritMathringH}

	In this section, we classify all TFD in the case where $\mathrm{Crit} \mathring{H} = \{ 0 \}$. 
	By Proposition \ref{proposition_GS}, we have $M_{-2 + \epsilon} \cong M_0 \cong M_{2 - \epsilon}$ so  that we may divide the proof into two cases:
\begin{itemize}
    \item $M_0 \cong S^2 \times S^2.$
    \item $M_0 \cong E_{S^2}.$
\end{itemize} 
We begin with the case $M_0 \cong S^2 \times S^2.$

\begin{theorem}\label{theorem_II_1}
	Let $(M,\omega)$ be a six-dimensional closed monotone semifree Hamiltonian $S^1$-manifold with $c_1(TM) = [\omega]$. Suppose that $\mathrm{Crit} H = \{ 2, 0, -2\}$ and
	$M_0 \cong S^2 \times S^2$. Then, up to orientation of $M$, the list of all possible topological fixed point data is given by
		\begin{table}[H]
			\begin{tabular}{|c|c|c|c|c|c|c|c|c|}
				\hline
				    & $(M_0, [\omega_0])$ & $e(P_{-2}^+)$ &$Z_{-2}$  & $Z_0$ & $Z_2$ & $b_2(M)$ & $c_1^3(M)$ \\ \hline \hline
				    {\bf (II-1.1)} & $(S^2 \times S^2, 2x + 2y)$ & $-y$  &$S^2$ & $Z_0 \cong S^2, ~\mathrm{PD}(Z_0) = x+y$ & $S^2$ & $2$ &$48$ \\ \hline
				    {\bf (II-1.2)} & $(S^2 \times S^2, 2x + 2y)$ & $-y$  &$S^2$ & $Z_0 \cong S^2, ~\mathrm{PD}(Z_0) = x$ & $S^2$ & $2$ & $56$\\ \hline    
				    {\bf (II-1.3)} & $(S^2 \times S^2, 2x + 2y)$ & $-y$  &$S^2$ &
				     \makecell{ $Z_0 = Z_0^1 ~\dot \cup ~ Z_0^2$ \\
				    $Z_0^1 \cong Z_0^2 \cong S^2$ \\ $\mathrm{PD}(Z_0^1) = \mathrm{PD}(Z_0^2) = y$}   & $S^2$ & $3$ & $48$\\ \hline
			\end{tabular}
			\vs{0.5cm}			
			\caption{\label{table_II_1} Topological fixed point data for $\mathrm{Crit} H = \{-2, -0, 2\}$, $M_0 \cong S^2 \times S^2$}
		\end{table}
			\vs{-0.7cm}
\end{theorem}

\begin{proof}
	Denote by $\mathrm{PD}(Z_0) = ax + by \in H^2(M_0; \Z)$ for some $a,b \in \Z$.
	By Lemma \ref{lemma_Euler_extremum}, we may assume that $b_{\mathrm{min}} = 2k$ for some integer $k \in \Z$ and that $e(P_{-2}^+) = kx - y$.
	By the Duistermaat-Heckman theorem \eqref{equation_DH}, we have
	\[
		[\omega_2] = [\omega_0] - 2(kx - y + \mathrm{PD}(Z_0)) = 2(1-a-k)x + (4-2b)y. 
	\]	
	As $\lim_{t \rightarrow 2} \int_{M_t} [\omega_t]^2 = 0$, we see that
	\begin{enumerate}
		\item $1-a-k=0$ and $4-2b > 0$, or
		\item $b=2$ and $1-a-k > 0$
	\end{enumerate}
	where the above two strict inequalities follow from the fact that $\int_{M_t} [\omega_t]^2 > 0$ for every $0 \leq t < 2$. Moreover, we have 
	\begin{equation}\label{equation_area_vol}
		\langle c_1(TM_0), [Z_0] \rangle = \langle [\omega_0], [Z_0] \rangle = 2a + 2b > 0, \quad \quad 
		\mathrm{Vol}(Z_{-2}) = 2k + 2 > 0 \hs{0.2cm} (\Leftrightarrow ~k \geq 0) 
	\end{equation}	
	by Corollary \ref{corollary_volume_extremum}. \\
	
	\noindent 
	{\bf Case (1).} ~If $a + k = 1$ and $b \leq 1$, then the integer solutions $(a,b,k)$ for \eqref{equation_area_vol} are
	\[
		(1,1,0), (1,0,0), (0,1,1).
	\]
	
	\noindent
	{\bf Case (2).} ~If $a + k \leq 0$ and $b = 2$, then the integer solutions for $(a,b,k)$ are 
	\[
		(0,2,0), (-1,2,0), (-1,2,1).
	\]
	However, we may rule out the last two solutions in {\bf Case (2)} using the adjuction formula 
	\begin{equation}\label{equation_adjunction}
		[Z_0]\cdot [Z_0]  + \sum (2 - 2g_i) = \langle c_1(TM_0), [Z_0] \rangle
	\end{equation}
	where the sum is taken over all fixed components of $Z_0$. If $(a,b,k)$ is $(-1,2,0)$ or $(-1,2,1)$, we have 
	\[
		-4 + \sum (2 - 2g_i) = \langle 2x + 2y, [Z_0] \rangle = 2, \quad \quad \mathrm{PD}(Z_0) = -x + 2y
	\]
	which implies that there are at least three components each of which is homeomorphic to a sphere. Meanwhile, since 
	$\langle c_1(TM_0), [Z_0] \rangle = 2$ is the symplectic area of $Z_0$, there should be at most two components in $Z_0$ and this leads to a contradiction.
	Summing up, we have 
	\begin{equation}\label{equation_2_1_bminbmax}
		\begin{array}{lll}
			(a,b,k) = (1,1,0) ~(b_{\min} = 0, b_{\max} = 0), & & (a,b,k) = (1,0,0) ~(b_{\min} = 0, b_{\max} = 2) \\
			(a,b,k) = (0,1,1) ~(b_{\min} = 2, b_{\max} = 0), & & (a,b,k) = (0,2,0) ~(b_{\min} = 0, b_{\max} = 0) \\
		\end{array}
	\end{equation}
	where $b_{\min} = 2k$ and $b_{\max}$ is computed by Lemma \ref{lemma_Euler_extremum}.
	Since we only need to classify TFD's satisfying $b_{\min} \leq b_{\max}$ by \eqref{equation_assumption}, the case $(a,b,k) = (0,1,1)$ can be ruled out.
	
	Notice that the symplectic area of each component of $Z_0$ is even (since $[\omega_0] = 2x + 2y$). Applying \eqref{equation_adjunction} to each 
	solutions in \eqref{equation_2_1_bminbmax},
	we deduce that
	\begin{equation}\label{equation_II_1}
		\begin{array}{lllll}
			\text{\bf (II-1.1)} : (a,b,k) = (1,1,0) & \Rightarrow & 2 + \sum (2-2g_i) = 4 & \Rightarrow & \text{$Z_0$ has at most two components,}\\
			\text{\bf (II-1.2)} : (a,b,k) = (1,0,0) & \Rightarrow & 0 + \sum (2-2g_i) = 2 & \Rightarrow & Z_0 \cong S^2,\\
			\text{\bf (II-1.3)} : (a,b,k) = (0,2,0) & \Rightarrow & 0 + \sum (2-2g_i) = 4 & \Rightarrow & \text{$Z_0$ has exactly two components.}\\
		\end{array}
	\end{equation}
	For the last case, it is easy to check that each two components are spheres (with area $2$) whose Poincar\'{e} dual classes are both $y$. 
	
	For the first case, if $Z_0$ consists of two components, say $Z_0^1$ and $Z_0^2$, then we can easily see that $Z_0^1 \cong S^2$ and $Z_0^2 \cong T^2$
	with 
	\[
		[Z_0^1] \cdot [Z_0^1] = 0, \quad [Z_0^2] \cdot [Z_0^2] = 2, \quad 
		[Z_0^1] \cdot [Z_0^2] = 0, \quad \mathrm{PD}(Z_0^1) + \mathrm{PD}(Z_0^2) = x+y.  
	\]
	The first and the third equalities imply that $(\mathrm{PD}(Z_0^1), \mathrm{PD}(Z_0^2)) = (ax, bx)$ 
	or $(ay, by)$ for some $a,b \in \Z$, but in either case, the second (as well as fourth) equality does not hold.
	Therefore, $Z_0$ is connected and homeomorphic to $S^2$.
	
	To calculate the Chern number for each fixed point data, we apply the localization theorem \ref{theorem_localization} : 
	\[
		\begin{array}{ccl}\vs{0.3cm}
			\ds \int_M c_1^{S^1}(TM)^3 & = &  \ds  
							\int_{Z_{\min}} \frac{\left(c_1^{S^1}(TM)|_{Z_{\min}}\right)^3}{e_{Z_{\min}}^{S^1}} + 
							\int_{Z_{\max}} \frac{\left(c_1^{S^1}(TM)|_{Z_{\max}}\right)^3}{e_{Z_{\max}}^{S^1}} + 
							\int_{Z_0} \frac{\overbrace{\left(c_1^{S^1}(TM)|_{Z_0}\right)^3}^{= 0}}{e_{Z_0}^{S^1}} \\ \vs{0.2cm}
							& = &  \ds  
							\int_{Z_{\min}} \frac{\left( (2+b_{\min}) u + 2\lambda \right)^3}{b_{\min} u\lambda + \lambda^2} + 
							\int_{Z_{\max}} \frac{\left( (2+b_{\max}) u - 2\lambda \right)^3}{-b_{\max} u\lambda + \lambda^2} \\ \vs{0.1cm}
							& = &  \ds  
							\int_{Z_{\min}} (\lambda - b_{\min}u)(12(2+b_{\min}) u\lambda^2 + 8\lambda^3) + 
							\int_{Z_{\max}} (\lambda + b_{\max}u)(12(2+ b_{\max}) u\lambda^2 - 8\lambda^3) \\ \vs{0.1cm}
							& = &\ds 24 + 4b_{\min} + 24 + 4b_{\max}.
		\end{array}
	\]
	By \eqref{equation_2_1_bminbmax}, this completes the proof. See Table \ref{table_II_1} and compare it with \eqref{equation_II_1}.
\end{proof}

\begin{remark}\label{remark_localization_surface}
	We use the following equations frequently for calculating the Chern numbers : 
	\[
		\left(c_1^{S^1}(TM)|_{Z_0} \right)^3 = 0, \quad \int_{Z_{\min}} \frac{\left(c_1^{S^1}(TM)|_{Z_{\min}}\right)^3}{e_{Z_{\min}}^{S^1}} = 24 + 4b_{\min}, \quad 
		\int_{Z_{\max}} \frac{\left(c_1^{S^1}(TM)|_{Z_{\max}}\right)^3}{e_{Z_{\max}}^{S^1}} = 24 + 4b_{\max}.
	\]
\end{remark}
\vs{0.3cm}

\begin{example}[Fano varieties of type {\bf (II-1)}]\label{example_II_1} 
	We denote by $T^k$ a $k$-dimensional compact torus, $\frak{t}$ the Lie algebra of $T$, and $\frak{t}^*$ the dual of $\frak{t}$. 
	We provide algebraic Fano examples for each topological fixed point data given in Theorem \ref{theorem_II_1} as follows. \vs{0.3cm}

	\begin{figure}[H]
		\scalebox{1}{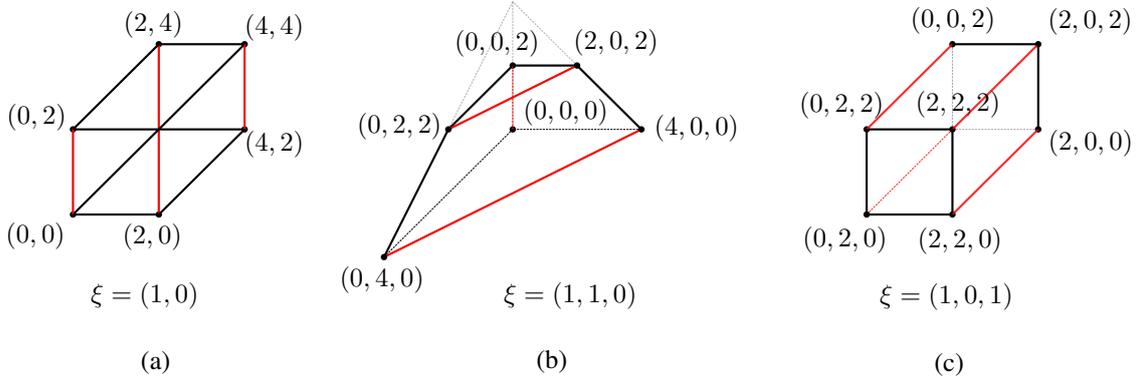}
		\caption{\label{figure_II_1} Fano varieties of type {\bf (II-1)}}
	\end{figure}          	
	
	\begin{enumerate}
		\item {\bf Case (II-1.1)} \cite[32nd in Section 12.3]{IP} : Let $W= \mcal{F}(3)$ be the complete flag variety of $\C^3$, or equivalently, a smooth divisor of bidegree $(1,1)$ in $\p^2 \times \p^2$
		(via the Pl\"{u}cker embedding).
		One can think of $M$ as a co-adjoint orbit 
		of $U(3)$. It is well-known that $M$ admits a unique $U(3)$-invariant monotone K\"{a}hler form $\omega$ (called a {\em Kirillov-Kostant-Souriau form}) such that 
		$c_1(TW) = [\omega]$. A maximal torus $T^2$ of $U(3)$ acts on $(W,\omega)$ in a Hamiltonian fashion with a moment map 
		\[
			\mu : W \rightarrow \frak{t}^*
		\]
		such that the moment map image can be described by Figure \ref{figure_II_1} (a), where edges corresponds to $T$-invariant spheres (called 1-skeleton in \cite{GKM}).   
		If we take a circle subgroup $S^1$ generated by $\xi = (1,0) \in \frak{t} \cong \R^2$, then the action is semifree and the balanced moment map is given by 
		\[
			\mu_\xi = \langle \mu, \xi \rangle - 2
		\] 
		The fixed point set for the $S^1$-action consists of three spheres 
		corresponding to the edges (colored by red in Figure \ref{figure_II_1} (a))
		\[
			e_1 = \overline{(0,0) ~(0,2)}, \quad e_2 = \overline{(2,0) ~(2,4)}, \quad e_3 = \overline{(4,2) ~(4,4)}
		\]
		The symplectic areas of the minimum $Z_{-2} = \mu^{-1}(e_1)$ and the maximum
		$Z_2 = \mu^{-1}(e_3)$ are both equal to $2 = 2 + b_{\min} = 2 + b_{\max}$ by Corollary \ref{corollary_volume_extremum} and hence 
		$b_{\min} = b_{\max} = 0$. Thus $W_{-2 + \epsilon} \cong S^2 \times S^2$ by Lemma \ref{lemma_Euler_extremum}. Therefore, the corresponding fixed point data
		should coincide with {\bf (II-1.1)} in Table \ref{table_II_1}.
		\vs{0.2cm}
		
		\item {\bf Case (II-1.2)} \cite[35th in Section 12.3]{IP} : Let $M = V_7$, the toric blow-up of $\p^3$ at a point. Then the moment polytope is given by 
		Figure \ref{figure_II_1} (b) where we denote the moment map by $\mu$. 
		If we take a circle subgroup generated by $\xi = (1,1,0) \in \frak{t}$, then 
		we can easily check that the $S^1$-action is semifree and the balanced moment map is given by 
		$\mu_\xi := \langle \mu, \xi \rangle - 2$.
		Moreover, the fixed components $Z_{-2}$, $Z_0$, and $Z_2$ are three spheres whose moment map images are the edges 
		(colored by red in Figure \ref{figure_II_1} (b))
		\[
			e_1 = \overline{(0,0,0) ~(0,0,2)}, \quad e_2 = \overline{(0,2,2) ~(2,0,2)}, \quad e_3 = \overline{(0,4,0) ~(4,0,0)}. 
		\]
		In this case, we have $Z_{-2} = \mu^{-1}(e_1)$ and $Z_2 = \mu^{-1}(e_3)$ with the symplectic areas 2 and 4, respectively.
		By Corollary \ref{corollary_volume_extremum}, we have $b_{\min} = 0$ and $b_{\max} = 2$ and so $M_{-2 + \epsilon} \cong S^2 \times S^2$
		by Lemma \ref{lemma_Euler_extremum}. Also, one can easily check that the fixed point data for the $S^1$-action equals {\bf (II-1.2)} in Table \ref{table_II_1}
		(see also \eqref{equation_2_1_bminbmax}).
		
		\vs{0.2cm}
		\item {\bf Case (II-1.3)} \cite[27th in Section 12.4]{IP} : Let $M = \p^1 \times \p^1 \times \p^1$ with the monotone K\"{a}hler form $\omega = 2\omega_{\mathrm{FS}} 
		\oplus 2\omega_{\mathrm{FS}} \oplus 2\omega_{\mathrm{FS}}$ so that $c_1(TM) = [\omega]$. Then the standard =Hamiltonian $T^3$-action admits a 
		moment map whose image is a  cube with side length 2, see Figure \ref{figure_II_1} (c). Take a circle subgroup $S^1$ of $T^3$ generated by $\xi = (1,0,1)$. 
		Then the induced $S^1$-action becomes semifree with the balanced moment map is given by $\mu_\xi = \langle \mu, \xi \rangle - 2$. It is easy to see that 
		there are four fixed components homeomorphic to spheres and their moment map images are 
		\[
			e_1 = \overline{(0,2,0) ~(0,0,0)}, \quad e_2 = \overline{(0,2,2) ~(0,0,2)}, \quad e_3 = \overline{(2,2,0) ~(2,0,0)}, \quad e_4 = \overline{(2,2,2) ~(2,0,2)}  
		\]
		colored by red in Figure \ref{figure_II_1} (c).
		Since $Z_{-2} = \mu^{-1}(e_1)$ and $Z_2 = \mu^{-1}(e_4)$ both have the symplectic area 2, we have $b_{\min} = b_{\max} = 0$ and this fixed point data coincides with
		{\bf (II-1.3)} in Table \ref{table_II_1}. 
	\end{enumerate} 
\end{example}

Now we consider the case of $M_0 \cong E_{S^2}$. 

\begin{theorem}\label{theorem_II_2}
	Let $(M,\omega)$ be a six-dimensional closed monotone semifree Hamiltonian $S^1$-manifold with $c_1(TM) = [\omega]$. Suppose that $\mathrm{Crit} H = \{ 2, 0, -2\}$ 
	and $M_0 \cong E_{S^2}$. Then, up to orientation of $M$, the list of all possible topological fixed point data is given by
		\begin{table}[H]
			\begin{tabular}{|c|c|c|c|c|c|c|c|c|}
				\hline
				    & $(M_0, [\omega_0])$ & $e(P_{-2}^+)$ &$Z_{-2}$  & $Z_0$ & $Z_2$ & $b_2(M)$ & $c_1^3(M)$ \\ \hline \hline
				    {\bf (II-2.1)} & $(E_{S^2}, 3x + 2y)$ & $-x -y$  &$S^2$ & 
				    	\makecell{ $Z_0 = Z_0^1 ~\dot \cup ~ Z_0^2$ \\
					    $Z_0^1 \cong Z_0^2 \cong S^2$ \\ $\mathrm{PD}(Z_0^1) = y$, $\mathrm{PD}(Z_0^2) = x+y$}  & $S^2$ & $3$ & $48$\\ \hline    
				    {\bf (II-2.2)} & $(E_{S^2}, 3x + 2y)$ & $-x-y$  &$S^2$ & $Z_0 \cong S^2, ~\mathrm{PD}(Z_0) = 2x+2y$ & $S^2$ & $2$ &$40$ \\ \hline
			\end{tabular}
			\vs{0.5cm}			
			\caption{\label{table_II_2} Topological fixed point data for $\mathrm{Crit} H = \{-2, -0, 2\}$, $M_0 \cong E_{S^2}$}
		\end{table}
			\vs{-0.7cm}
\end{theorem}

\begin{proof}
	The idea of the proof is essentially similar to the proof of Theorem \ref{theorem_II_1}.
	
	In this case, Lemma \ref{lemma_Euler_extremum} implies that $b_{\mathrm{min}} = 2k+1$ and $e(P_{-2}^+) = kx - y$ for some integer $k \in \Z$.
	If we denote by $\mathrm{PD}(Z_0) = ax + by \in H^2(M_0; \Z)$ for some $a,b \in \Z$, then it follows that 
          \[
	          \langle c_1(TM_0), [Z_0] \rangle > 0, \quad  \mathrm{Vol}(Z_{-2}) = 2k + 3 > 0 \quad \quad \Rightarrow \quad \quad 2a+b \geq 1, \quad k \geq -1.
          \]  by Corollary \ref{corollary_volume_extremum}.
	Also, by the Duistermaat-Heckman theorem \eqref{equation_DH}, we obtain
	\[
		[\omega_2] = [\omega_0] - 2(kx - y + \mathrm{PD}(Z_0)) = (3-2a-2k)x + (4-2b)y.
	\]
          Since $\lim_{t \rightarrow 2} \int_{M_t} [\omega_t]^2 = 0$, we have 
          \[
          		2(3-2a-2k)(4-2b) - (4-2b)^2 = 0 \quad \Rightarrow \quad b=2 \hs{0.3cm} \text{or} \hs{0.3cm} 1+b = 2a + 2k 
          \]
          Note that in  the latter case, $b$ becomes odd and this implies that 
          \begin{equation}\label{equation_2_2_bmax}
	          	\langle e(P_2^-)^2, [M_0] \rangle = \langle ((a+k)x + (b-1)y)^2, [M_0] \rangle = 2(a+k)(b-1) - (b-1)^2  \equiv 0 ~\mod 2
	\end{equation}
	which contradicts that $- b_{\max} = \langle e(P_2^-)^2, [M_0] \rangle$ is odd by Lemma \ref{lemma_Euler_extremum}
	(since $M_{2-\epsilon} \cong M_0 \cong E_{S^2}$). 
	Consequently, we get 
	\begin{equation}\label{equation_2_2}
		b=2, \quad a \geq 0, \quad k \geq -1, \quad a + k \leq 1 ~( \Leftrightarrow ~b_{\max} + 2 = \mathrm{vol}(Z_{\max}) \geq 1). 
	\end{equation}
          Therefore, all possible solutions $(k,a,b)$ to \eqref{equation_2_2} are given by 
          \[
          		(-1,0,2), (-1,1,2), (-1, 2, 2), (0,0,2), (0,1,2),(1,0,2).
          \]          
          Applying the adjunction formula, we may rule out some solutions : if $a=0$, then $\mathrm{PD}(Z_0) = 2y$ so that 
          we have $[Z_0] \cdot [Z_0] = -4$ and $\langle c_1(TM_0), [Z_0] \rangle = 2$ and hence there are at most two connected component in $Z_0$. 
          On the other hand, the adjunction formula \eqref{equation_adjunction} implies that 
          \[
		\underbrace{[Z_0] \cdot [Z_0]}_{= ~-4} + \sum (2 - 2g_i) = \langle c_1(TM_0), [Z_0] \rangle = 2
          \]
	and therefore there should be at least three spheres, which contradicts that $Z_0$ consists of at most two connected components.
	Also, if $(k,a,b)=(0,1,2)$, then the formula \eqref{equation_2_2_bmax} induces that 
	 $b_{\min} = 1$ and $b_{\max} = -1$ (in particular $b_{\min} > b_{\max}$) and hence we may rule out this case by \eqref{equation_assumption}.
	To sum up, we have only two possible cases  : \vs{0.2cm}
	
	\noindent
	{\bf (II-2.1)} : $(k,a,b)=(-1,1,2)$. In this case, $[Z_0] \cdot [Z_0] = 0$ and $\langle c_1(TM_0), [Z_0] \rangle = 4$, $b_{\min} = -1$ and $b_{\max} = 1$. The adjunction formula 
	implies that there are at least two spheres denoted by $C_1$ and $C_2$ where the followings are satisfied : 
		\begin{itemize}
			\item $1 \leq \langle [\omega_0], [C_i] \rangle \leq 3$. 
			\item $2 \leq \langle [\omega_0], [C_1] + [C_2]  \rangle \leq 4$.
			\item $[C_1] \cdot [C_2] =0$.
		\end{itemize}
	Let $\mathrm{PD}(C_1) = p x + q y$. If $\langle [\omega_0], [C_1] \rangle = 2p + q =  1$, then $2pq - q^2 = -1$ by the adjunction formula so that 
	we have $(p,q) = (0,1)$. Similarly, if $\langle [\omega_0], [C_1] \rangle = 2p + q =  2$, then we have $2pq - q^2 = 0$ and hence 
	\[
		q = 0 ~(p = 1) \quad \text{or} \quad q = 2p ~(4p = 2). 
	\]
	So, we have $(p,q) = (1,0)$. 
	
	Note that if $\langle [\omega_0], [C_i] \rangle \leq 2$ for every $i = 1,2$, since $[C_1] \cdot [C_2] = 0$, the only possible case is 
	$\langle [\omega_0], [C_i] \rangle = 2$ for every $i=1,2.$ However, this cannot be happened since $\mathrm{PD}(C_1) + \mathrm{PD}(C_2) \neq x + 2y$.
	Thus the only possibility is that $\langle [\omega_0], [C_1] \rangle = 1$ and $\langle [\omega_0], [C_2] \rangle = 3$.
	Therefore, we obtain $\mathrm{PD}(C_1) = y$, $\mathrm{PD}(C_2) = x+y$, and $C_1 \cong C_2 \cong S^2$. See Table \ref{table_II_2}: {\bf (II-2.1)}.
	\vs{0.3cm}
			
	\noindent
	{\bf (II-2.2)} : $(k,a,b)=(-1,2,2)$. In this case, we have $[Z_0] \cdot [Z_0] = 4$ and $\langle c_1(TM_0), [Z_0] \rangle = 6$, $b_{\min} = -1$ and $b_{\max} = -1$. 
	By the adjunction formula, there exists a component $C \cong S^2$ of $Z_0$ where we denote by $\mathrm{PD}(C) = px + qy$.  Then, we have 
	\[
		[C] \cdot ([Z_0] - [C]) = \langle (px + qy) \cdot ((2-p)x + (2-q)y), [M_0] \rangle = 0 \quad \Leftrightarrow \quad -2pq + 2p + q^2 = 0.
	\]
	Also, since 
	\[
		V := \mathrm{vol}(C) = [C] \cdot [C] + 2 = \langle (px+qy)^2, [M_0] \rangle + 2 = 2pq - q^2 + 2, 
	\] we get $2p + 2 - V = 0$. If $V = 6$, then $Z_0$ is connected so that we are done. 
	If $V=2$, then $p = q= 0$ which is impossible. Finally if $V=4$, then $p = 1$ and $q^2 - 2q + 2 = 0$ whose solution cannot be real. Therefore, we have $V=6$
	and $Z_0$ is connected and homemorphic to $S^2$. See Table \ref{table_II_2}: {\bf (II-2.2)}.
	\vs{0.3cm}
	
	Note that the Chern number computations in Table \ref{table_II_2} immediately follow from Remark \ref{remark_localization_surface}.

\end{proof}

\begin{example}[Fano varieties of type {\bf (II-2)}]\label{example_II_2} 

	We illustrate algebraic Fano varieties with holomorphic Hamiltonian torus actions having each topological fixed point data given in Theorem \ref{theorem_II_2}.\vs{0.3cm}
	
	\begin{figure}[h]
		\scalebox{1}{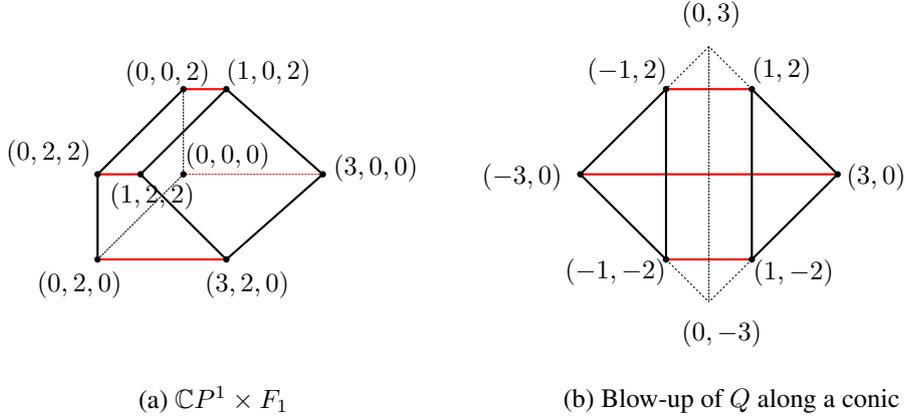}
		\caption{\label{figure_II_2} Fano varieties of type {\bf (II-2)}}
	\end{figure}          	
		
	\begin{enumerate}
		\item {\bf Case (II-2.1)} \cite[28th in Section 12.4]{IP} : Let $M = \p^1 \times F_1$ where $F_1 = \p(\mcal{O} \oplus \mcal{O}(1))$ is the Hirzebruch surface.
		Equip $M$ with the toric K\"{a}hler form $\omega$ such that $c_1(TM) = [\omega]$ so that the moment map $\mu : M \rightarrow \frak{t}^*$ 
		has the image FIgure \ref{figure_II_2} (a). 
		If we take a circle subgroup $S^1$ generated by $\xi = (0,1,-1) \in \frak{t}$, then one can check that the action is semifree and the balanced moment map is given by 
		\[
			\mu_\xi = \langle \mu, \xi \rangle.
		\] 
		The fixed point set for the $S^1$-action has four connected components each of which are all spheres and have the moment map images 
		(colored by red in Figure \ref{figure_II_2} (a))
		\[
			e_1 = \overline{(0,0,2) ~(1,0,2)}, \quad e_2 = \overline{(0,2,2) ~(1,2,2)}, \quad e_3 = \overline{(0,0,0) ~(3,0,0)}, \quad e_4 = \overline{(0,2,0) ~(3,2,0)}.
		\]
		The symplectic areas of the minimum $Z_{-2} = \mu^{-1}(e_1)$ and the maximum
		$Z_2 = \mu^{-1}(e_4)$ are 1 and 3, respectively, so that $b_{\min} = -1$ and $b_{\max} = 1$ by Corollary \ref{corollary_volume_extremum}. 
		Thus $M_{-2 + \epsilon} \cong E_{S^2}$ by Lemma \ref{lemma_Euler_extremum} and the corresponding fixed point data
		coincides with {\bf (II-2.1)} in Table \ref{table_II_2}.
		\vs{0.2cm}
		
		\item {\bf Case (II-2.2)} \cite[29th in Section 12.3]{IP} : Let $M$ be a smooth quadric in $\p^4$. As a co-adjoint orbit of $SO(5)$, $M$ admits a $SO(5)$-invariant 
		monotone K\"{a}hler form $\omega$ such that $c_1(TM) = [\omega]$. With respect to the maximal torus $T^2$-action on $(M,\omega)$, we get a moment map
		$\mu : M \rightarrow \frak{t}^*$ whose image is a square with four vertices $(0, \pm 3)$, $(\pm 3, 0)$ (see Figure \ref{figure_II_2} (b)). 
		Let $C$ be the $T^2$-invariant sphere $\mu^{-1}(\overline{(0,-3) ~(0,3)})$ and define 
		\[
			\widetilde{M} := ~\text{$T^2$-equivariant (or GKM) blow-up of $M$ along $C$}
		\]
		where the {\em $T^2$-equivariant blowing up} can be done via the following two steps:\vs{0.2cm}
		\begin{itemize}
			\item Take a $T^2$-equivariant neighborhood $\mcal{U}$ of $C$, isomorphic to some $T^2$-equivariant $\C^2$-bundle over $\p^1$, and extend the $T^2$-action 
			to (any effective Hamiltonian) $T^3$-action so that we get a toric model. 
			\item Take the toric blow-up of $\mcal{U}$ along the zero section, i.e., $C$, and restrict the toric action to the original $T^2$-action. 
		\end{itemize}
		\vs{0.2cm}
		The resulting moment map image is given in Figure \ref{figure_II_2} (b). 
		
		Now, we take a circle subgroup generated by $\xi = (0,1) \in \frak{t}$. One can directly check that 
		the $S^1$-action is semifree and the balanced moment map is given by 
		$\mu_\xi := \langle \mu, \xi \rangle - 2$.
		Moreover, the fixed components $Z_{-2}$, $Z_0$, and $Z_2$ are given by 
		\[
			Z_{-2} = \mu^{-1}(e_1), \quad Z_{-2} = \mu^{-1}(e_2), \quad Z_{-2} = \mu^{-1}(e_3)
		\]
		where 
		\[
			e_1 = \overline{(-1,-2) ~(1,-2)}, \quad e_2 = \overline{(-3,0) ~(3,0)}, \quad e_3 = \overline{(-1,2) ~(1,2)}
		\]
		(colored by red in Figure \ref{figure_II_2} (b). In particular, we have $\mathrm{vol}(Z_{-2}) = \mathrm{vol}(Z_{-2}) = 1$ so that  
		$b_{\min} = b_{\max} = -1$. By Lemma \ref{lemma_Euler_extremum}, we have 
		 $M_{-2 + \epsilon} \cong S^2 \times S^2$. So, the fixed point data for the $S^1$-action coincides with {\bf (II-2.2)} in Table \ref{table_II_2}. \\
		
	\end{enumerate} 
	
\end{example}

\section{Case III : $\mathrm{Crit} ~\mathring{H} = \{-1, 1\}$}	
\label{secCaseIIIMathrmCritMathringH11}

	In this section, we classify all TFD in the case where $\mathrm{Crit} \mathring{H} = \{-1, 1\}$. 
	Let $m = |Z_{-1}|$ ($ m \in \Z_{>0}$) be the number of isolated fixed points of index two. By the Poincar\'{e} duality, we have $|Z_1| = m$. Applying the localization theorem to 
	$1 \in H^0_{S^1}(M)$ and $c_1^{S^1}(TM) \in H^2_{S^1}(M)$, 
we obtain
	\begin{equation}\label{equation_3_localization_0}
		\begin{array}{ccl}\vs{0.2cm}
			0 = \ds \int_M 1 & = & \ds  \int_{Z_{\min}} \frac{1}{e_{Z_{\min}}^{S^1}} + m \cdot \frac{1}{-\lambda^3} + m \cdot \frac{1}{\lambda^3} + 
							\int_{Z_{\max}} \frac{1}{e_{Z_{\max}}^{S^1}}  \\ \vs{0.2cm}
							& = &  \ds  
							\int_{Z_{\min}} \frac{1}{b_{\min} u\lambda + \lambda^2} +
							\int_{Z_{\max}} \frac{1}{-b_{\max} u\lambda + \lambda^2} \\ \vs{0.2cm}
							& = &\ds \frac{- b_{\min} + b_{\max}}{\lambda^3} 
		\end{array}
	\end{equation}
	and
	\begin{equation}\label{equation_3_localization_c1}
		\begin{array}{ccl}\vs{0.2cm}
			0 =  \ds \int_M c_1^{S^1}(TM) & = & \ds  \int_{Z_{\min}} \frac{c_1^{S^1}(TM)|_{Z_{\min}}}{e_{Z_{\min}}^{S^1}} + m \cdot \frac{\lambda}{-\lambda^3} + 
							m \cdot \frac{-\lambda}{\lambda^3} + 
							\int_{Z_{\max}} \frac{c_1^{S^1}(TM)|_{Z_{\max}}}{e_{Z_{\max}}^{S^1}}  \\ \vs{0.2cm}
							& = &  \ds  
							\int_{Z_{\min}} \frac{2\lambda + (b_{\min}+2)u}{b_{\min} u\lambda + \lambda^2} -2m \cdot  \frac{\lambda}{\lambda^3} + 
							\int_{Z_{\max}} \frac{- 2\lambda + (b_{\max} + 2)u}{-b_{\max} u\lambda + \lambda^2} \\ \vs{0.2cm}
							& = &\ds \frac{- b_{\min} - b_{\max} -2m +4}{\lambda^2}.
		\end{array}
	\end{equation}
	From \eqref{equation_3_localization_0} and \eqref{equation_3_localization_c1}, we get $b_{\max} = b_{\min}$ and $b_{\min} + m = 2$. Moreover, Corollary \ref{corollary_volume_extremum}
	implies that $b_{\min} \geq -1$ and therefore we have three possible cases : 
	\[
		(b_{\min}, m) = (1,1), (0,2), (-1, 3).
	\]
	Therefore we obtain the following.
	
	\begin{theorem}\label{theorem_III}
	Let $(M,\omega)$ be a six-dimensional closed monotone semifree Hamiltonian $S^1$-manifold with $c_1(TM) = [\omega]$. Suppose that $\mathrm{Crit} H = \{ 2, 1, -1, -2\}$. 
	Then the list of all possible topological fixed point data is given by
		\begin{table}[H]
			\begin{tabular}{|c|c|c|c|c|c|c|c|c|}
				\hline
				    & $(M_0, [\omega_0])$ & $e(P_{-2}^+)$ &$Z_{-2}$  & $Z_{-1}$ & $Z_1$ & $Z_2$ & $b_2(M)$ & $c_1^3(M)$ \\ \hline \hline
				    {\bf (III.1)} & $(E_{S^2} \# ~\overline{\p^2}, 3x + 2y - E_1)$ & $-y$  &$S^2$ & 
				    	{\em pt} & {\em pt} & $S^2$ & $2$ & $54$\\ \hline    
				    {\bf (III.2)} & $(S^2 \times S^2 \# ~2\overline{\p^2}, 2x + 2y - E_1 - E_2)$ & $-y$  &$S^2$ & 
				    	{\em 2 pts} & {\em 2 pts}  & $S^2$ & $3$ & $44$\\ \hline    
				    {\bf (III.3)} & $(E_{S^2} \# ~\overline{\p^2}, 3x + 2y - E_1)$ & $-x-y$  &$S^2$ & {\em 3 ~pts} & {\em 3 ~pts} & $S^2$ & $4$ &$34$ \\ \hline
			\end{tabular}
			\vs{0.5cm}
			\caption{\label{table_III} Topological fixed point data for $\mathrm{Crit} H = \{-2, -1,1, 2\}$.}
		\end{table}
			\vs{-0.7cm}
	\end{theorem}

	\begin{proof}
		The formula follows from Lemma \ref{lemma_Euler_extremum} that $e(P_{-2}^+) = kx - y$ with $b_{\min} = 2k +1$. 
		Also the Chern number computations can be easily obtained by 
		Remark \ref{remark_localization_surface}. 
	\end{proof}	

\begin{example}[Fano varieties of type {\bf (III)}]\label{example_III} 
We provide algebraic Fano varieties with holomorphic Hamiltonian $S^1$-action with topological fixed point data given in Theorem \ref{theorem_III} as follows.

	\begin{figure}[h]
		\scalebox{1}{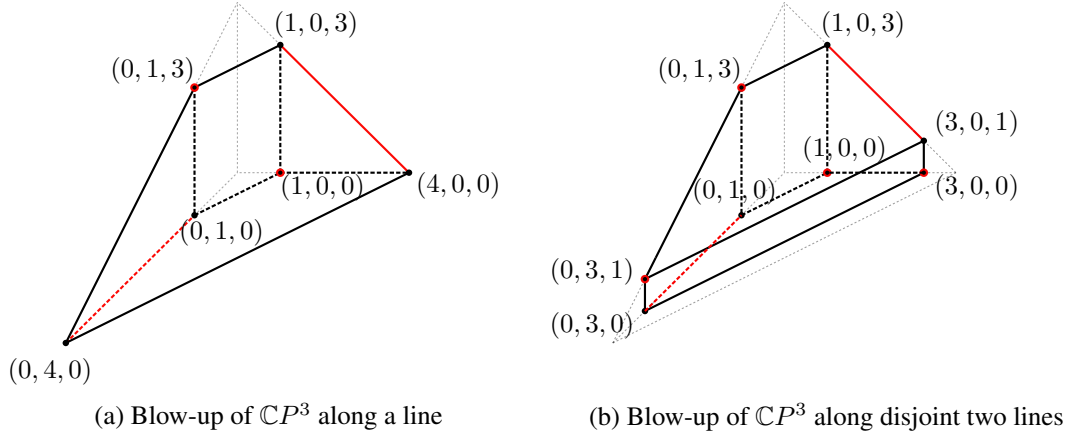}
		\caption{\label{figure_III} Fano varieties of type {\bf (III)}}
	\end{figure}          	

	\begin{enumerate}
	            \item {\bf Case (III.1)} \cite[33rd in Section 12.3]{IP} : Let $M$ be the toric blow-up of $\p^3$ along a $T^3$-invariant line. WIth respect to the $T^3$-invariant 
	            normalized monotone K\"{a}hler form, we get a moment map $\mu : M \rightarrow \frak{t}^*$ whose image is given by Figure \ref{figure_III} (a). If we take a circle 
	            subgroup $S^1$ generated by $\xi = (1,0,1) \in \frak{t}$, then the action is semifree with the balanced moment map $\mu_\xi = \langle \mu, \xi \rangle - 2$ 
	            and the fixed point set consists of 
	            \[
	            	Z_{-2} = \mu^{-1}(e_1), \quad Z_{-1} = \mu^{-1}(1,0,0), \quad Z_1 = \mu^{-1}(0,1,3), \quad \mu^{-1}(e_2)
	            \]
	            where $e_1 = \overline{(0,1,0) ~(0,4,0)}$ and $e_2 = \overline{(1,0,3) ~(4,0,0)}$. Note that $\mathrm{Vol}(Z_{-2}) = \mathrm{Vol}(Z_2) = 3$ and so 
	            $b_{\min} = b_{\max} = 1$ by Corollary \ref{corollary_volume_extremum}. Thus the fixed point data for the $S^1$-action coincides with 
	            Table \ref{table_III} {\bf (III.1)}. \vs{0.2cm}
	            
           	 \item {\bf Case (III.2)} \cite[25th in Section 12.4]{IP} : Let $M$ be the toric blow-up of $\p^3$ along two disjoint $T^3$-invariant lines. Then the image 
           	 of a moment map $\mu : M \rightarrow \frak{t}^*$
           	 (with respect to the normalized $T^3$-invariant K\"{a}hler form) is described as in Figure \ref{figure_III} (b). One can easily check that the circle action generated by $\xi = (1,0,1)
           	 \in \frak{t}$ is semifree and the balanced moment map is given by $\mu_\xi = \langle \mu, \xi \rangle - 2$. The fixed components are 
           	 \[
           	 	Z_{-2} = \mu^{-1}(e_1), \quad Z_{-1} = \{ (0,3,1), (1,0,0) \}, \quad Z_1 = \{ (0,1,3), (3,0,0)\}, \quad Z_2 = \mu^{-1}(e_2)
           	 \]
           	 where $e_1 = \overline{(0,3,0) ~(0,1,0)}$ and $e_2 = \overline{(1,0,3) ~(3,0,1)}$. As the symplectic volumes of $Z_{-2}$ and $Z_2$ are both 2, we have 
           	 $b_{\min} = b_{\max} = 0$ and so the fixed point data of the action is the same as Table \ref{table_III} {\bf (III.2)}.

           	 \item {\bf Case (III.3)} \cite[6th in Section 12.5]{IP} : Consider $M = \p^1 \times \p^1 \times \p^1$ equipped with the normalized monotone K\"{a}hler form $\omega$ on $M$
           	 with the standard $\omega$-compatible integrable complex structure $J$ on $M$. 
           	 Consider the standard $T^3$-action on $(M,\omega)$ with a moment map given by 
           	 \[
           	 	\mu([x_0, x_1], [y_0, y_1], [z_0, z_1]) = \left( \frac{2x_0|^2}{|x_0|^2 + |x_1|^2}, \frac{2|y_0|^2}{|y_0|^2 + |y_1|^2}, \frac{2|z_0|^2}{|z_0|^2 + |z_1|^2} \right).
           	 \]
           	 For the diagonal circle subgroup
           	 \[
           	 	S^1 = \{(t,t,t) ~|~ t \in S^1 \} \subset T^3, 
           	 \]
           	 generated by $\xi = (1,1,1) \in \frak{t}$, 
           	 the induced $S^1$-action on $(M,\omega, J)$ is semifree with the balanced moment map $\mu_\xi = \langle \mu, \xi \rangle - 3$. See Figure 2 in \cite[Example 6.6]{Cho}.
           	 
           	 Now, we take the $S^1$-invariant diagonal sphere $D = \{ \left([z_0, z_1], [z_0, z_1], [z_0, z_1] \right) ~|~ [z_0, z_1] \in \p^1 \}$ in $M$, which is obviously a K\"{a}hler 
           	 submanifold of $(M,\omega,J)$. One can obtain an equivariant blowing-up $(\widetilde{M}, \widetilde{\omega}, \widetilde{J})$ 
           	 of $(M,\omega,J)$ along $D$ as follows (where the construction seems to be well-known to experts): \vs{0.2cm}
           	 \begin{itemize}
           	 	\item Let $\mcal{U}$ be a sufficiently small $T^3$-invariant neighborhood of $D$ such that $\mcal{U}$ equipped with the induced K\"{a}hler structure is 
           	 	$S^1$-equivariantly isomorphic to 
           	 	some neighborhood of the zero section of $E_D := \mcal{O}(k_1) \oplus \mcal{O}(k_2)$ where 
           	 	\begin{itemize}
           	 		\item $E_D$ is equipped with the K\"{a}hler structure whose restriction on each fiber of $E_D$ equals the standard symplectic form on $\C \oplus \C$, \vs{0.1cm} 
           	 		\item $E_D$ admits an $S^1$-action compatible with the bundle structure such that the normal bundle $\nu_D$ of $D$ in $M$ is 
           	 		$S^1$-equivariantly isomorphic to $E_D$. \vs{0.1cm}
           	 	\end{itemize}  
			Note that each $\mcal{O}(k_i)$ has a fiberwise circle action so that $E_D$ has a fiberwise $T^2$-action. Together with the $S^1$-action given, $E_D$ becomes 
			a (non-complete) toric variety and a zero section becomes $T^3$-invariant. \vs{0.1cm}
						           	 	  
			\item Equip $\mcal{U}$ the toric structure (called a {\em local toric structure near $D$}) 
			induced by the $T^3$-action on $E_D$. Then one can obtain a toric blow-up of $\mcal{U}$ along $D$ so that we obtain a new K\"{a}hler manifold, say  
			$(\widetilde{M}, \widetilde{\omega}, \widetilde{J})$. We finally restrict the $T^3$-action to the $S^1$-subgroup of $T^3$. 
			
			\begin{figure}[h]
				\scalebox{1}{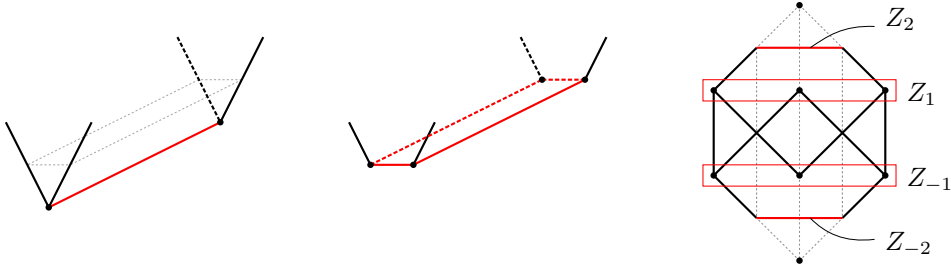}
				\caption{\label{figure_III_blowup} Blow up along an $S^1$-invariant sphere}
			\end{figure}          	
           	 \end{itemize}

	It is not hard to see that the induced $S^1$-action on $\widetilde{M}$ is semifree. Also, new fixed components which appear on $\widetilde{M}$ 
	instead of two isolated fixed points on $D$ in $M$ are two spheres and hence the fixed point data coincides with Table \ref{table_III} {\bf (III.3)}
	(see Figure \ref{figure_III_blowup}). \\
           	 
          \end{enumerate}		
\end{example}

\section{Case IV : $\mathrm{Crit} ~\mathring{H} = \{-1, 0, 1\}$}	
\label{secCaseIVMathrmCritMathringH11}

	In this section, we classify all TFD in the case where $\mathrm{Crit} \mathring{H} = \{-1,0,1\}$. 
	Let $m = |Z_{-1}|= |Z_1| > 0$ be the number of fixed points of index two. 

\begin{lemma}\label{lemma_number_indextwo}
	We have $m=1$ or $2$.
\end{lemma}

\begin{proof}
	Applying the localization theorem to $c_1^{S^1}(TM)$, we obtain
	\[
		\begin{array}{ccl}\vs{0.2cm}
			0 & = & \ds \int_M c_1^{S^1}(TM) \\  \vs{0.2cm}
					& = &  \ds  
							\int_{Z_{\min}} \frac{c_1^{S^1}(TM)|_{Z_{\min}}}{e_{Z_{\min}}^{S^1}} + m \cdot \frac{\lambda}{-\lambda^3} + m \cdot \frac{-\lambda}{\lambda^3} + 
							\int_{Z_0} \frac{c_1^{S^1}(TM)|_{Z_0}}{e_{Z_0}^{S^1}} + \int_{Z_{\max}} \frac{c_1^{S^1}(TM)|_{Z_{\max}}}{e_{Z_{\max}}^{S^1}}  \\ \vs{0.2cm}
							& = &  \ds  
							\int_{Z_{\min}} \frac{2\lambda + (b_{\min}+2)u}{b_{\min} u\lambda + \lambda^2} -2m \cdot  \frac{\lambda}{\lambda^3}
							+ \int_{Z_0} \frac{\overbrace{c_1(TM)|_{Z_0}}^{=~\mathrm{Vol}(Z_0)}}{(b^- - b^+) u\lambda - \lambda^2} + 
							\int_{Z_{\max}} \frac{- 2\lambda + (b_{\max} + 2)u}{-b_{\max} u\lambda + \lambda^2} \\ \vs{0.2cm}
							& = &\ds \frac{- b_{\min} - b_{\max} -2m +4 - \mathrm{Vol}(Z_0)}{\lambda^2}
		\end{array}
	\]
	where $b^+$ and $b^-$ denote the Chern numbers of the positive and negative normal bundle of $Z_0$ in $M$, respectively. 
	So, we have 
	\begin{equation}\label{equation_m}
		b_{\min} + b_{\max} + 2m + \mathrm{Vol}(Z_0) = 4.
	\end{equation} 
	Moreover, since 
	$b_{\min}, ~b_{\max} \geq -1$ by Corollary \ref{corollary_volume_extremum}, 
	we have $m \leq 2.$ 
\end{proof}

By Lemma \ref{lemma_number_indextwo}, we may divide the classification into two cases: $m=1$ and $m=2$. 
Indeed, it follows directly from \eqref{equation_m} that there are 13 solutions for $(m, \mathrm{Vol}(Z_0), b_{\min}, b_{\max})$: 
	\begin{equation}\label{equation_8_solutions}
		m=2, ~\begin{cases}
			\underline{\mathrm{Vol}(Z_0) = 2, (b_{\min}, b_{\max}) = (-1,-1)}\\
			\underline{\mathrm{Vol}(Z_0) = 1, (b_{\min}, b_{\max}) = (-1,0)}\\
			\mathrm{Vol}(Z_0) = 1, (b_{\min}, b_{\max}) = (0,-1)\\
		\end{cases}
		m=1, ~\begin{cases}
			\underline{\mathrm{Vol}(Z_0) = 4, (b_{\min}, b_{\max}) = (-1,-1)}\\
			\underline{\mathrm{Vol}(Z_0) = 3, (b_{\min}, b_{\max}) = (-1,0)}\\
			\mathrm{Vol}(Z_0) = 3, (b_{\min}, b_{\max}) = (0,-1)\\
			\underline{\mathrm{Vol}(Z_0) = 2, (b_{\min}, b_{\max}) = (-1,1)}\\
			\underline{\mathrm{Vol}(Z_0) = 2, (b_{\min}, b_{\max}) = (0,0)}\\
			\mathrm{Vol}(Z_0) = 2, (b_{\min}, b_{\max}) = (1,-1)\\						
			\underline{\mathrm{Vol}(Z_0) = 1, (b_{\min}, b_{\max}) = (-1,2)}\\						
			\underline{\mathrm{Vol}(Z_0) = 1, (b_{\min}, b_{\max}) = (0,1)}\\						
			\mathrm{Vol}(Z_0) = 1, (b_{\min}, b_{\max}) = (1,0)\\						
			\mathrm{Vol}(Z_0) = 1, (b_{\min}, b_{\max}) = (2,-1)\\												
		\end{cases}
	\end{equation}
	As \eqref{equation_assumption}, we may rule out the case of ``$b_{\min} > b_{\max}$'', and therefore we only need to deal with 8 solutions (underlined in \eqref{equation_8_solutions})
	with $b_{\min} \leq b_{\max}$
	and obtain the following. 

	\begin{theorem}\label{theorem_IV_1}
		Let $(M,\omega)$ be a six-dimensional closed monotone semifree Hamiltonian $S^1$-manifold with $c_1(TM) = [\omega]$. Suppose that $\mathrm{Crit} H = \{ 2, 1, 0, -1, -2\}$. 
		If the number of fixed points of index two equals two, up to orientation of $M$, the list of all possible topological fixed point data is given in the Table \ref{table_IV_1}
		\begin{table}[h]
			\begin{tabular}{|c|c|c|c|c|c|c|c|c|c|}
				\hline
				    & $(M_0, [\omega_0])$ & $e(P_{-2}^+)$ &$Z_{-2}$  & $Z_{-1}$ & $Z_0$ & $Z_1$ & $Z_2$ & $b_2(M)$ & $c_1^3(M)$ \\ \hline \hline
				    {\bf (IV-1-1.1)} & \makecell{$(E_{S^2} \# ~2\overline{\p^2},$ \\$3x + 2y - E_1-E_2)$} & $-x-y$  &$S^2$ & 
				    	{\em 2 pts} &
				    		\makecell{ $Z_0 = Z_0^1 ~\dot \cup ~ Z_0^2$ \\ $Z_0^1 \cong Z_0^2 \cong S^2$ \\ 
				    		$\mathrm{PD}(Z_0^1) = x+y-E_1 - E_2$ \\ $\mathrm{PD}(Z_0^2) = x - E_1$}
					     & {\em 2 pts} & $S^2$ & $5$ & $36$\\ \hline    
				    {\bf (IV-1-1.2)} & \makecell{$(E_{S^2} \# ~2\overline{\p^2},$ \\$3x + 2y - E_1-E_2)$} & $-x-y$  &$S^2$ & 
				    	{\em 2 pts} &
				    		\makecell{ $Z_0 = Z_0^1 ~\dot \cup ~ Z_0^2$ \\ $Z_0^1 \cong Z_0^2 \cong S^2$ \\ $\mathrm{PD}(Z_0^1) = y$ \\ 
				    		$\mathrm{PD}(Z_0^2) = x+y-E_1 - E_2$}
					     & {\em 2 pts} & $S^2$ & $5$ & $36$\\ \hline    					     
				    {\bf (IV-1-1.3)} & \makecell{$(E_{S^2} \# ~2\overline{\p^2},$ \\$3x + 2y - E_1-E_2)$} & $-x-y$  &$S^2$ & 
				    	{\em 2 pts} &
				    		\makecell{ $Z_0 \cong S^2$  \\ $\mathrm{PD}(Z_0) = x+y-E_1$}
					     & {\em 2 pts} & $S^2$ & $4$ & $36$\\ \hline    
				    {\bf (IV-1-2)} & \makecell{$(E_{S^2} \# ~2\overline{\p^2},$ \\$3x + 2y - E_1-E_2)$} & $-x-y$  &$S^2$ & 
				    	{\em 2 pts} & 
				    		\makecell{ $Z_0 \cong S^2$  \\ $\mathrm{PD}(Z_0) = x - E_1$}
				    	& {\em 2 pts}  & $S^2$ & $4$ & $40$\\ \hline 
			\end{tabular}		
			\vs{0.5cm}			
			\caption{\label{table_IV_1} Topological fixed point data for $\mathrm{Crit} H = \{-2, -1,0,1, 2\}$ with $|Z_{-1}| = 2$.}
		\end{table}				   
	\end{theorem}
		
	\begin{proof}
		
		As in \eqref{equation_8_solutions}, $b_{\min} = -1$ so that $M_{-2 + \epsilon} \cong E_{S^2}$ by Lemma \ref{lemma_Euler_extremum}, and therefore $M_0$ is a 
		two points blow-up of $E_{S^2}$ where we denote the dual classes of the exceptional divisors are denote by $E_1$ and $E_2$. 
		Also, we have $e(P_{-2}^+) = kx - y = -x -y$ as 
		$b_{\min} = 2k+1 = -1$.
		
		Let $\mathrm{PD}(Z_0) = ax + by + cE_1 + dE_2$ for some $a,b,c,d \in \Z$. By the Duistermaat-Heckman theorem \eqref{equation_DH}, we have 
		\[
			\begin{array}{ccl}
				[\omega_1] = [\omega_0] - e(P_0^+) & = & (3x + 2y - E_1 - E_2) - (-x - y + E_1 + E_2 + \mathrm{PD}(Z_0)) \\ 
				& = &(4-a)x + (3-b)y - (2+c)E_1 - (2+d)E_2
			\end{array}
		\] 
		where $[\omega_0] = c_1(TM_0) = 3x + 2y - E_1 - E_2$ and $e(P_0^+) = -x - y + E_1 + E_2 + \mathrm{PD}(Z_0)$ by Lemma \ref{lemma_Euler_class}.
		Observe that exactly two blow-downs occur simultaneously at $M_1$ and we denote by $C_1, C_2$ the vanishing cycles so that
		\begin{equation}\label{equation_vanishing_IV_1}
			\langle [\omega_1], C_1 \rangle = \langle [\omega_1], C_2 \rangle = 0, \quad C_1 \cdot C_2 = 0.
		\end{equation}
		By Lemma \ref{lemma_list_exceptional}, 
		the list of all possible $(\mathrm{PD}(C_1), \mathrm{PD}(C_2))$ (up to permutation on $\{E_1, E_2\}$) is given by
		\[
			 (E_1, E_2), \quad (E_1, E_3), \quad (E_3, u - E_1 - E_2), \quad (E_1, u - E_2- E_3),
			\quad (u-E_1-E_2, u-E_1- E_3), \quad (u-E_1-E_3, u-E_2- E_3)
		\]
		with the identification $u = x+y$ and $E_3 = y$, or equivalently, in terms of $\{x,y,E_1, E_2\}$, possible candidates for $(\mathrm{PD}(C_1), \mathrm{PD}(C_2))$ are
		\[
			 (E_1, E_2), \quad (E_1, y), \quad (y, x+y  - E_1 - E_2), \quad (E_1, x - E_2),
			\quad (x+y -E_1-E_2, x - E_1), \quad (x -E_1, x -E_2)
		\] 		
		We divide the proof into two cases; {\bf (IV-1-1)}: $(b_{\min}, b_{\max}) = (-1,-1)$ and {\bf (IV-1-2)}: $(b_{\min}, b_{\max}) = (-1,0)$ as listed 
		in \eqref{equation_8_solutions}.
		\vs{0.3cm}

		\noindent
		{\bf (IV-1-1) : $m = 2, \mathrm{Vol}(Z_0) = 2, (b_{\min}, b_{\max}) = (-1,-1)$} \vs{0.3cm}
		
		\noindent 
		Note that there are at most two connected components of $Z_0$ since $\mathrm{Vol}(Z_0) = 2$. 
		Because $\mathrm{Vol}(Z_0) = 2$ and $b_{\max} = -1$, it follows that 
		\begin{equation}\label{equation_IV_1_1}
			\mathrm{Vol}(Z_0) = 2a+b+c+d = 2, \quad \langle e(P_2^-)^2, [M_{2-\epsilon}] \rangle = 1 ~\text{so that $\langle e(P_0^+)^2, [M_0] \rangle = -1$}
		\end{equation}
		by Lemma \ref{lemma_Euler_extremum} and Lemma \ref{lemma_Euler_class}. \vs{0.3cm}
			
		\noindent
		{\bf Case (1) : $(\mathrm{PD}(C_1), \mathrm{PD}(C_2)) = (E_1, E_2)$}  \vs{0.3cm}
		
		\noindent
		In this case, we have $c=d=-2$ by \eqref{equation_vanishing_IV_1}. 
		Also, by \eqref{equation_IV_1_1}, it follows that $2a + b = 6$ and
		\[
			\langle ((a-1)x + (b-1)y + (c+1)E_1 + (d+1)E_2)^2, [M_0] \rangle = 2(a-1)(b-1) - (b-1)^2 - 2 = -1.
		\]
		So, we get $a=2, ~b=2, ~c=d=-2.$, i.e., $\mathrm{PD}(Z_0) = 2x + 2y - 2E_1 - 2E_2$ which implies that $Z_0 \cdot Z_0 = -4.$ 
		Because the number of connected components of $Z_0$ is at most two, there is no such manifold by the adjunction formula \eqref{equation_adjunction} : 
		\[
			[Z_0] \cdot [Z_0] + \sum (2 - 2g_i) = \langle c_1(TM_0), [Z_0] \rangle = 2
		\]
		where the sum is taken over all connected components of $Z_0$
		\vs{0.1cm}

		\noindent		
		{\bf Case (2) : $(\mathrm{PD}(C_1), \mathrm{PD}(C_2)) = (E_1, y)$}  \vs{0.1cm}
		
		\noindent 
		By \eqref{equation_vanishing_IV_1}, we obtain $c = -2$ and $a = b + 1$. Also from \eqref{equation_IV_1_1}, we get
		\[
			b =  1 \hs{0.2cm}(a = 2) \quad \text{and} \quad d = -1, 
		\]
		that is, $\mathrm{PD}(Z_0) = 2x + y - 2E_1 - E_2$ and $[Z_0] \cdot [Z_0] = -2$. The adjunction formula \eqref{equation_adjunction} says that 
		\[
			[Z_0] \cdot [Z_0] + \sum (2 - 2g_i) = \langle c_1(TM_0), [Z_0] \rangle = 2
		\]
		and this implies that $Z_0$ consists of two spheres $Z_0^1$ and $Z_0^2$ (since $Z_0$ consists at most two components) with 
		\begin{equation}\label{equation_IV_1_1_1}
			\mathrm{PD}(Z_0^1) = x + y - E_1 - E_2 \quad \mathrm{PD}(Z_0^2) = x - E_1
		\end{equation}
		up to permutation on $\{E_1, E_2\}$. (Note that this computation can be easily obtained from the fact that each $[Z_0^i]$ is an exceptional class 
		so  that one can apply Lemma \ref{lemma_list_exceptional}.) See Table \ref{table_IV_1} : {\bf (IV-1-1.1)}. 
		
		For the Chern number computation, we apply the localization theorem \ref{theorem_localization} : 
		\begin{equation}\label{equation_Chern_IV_1_1}
			\begin{array}{ccl}\vs{0.3cm}
				\ds \int_M c_1^{S^1}(TM)^3 & = &  \ds  
							\int_{Z_{\min}} \frac{\left(c_1^{S^1}(TM)|_{Z_{\min}}\right)^3}{e_{Z_{\min}}^{S^1}} + 2 \frac{\overbrace{\lambda^3}^{Z_{-1} ~\text{term}}}
							{-\lambda^3}
							+ \int_{Z_0} \frac{\overbrace{\left(c_1^{S^1}(TM)|_{Z_0}\right)^3}^{= 0}}{e_{Z_0}^{S^1}}
							 + 2 \frac{\overbrace{-\lambda^3}^{Z_1 ~\text{term}}}{\lambda^3} + \int_{Z_{\max}} \frac{\left(c_1^{S^1}(TM)|_{Z_{\max}}\right)^3}
							 {e_{Z_{\max}}^{S^1}} \\ \vs{0.2cm}
							& = &  (24 + 4b_{\min}) + (24 + 4b_{\max}) - 4 = 36
			\end{array}			
		\end{equation}
		by Remark \ref{remark_localization_surface}.
		\vs{0.3cm}
		
		\noindent		
		{\bf Case (3) : $(\mathrm{PD}(C_1), \mathrm{PD}(C_2)) = (y, x+y-E_1-E_2)$} \vs{0.1cm}

		\noindent
		From \eqref{equation_vanishing_IV_1} and \eqref{equation_IV_1_1}, we have 
		\[
			a = b+1, \quad a+c+d = 0, \quad 2a + b + c + d = 2  \quad (\Leftrightarrow 3b + c + d = 0).
		\]
		This implies that $a = 3b$ so that $b = \frac{1}{2}$ and it leads to a contradiction. Thus no such manifold exists.
		\vs{0.3cm}
		
		\noindent		
		{\bf Case (4) : $(\mathrm{PD}(C_1), \mathrm{PD}(C_2)) = (E_1, x-E_2)$} \vs{0.1cm}

		\noindent
		We similarly obtain 
		\[
			c = -2, \quad b+d = 1, \quad 2a + b + c + d = 2 \quad (\Leftrightarrow 2a + c = 1).
		\]
		Then we see that $a = \frac{3}{2}$, which is not an integer. Therefore no such manifold exists. 
		
		\vs{0.3cm}
		\noindent		
		{\bf Case (5) : $(\mathrm{PD}(C_1), \mathrm{PD}(C_2)) = (x+y-E_1 - E_2, x-E_1)$} \vs{0.1cm}

		\noindent
		In this case, we have 
		\[
			a+c+d = 0, \quad b+c = 1, \quad 2a+b+c+d=2 \quad (\Leftrightarrow 2a + d = 1), 
		\]
		and 
		\[
			\langle e(P_0^+)^2, [M_0] \rangle = 2(a-1)(b-1) - (b-1)^2 - (c+1)^2 - (d+1)^2 = -1.
		\]
		Those equations have the unique solution $(a,b,c,d) = (1,1,0,-1)$ so that $\mathrm{PD}(Z_0) = x + y - E_2$. In particular, we have 
		$[Z_0] \cdot [Z_0] = 0$ and $Z_0$ is connected, and therefore $Z_0 \cong S^2$ by the adjunction formula \eqref{equation_adjunction}. 
		The Chern number can be obtained in exactly the same way as in \eqref{equation_Chern_IV_1_1}. See Table \ref{table_IV_1} : {\bf (IV-1-1.3)}. 
		(The connectedness of $Z_0$ is proved as follows : if $Z_0^1$ and $Z_0^2$ are connected components of $Z_0$, then 
			\begin{itemize}
				\item $\mathrm{Vol}(Z_0^1) = \mathrm{Vol}(Z_0^2) = 1$, and 
				\item $[Z_0^1] \cdot [Z_0^1] = -1$ and $[Z_0^2] \cdot [Z_0^2] = 1$ since 
				\[
					[Z_0^i] \cdot [Z_0^i] + 2 - 2g_i = 1\quad \text{and} \quad [Z_0^1] \cdot [Z_0^1] + [Z_0^2] \cdot [Z_0^2] = 0.
				\]
			\end{itemize}
		Then $Z_0^1 \cong S^2$ by the adjunction formula \eqref{equation_adjunction} and $\mathrm{PD}(Z_0^1)$ should be on the list in Lemma \ref{lemma_list_exceptional}.
		However, it contradicts that $\mathrm{PD}(Z_0^1) \cdot (x + y -E_2 - \mathrm{PD}(Z_0^1)) = 0$. Therefore $Z_0$ has to be connected.)			
		\vs{0.3cm}
		
		\noindent		
		{\bf Case (6) : $(\mathrm{PD}(C_1), \mathrm{PD}(C_2)) = (x - E_1, x - E_2)$} \vs{0.1cm}		
		
		\noindent 
		Again by \eqref{equation_vanishing_IV_1} and \eqref{equation_IV_1_1} , we get
		\[
			b+c = 1, \quad b+d = 1, \quad 2a + b + c + d = 2 \quad (\Leftrightarrow 2a + d = 2a + c = 1), 
		\]
		and 
		\[
			\langle e(P_0^+)^2, [M_0] \rangle = 2(a-1)(b-1) - (b-1)^2 - (c+1)^2 - (d+1)^2 = -1.
		\]
		Then we get the unique solution $(a,b,c,d) = (1,2,-1,-1)$ so that $\mathrm{PD}(Z_0) = x + 2y - E_1 - E_2$. Moreover, since $[Z_0] \cdot [Z_0] = -2$, the adjunction formula
		 \eqref{equation_adjunction}
		implies that $Z_0$ consists of two spheres $Z_0^1$ and $Z_0^2$ such that $[Z_0^1] \cdot [Z_0^1] = [Z_0^2] \cdot [Z_0^2] = -1$. 
		Applying Lemma \ref{lemma_list_exceptional}, we finally obtain
		\[
			\mathrm{PD}(Z_0^1) = y \quad \text{and} \quad \mathrm{PD}(Z_0^2) = x + y - E_1 - E_2.
		\]
		See Table \ref{table_IV_1} : {\bf (IV-1-1.2)}. \vs{0.5cm}
		
		\noindent
		{\bf (IV-1-2) : $m = 2, \mathrm{Vol}(Z_0) = 1, (b_{\min}, b_{\max}) = (-1,0)$} \vs{0.3cm}		
		
		\noindent 
		In this case, $Z_0$ is connected by the assumption $\mathrm{Vol}(Z_0) = 1$.  
		Together with the condition $b_{\max} = 0$, we have
		\begin{equation}\label{equation_IV_1_2}
			\mathrm{Vol}(Z_0) = 2a+b+c= 1, \quad \langle e(P_2^-)^2, [M_{2-\epsilon}] \rangle = 0 ~\text{so that $\langle e(P_0^+)^2, [M_0] \rangle = -2$}
		\end{equation}
		by Lemma \ref{lemma_Euler_extremum}. The latter equation can be re-written as 
		\begin{equation}\label{equation_IV_1_2_Euler}
			2(a-1)(b-1) - (b-1)^2 - (c+1)^2 - (d+1)^2 = -2.
		\end{equation}
		Using \eqref{equation_vanishing_IV_1}, \eqref{equation_IV_1_2}, \eqref{equation_IV_1_2_Euler}, we analyze each cases as follows:
		\vs{0.3cm}

		\noindent
		{\bf Case (1) : $(\mathrm{PD}(C_1), \mathrm{PD}(C_2)) = (E_1, E_2)$}  \vs{0.1cm}		
		\[
			c=-2, \quad d=-2, \quad 2a+b+c+d = 1 \quad (\Leftrightarrow 2a + b = 5), \quad 2(a-1)(b-1) - (b-1)^2 = 0
		\]
		so that $(a,b,c,d) = (2,1,-2,-2)$, i.e., $\mathrm{PD}(Z_0) = 2x + y  - 2E_1 - 2E_2$. However, since $[Z_0] \cdot [Z_0] = -5$, no such manifold exists by the adjunction 
		formula \eqref{equation_adjunction}.
		
		\vs{0.3cm}
		\noindent
		{\bf Case (2) : $(\mathrm{PD}(C_1), \mathrm{PD}(C_2)) = (E_1, y)$}  \vs{0.1cm}
		\[
			c = -2, \quad a = b+1, \quad 2a+b+c+d = 1 \quad (\Leftrightarrow 3b + d = 1), \quad \underbrace{2(a-1)(b-1) - (b-1)^2 - (d+1)^2}_{= ~-8b^2 + 12b - 5} = -1
		\]
		so that $(a,b,c,d) = (2,1,-2,-2)$, i.e., $\mathrm{PD}(Z_0) = 2x + y  - 2E_1 - 2E_2$. Again by \eqref{equation_adjunction}, no such manifold exists.

		\vs{0.3cm}		
		\noindent
		{\bf Case (3) : $(\mathrm{PD}(C_1), \mathrm{PD}(C_2)) = (y, x+y-E_1 - E_2)$}  \vs{0.1cm}	
		\[
			a = b+1, \quad a+c+d = 0, \quad 2a+b+c+d = 1 \quad (\Leftrightarrow a + b = 1 ~\Leftrightarrow ~b=0, a=1), \quad (c+1)^2 + (d+1)^2 = 1
		\]
		so that $(a,b,c,d) = (1,0,-1,0)$ or $(1,0,0,-1)$ (where they are equal up to permutation on $\{E_1, E_2\}$.) 
		In this case, we have $Z_0 \cong S^2$ by \eqref{equation_adjunction}. See Table \ref{table_IV_1} : {\bf (IV-1-2)}.

		\vs{0.3cm}		
		\noindent
		{\bf Case (4) : $(\mathrm{PD}(C_1), \mathrm{PD}(C_2)) = (E_1, x - E_2)$}  \vs{0.1cm}		
		\[
			c=-2, \quad b+d = 1, \quad 2a+b+c+d = 1 \quad (\Leftrightarrow 2a + c = 0 \Leftrightarrow a=1), \quad (b-1)^2 + (d+1)^2 = 1
		\]
		so that $(a,b,c,d) = (1,1,-2,0)$ or $(1,2,-2,-1)$. In either case, $[Z_0] \cdot [Z_0] < -1$ so that it violates \eqref{equation_adjunction}. 
		Therefore no such manifold exists.

		\vs{0.3cm}		
		\noindent
		{\bf Case (5) : $(\mathrm{PD}(C_1), \mathrm{PD}(C_2)) = (x+y - E_1 - E_2, x - E_1)$}  \vs{0.1cm}		
		\[
			a+c+d = 0, \quad b+c = 1, \quad \underbrace{2a+b+c+d = 1}_{\Leftrightarrow ~a+b = 1, ~2a+d=0}, \quad \underbrace{-2b(b-1) - (b-1)^2 - (2-b)^2 - (2b-1)^2}_{= ~-8b^2 + 12b - 6} = -2
		\]
		and we obtain $(a,b,c,d) = (0,1,0,0)$, i.e., $\mathrm{PD}(Z_0) = y$. However, we can check that a cycle representing $x - E_2$ vanishes on $M_1$ which leads to a contradiction.
		Therefore no such manifold exists.
				
		\vs{0.3cm}		
		\noindent
		{\bf Case (6) : $(\mathrm{PD}(C_1), \mathrm{PD}(C_2)) = (x - E_1, x - E_2)$}  \vs{0.1cm}										
		\[
			b+c=1, \quad b+d = 1, \quad \underbrace{2a+b+c+d = 1}_{\Leftrightarrow 2a + d =0, ~2a + c = 0}, \quad \underbrace{2(a-1)(b-1) - (b-1)^2 - (c+1)^2 - (d+1)^2}_{4a(a-1) -
			4a^2 - (1-2a)^2 - 
			(1-2a)^2} = -2. 
		\]
		So, $(a,b,c,d) = (0,1,0,0)$. 
		Similar to {\bf Case (5)}, a cycle representing $x+y - E_1 -E_2$ vanishes on $M_1$, and therefore no such manifold exists. \vs{0.5cm}

	\end{proof}		
		
\begin{example}[Fano varieties of type {\bf (IV-1)}]\label{example_IV_1} 
	In this example, we illustrate algebraic Fano varieties with holomorphic Hamiltonian $S^1$-action with topological fixed point data given in Theorem \ref{theorem_IV_1}. 

	\begin{enumerate}
	            \item {\bf Case (IV-1-1.1)} \cite[2nd in Section 12.6]{IP} : Let $Y$ be the toric blow-up of $\p^3$ along two disjoint $T^3$-invariant lines where the moment map image
	            			is described in Figure \ref{figure_IV_1_1_1} (see also Figure \ref{figure_III} (b)).
	            			We take two disjoint lines $C_1$ and $C_2$ corresponding to the edges 
	            			\[
	            				e_1 = \overline{(0,1,3) ~(1,0,3)} \quad e_2 = \overline{(0,1,0) ~(1,0,0)}
	            			\]
	            			respectively. Let $M$ be the monotone toric blow-up of $Y$ along $C_1$ and $C_2$ so that the resulting moment polytope 
	            			(with respect to a moment map $\mu : M \rightarrow \R^3$
	            			is illustrated 
	            			on the right of Figure \ref{figure_IV_1_1_1}. Now, we take the circle subgroup of $T^3$ generated by
	            			$\xi = (1,0,1)$. It is straightforward (by calculating the inner product of $\xi$ and each primitive edge vector) that the action is semifree
	            			and the balanced moment map is given by 
	            			\[
	            				\mu_\xi = \langle \mu, \xi \rangle -2.
	            			\] 
	            			Moreover, the fixed point set consists of 
	            			\begin{itemize}
	            				\item $Z_{-2} = \mu^{-1}(\overline{(0,2,0) ~(0,3,0)})$
	            				\item $Z_{-1} = \mu^{-1}(0,3,1) \cup \mu^{-1}(0,1,1)$
	            				\item $Z_{-2} = \mu^{-1}(\overline{(0,2,2) ~(0,1,2)}) \cup \mu^{-1}(\overline{(1,0,1) ~(2,0,0)})$
	            				\item $Z_1 = \mu^{-1}(1,0,2) \cup \mu^{-1}(3,0,0)$
	            				\item $Z_2 = \mu^{-1}(\overline{(2,0,2) ~(3,0,1)})$
	            			\end{itemize}
					Furthremore, the symplectic areas of $Z_{-2}, Z_0^1, Z_0^2,$  and $Z_2$ are all 1 (see \eqref{equation_IV_1_1_1})
					and hence $b_{\min} = b_{\max} = -1$. Thus the fixed point data of $M$ coincides with the one in Table \ref{table_IV_1} {\bf (IV-1-1.1)}.
	            	            
			\begin{figure}[H]
				\scalebox{1}{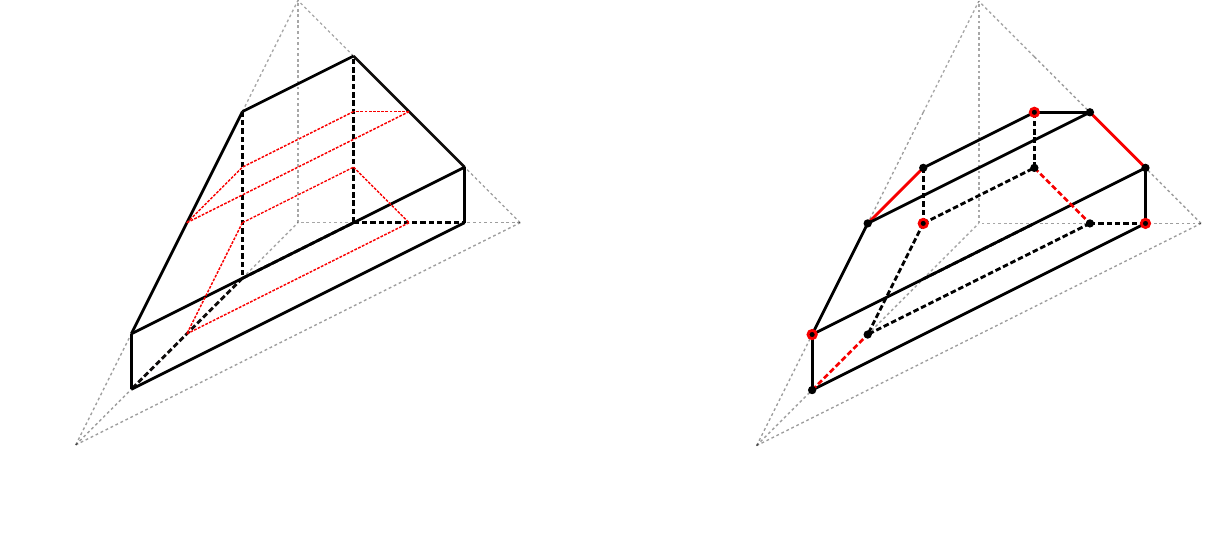}
				\caption{\label{figure_IV_1_1_1} Blow up of $Y$ along two lines $C_1$ and $C_2$ lying on the same exceptional components}
			\end{figure}          	
	            	            
           	 \item {\bf Case (IV-1-1.2)} \cite[3rd in Section 12.6]{IP} : Let $M = \p^1 \times X_3$ where $X_k$ denotes the blow-up of $\p^2$ at $k$ generic points.
           					 In particular we assume that $X_3$ is the toric blow-up of $\p^3$ equipped with the standard toric structure. 
           	 
				           	 Equip $M$ with the monotone toric K\"{a}hler form $\omega$ such that $c_1(TM) = [\omega]$ so that the moment map $\mu : M \rightarrow \R^3$ 
				           	 has the image given in Figure \ref{figure_IV_1_1_2}. Take $\xi = (0,-1,1)$. Then the $S^1$-action generated by $\xi$ is semifree and the 
				           	 balanced moment map is given by $\mu_\xi = \langle \mu, \xi \rangle$. The fixed point set consists of 
	            				\begin{itemize}
	            					\item $Z_{-2} = \mu^{-1}(\overline{(0,2,0) ~(1,2,0)})$
		            				\item $Z_{-1} = \mu^{-1}(0,1,0) \cup \mu^{-1}(2,1,0)$
		            				\item $Z_{-2} = \mu^{-1}(\overline{(0,2,2) ~(1,2,2)}) \cup \mu^{-1}(\overline{(1,0,0) ~(2,0,0)})$
	           	 				\item $Z_1 = \mu^{-1}(0,1,2) \cup \mu^{-1}(2,1,2)$
	            					\item $Z_2 = \mu^{-1}(\overline{(2,0,2) ~(1,0,2)})$
	            				\end{itemize}
						It is not hard to check that the fixed point data of $M$ coincides with the one in Table \ref{table_IV_1} {\bf (IV-1-1.2)}.
						
			\begin{figure}[H]
				\scalebox{1}{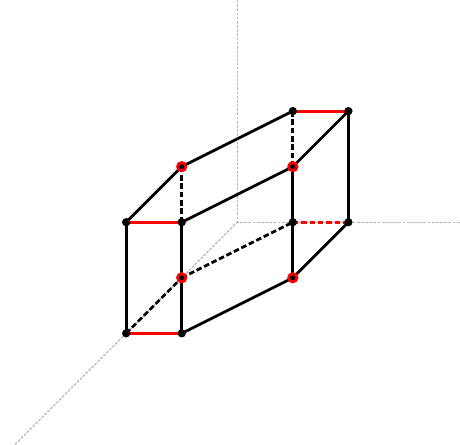}
				\caption{\label{figure_IV_1_1_2} $\p^1 \times X_3$}
			\end{figure}          	

           	 \item {\bf Case (IV-1-1.3)} \cite[7th in Section 12.5]{IP} : 
           	 Let $(W, \omega)$ be the monotone complete flag variety 
           	 given in Example \ref{example_II_1} (1) equipped with the Hamiltonian $T^2$-action where the moment polytope
           	 is described on the left of Figure \ref{figure_IV_1_1_3}. 
           	 
           	 Consider two edges $A$ and $B$ indicated in Figure \ref{figure_IV_1_1_3} and 
           	 denote by $C_A$ and $C_B$ the corresponding $T^2$-invariant spheres, respectively. (Note that $C_A$ and $C_B$ are curves of bidegree $(1,0)$ and $(0,1)$
           	 with respect to the Pl\"{u}cker embedding $W \subset \p^2 \times \p^2$.) 
      		 Using local toric structures on the normal bundles of $C_A$ and $C_B$, respectively, we may take $T^2$-equivariant blow up of $W$ along $C_A$ and $C_B$
      		 and denote the resulting manifold by $M$ and the image of the moment map $\mu : M \rightarrow \R^2$
      		  is given on the right of Figure \ref{figure_IV_1_1_3} (with respect to the monotone 
		 K\"{a}hler form). 
      		            	 
			\begin{figure}[H]
				\scalebox{1}{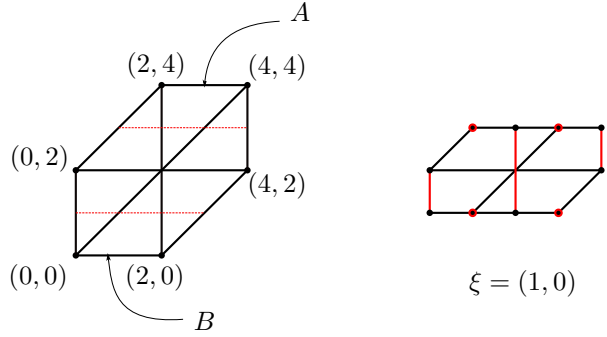}
				\caption{\label{figure_IV_1_1_3} Blow up of $W$ along two disjoint curves of bidegree $(1,0)$ and $(0,1)$. }
			\end{figure}          	

      		 Take the circle subgroup $S^1$ generated by $\xi = (1,0)$. Then the $S^1$-action is semifree and the balanced moment map is given by 
      		 $\mu_\xi = \langle \mu, \xi \rangle - 2$. The fixed point set consists of       		 

		\begin{itemize}
			\item $Z_{-2} = \mu^{-1}(\overline{(0,1) ~(0,2)})$
			\item $Z_{-1} = \mu^{-1}(1,1) \cup \mu^{-1}(1,3)$
			\item $Z_{-2} = \mu^{-1}(\overline{(2,1) ~(2,3)})$
			\item $Z_1 = \mu^{-1}(3,1) \cup \mu^{-1}(3,3)$
			\item $Z_2 = \mu^{-1}(\overline{(4,2) ~(4,3)})$
		\end{itemize}
		and we can easily check that this should coincide with {\bf (IV-1-1.3)} in Table \ref{table_IV_1}.
           	(Note that the symplectic area of $Z_{-2}$ and $Z_2$ are both 1 so that $b_{\min} = b_{\max} = -1$.) \vs{0.5cm}

           	 \item {\bf Case (IV-1-2)} \cite[9th in Section 12.5]{IP} : Let Y be the toric blow-up of $\p^3$ along two disjoint $T^3$-invariant lines where the moment map 
           	 image is given on the left of Figure \ref{figure_IV_1_2} (see also Figure \ref{figure_III} (b)). Let $M$ be a toric blow up of $Y$ along a $T$-invariant exceptional line 
           	 (corresponding to the edge $A$ in Figure \ref{figure_IV_1_2}). With respect to the $T^3$-invariant monotone K\"{a}hler form, the image of a moment map $\mu$ is described 
           	 on the right of Figure \ref{figure_IV_1_2}. 
           	 
			\begin{figure}[H]
				\scalebox{1}{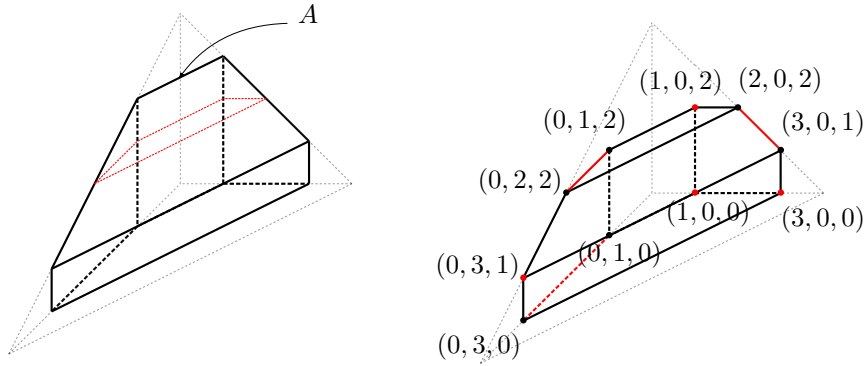}
				\caption{\label{figure_IV_1_2} Blow up of $Y$ along an exceptional line on $Y$. }
			\end{figure}          	
			\vs{-0.5cm}
		\noindent
		Take the circle subgroup $S^1$ of $T^3$ generated by $\xi = (-1, 0, -1)$. Then it is easy to check that the $S^1$-action is semifree and has the balanced moment map
		given by $\mu_\xi = \langle \mu, \xi \rangle + 2$. Also, the fixed point set consists of 

		\begin{itemize}
	            				\item $Z_{-2} = \mu^{-1}(\overline{(2,0,2) ~(3,0,1)})$
	            				\item $Z_{-1} = \mu^{-1}(1,0,2) \cup \mu^{-1}(3,0,0)$
	            				\item $Z_{-2} = \mu^{-1}(\overline{(0,1,2) ~(0,2,2)})$
	            				\item $Z_1 = \mu^{-1}(0,3,1) \cup \mu^{-1}(1,0,0)$
	            				\item $Z_2 = \mu^{-1}(\overline{(0,1,0) ~(0,3,0)})$
		\end{itemize}
		where $\mathrm{Area}(Z_{-2}) = \mathrm{Area}(Z_{0}) = 1$ and $\mathrm{Area}(Z_{2}) = 2$. Thus one can see that the fixed point data of $M$ 
		coincides with {\bf (IV-1-2)} in Table \ref{table_IV_1}. 
          \end{enumerate}		
          
\end{example}

	\begin{theorem}\label{theorem_IV_2}
		Let $(M,\omega)$ be a six-dimensional closed monotone semifree Hamiltonian $S^1$-manifold with $c_1(TM) = [\omega]$. Suppose that $\mathrm{Crit} H = \{ 2, -1, 0, 1, -2\}$. 
		If the number of fixed points of index two equals one, up to orientation of $M$, the list of all possible topological fixed point data is given in the Table \ref{table_IV_2}
		\begin{table}[h]
			\begin{tabular}{|c|c|c|c|c|c|c|c|c|c|}
				\hline
				    & $(M_0, [\omega_0])$ & $e(P_{-2}^+)$ &$Z_{-2}$  & $Z_{-1}$ & $Z_0$ & $Z_1$ & $Z_2$ & $b_2(M)$ & $c_1^3(M)$ \\ \hline \hline
				    {\bf (IV-2-1.1)} & \makecell{$(E_{S^2} \# ~\overline{\p^2},$ \\$3x + 2y - E_1)$} & $-x-y$  &$S^2$ & {\em pt} &
				    		\makecell{ $Z_0 \cong S^2$  \\ $\mathrm{PD}(Z_0) = 2x + y - E_1$}				    
				     &{\em pt} & $S^2$ & $3$ &$38$ \\ \hline
				    {\bf (IV-2-1.2)} & \makecell{$(E_{S^2} \# ~\overline{\p^2},$ \\$3x + 2y - E_1)$} & $-x-y$  &$S^2$ & {\em pt} &
				    		\makecell{ $Z_0 = Z_0^1 ~\dot \cup ~ Z_0^2$ \\ $Z_0^1 \cong Z_0^2 \cong S^2$ \\ 
				    		$\mathrm{PD}(Z_0^1) = \mathrm{PD}(Z_0^2) = x + y - E_1$}
				     &{\em pt} & $S^2$ & $4$ &$38$ \\ \hline
				    {\bf (IV-2-2.1)} & \makecell{$(E_{S^2} \# ~\overline{\p^2},$ \\$3x + 2y - E_1)$} & $-x-y$  &$S^2$ & {\em pt} & 
				    		\makecell{ $Z_0 \cong S^2$  \\ $\mathrm{PD}(Z_0) = x + y$}						    
				    &{\em pt} & $S^2$ & $3$ &$42$ \\ \hline
				    {\bf (IV-2-2.2)} & \makecell{$(E_{S^2} \# ~\overline{\p^2},$ \\$3x + 2y - E_1)$} & $-x-y$  &$S^2$ & {\em pt} & 
				    		\makecell{ $Z_0 = Z_0^1 ~\dot \cup ~ Z_0^2$ \\ $Z_0^1 \cong Z_0^2 \cong S^2$ \\ 
				    		$\mathrm{PD}(Z_0^1) = y$ \\ $\mathrm{PD}(Z_0^2)= x + y - E_1$}				    				    
				    &{\em pt} & $S^2$ & $4$ &$42$ \\ \hline				    
				    {\bf (IV-2-3)} & \makecell{$(E_{S^2} \# ~\overline{\p^2},$ \\$3x + 2y - E_1)$} & $-x-y$  &$S^2$ & {\em pt} &
				    		\makecell{ $Z_0 \cong S^2$  \\ $\mathrm{PD}(Z_0) = x$}				    
				     &{\em pt} & $S^2$ & $3$ &$46$ \\ \hline
				    {\bf (IV-2-4)} & \makecell{$(E_{S^2} \# ~\overline{\p^2},$ \\$3x + 2y - E_1)$} & $-x-y$  &$S^2$ & {\em pt} &
				    		\makecell{ $Z_0 \cong S^2$ 
				    		 \\ $\mathrm{PD}(Z_0) = E_1$}			    
				     &{\em pt} & $S^2$ & $3$ &$50$ \\ \hline
				    {\bf (IV-2-5)} & \makecell{$(S^2 \times S^2  \# ~\overline{\p^2},$ \\$2x + 2y - E_1)$} & $-y$  &$S^2$ & {\em pt} & 
				    		\makecell{ $Z_0 = Z_0^1 \dot \cup Z_0^2$ \\ $Z_0^1 \cong Z_0^2 \cong S^2$ \\ 
				    		$\mathrm{PD}(Z_0^1) = x - E_1$ \\  $\mathrm{PD}(Z_0^2) = y - E_1$  \\ }				    
				    &{\em pt} & $S^2$ & $4$ &$46$ \\ \hline
				    {\bf (IV-2-6)} & \makecell{$(S^2 \times S^2  \# ~\overline{\p^2},$ \\$2x + 2y - E_1)$} & $-y$  &$S^2$ & {\em pt} &
				    		\makecell{ $Z_0 \cong S^2$  \\ $\mathrm{PD}(Z_0) = x - E_1$}					    
				     &{\em pt} & $S^2$ & $3$ &$50$ \\ \hline
			\end{tabular}		
			\vs{0.5cm}			
			\caption{\label{table_IV_2} Topological fixed point data for $\mathrm{Crit} H = \{-2, -1,0,1, 2\}$ with $|Z_{-1}| = 1$.}
		\end{table}				   
	\end{theorem}
	
	\begin{proof}
	
		As we have seen in \eqref{equation_8_solutions}, $b_{\min}$ is either  $-1$ or $0$.  For each cases, we have 
		\begin{equation}\label{equation_bmin_IV_2}
			\begin{cases}
				M_{-2 + \epsilon} \cong E_{S^2}, \quad c_1(TM_0) = [\omega_0] = 3x + 2y - E_1, \quad e(P_{-2}^+) = kx - y = -x -y & \text{if $b_{\min} = -1$} \\ \vs{0.1cm}
				M_{-2 + \epsilon} \cong S^2 \times S^2, \quad c_1(TM_0) = [\omega_0] = 2x + 2y - E_1, \quad e(P_{-2}^+) = kx - y = -y & \text{if $b_{\min} = 0$} 
			\end{cases}
		\end{equation}
		by Lemma \ref{lemma_Euler_extremum}, where $M_0$ is a one point blow-up of $M_{-2 + \epsilon}$ and $E_1$ is the dual class of the 
		exceptional divisor on $M_0$.
		
		Let $\mathrm{PD}(Z_0) = ax + by + cE_1$ for some $a,b,c \in \Z$. By the Duistermaat-Heckman theorem \eqref{equation_DH}, we have 
		\[
				[\omega_1] = [\omega_0] - e(P_0^+) = \begin{cases}
						(4-a)x + (3-b)y - (2+c)E_1 & \text{if $b_{\min} = -1$}\\ \vs{0.1cm}
						(2-a)x + (3-b)y - (2+c)E_1 & \text{if $b_{\min} = 0$}.				
				\end{cases} 
		\] 
		Moreover, only one blow-down occurs at $M_1$ with the vanishing cycle $C$ so that 
		\begin{equation}\label{equation_vanishing_IV_2}
			\langle [\omega_1], C \rangle = 0. 
		\end{equation}
		By Lemma \ref{lemma_list_exceptional}, 
		the list of all possible $\mathrm{PD}(C)$ is given by
		\[
			u - E_1 - E_2, \quad E_1, \quad E_2
		\]
		or equivalently, in terms of $\{x,y,E_1\}$, 
		\begin{itemize}
			\item if $b_{\min} = -1$, then 
				\[
					x - E_1, \quad E_1, \quad y.
				\] 		
			\item if $b_{\min} = 0$, then 
				\[
					E_1, \quad x - E_1, \quad y - E_1.
				\]
		\end{itemize}
		
		Now we compute the fixed point data for remaining six cases (on the right of \eqref{equation_8_solutions}) as follows.
		(Note that the Chern number computation can be easily obtained from the localization theorem \ref{theorem_localization} and Remark \ref{remark_localization_surface}.)
		
		\vs{0.3cm}
				
		\noindent
		{\bf (IV-2-1) : $m = 1, \mathrm{Vol}(Z_0) = 4, (b_{\min}, b_{\max}) = (-1,-1)$} \vs{0.3cm}				
		
		\noindent
		Because $\mathrm{Vol}(Z_0) = 4$ and $b_{\max} = -1$, it follows that 
		\begin{equation}\label{equation_IV_2_1}
			\mathrm{Vol}(Z_0) = 2a+b+c= 4, \quad \langle e(P_2^-)^2, [M_{2-\epsilon}] \rangle = 1 ~\text{so that $\langle e(P_0^+)^2, [M_0] \rangle = 0$}
		\end{equation}
		by Lemma \ref{lemma_Euler_extremum}. The latter equation can be re-written as
		\[
			2(a-1)(b-1) - (b-1)^2 - (c+1)^2 = 0 \quad \quad \text{as \quad $e(P_0^+) = (a-1)x + (b-1)y + (c+1)E_1$.}
		\]
		\vs{0.1cm}
		
		\noindent
		{\bf Case (1) :} $\mathrm{PD}(C) = x - E_1$. \vs{0.1cm}
		
		\noindent
		Since $b+c = 1$ by \eqref{equation_vanishing_IV_2}, we have $2a=3$ by \eqref{equation_IV_2_1}, and hence no such manifold exists. \vs{0.3cm}
		
		\noindent
		{\bf Case (2) :} $\mathrm{PD}(C) = E_1$.  \vs{0.1cm}

		\noindent
		In this case, we have $c = -2$ by \eqref{equation_vanishing_IV_2}. Then \eqref{equation_IV_2_1} implies that 
		\[
			2a + b = 6, \quad 2(a-1)(b-1) - (b-1)^2 = (b-1)(2a - b -1) = 1
		\]
		which has the unique integeral solution $(a,b,c) = (2,2,-2)$. So, $\mathrm{PD}(Z_0) = 2x + 2y - 2E_1$ and $[Z_0] \cdot [Z_0] = 0$.
		Then the adjunction formula \eqref{equation_adjunction} implies that 
		\[
			[Z_0] \cdot [Z_0] + \sum (2 - 2g_i) = 4 \quad \text{(sum is taken over connected components of $Z_0$)}. 
		\]
		Thus there are at least two spheres, namely $Z_0^1$ and $Z_0^2$. Moreover, they satisfy (again by \eqref{equation_adjunction})
		\[
			[Z_0^1] \cdot [Z_0^1] \geq -1 \quad \text{and } \quad [Z_0^2] \cdot [Z_0^2] \geq -1.
		\]
		Note that if $[Z_0^i] \cdot [Z_0^i] = -1$, then $([Z_0] - [Z_0^i]) \cdot [Z_0^i] \neq 0$ by Lemma \ref{lemma_list_exceptional}.  So, 
		\[
			[Z_0^1] \cdot [Z_0^1] \geq 0  \quad \text{and } \quad [Z_0^2] \cdot [Z_0^2] \geq 0.
		\]		
		In particular, we have $\mathrm{Vol}(Z_0^i) = [Z_0^i] \cdot [Z_0^i] + 2 \geq 2$ so that the only possibility is that 
		\[
			[Z_0^i] \cdot [Z_0^i] = 0, \quad i=1,2.
		\]
		One can easily see that  $\mathrm{PD} (Z_0^1) = \mathrm{PD}(Z_0^2) = x + y - E_1$. See Table \ref{table_IV_2} : {\bf (IV-2-1.2)}. 
		
		\vs{0.3cm}
		
		\noindent
		{\bf Case (3) :} $\mathrm{PD}(C) = y$.  \vs{0.1cm}
		
		\noindent
		From \eqref{equation_vanishing_IV_2}, we get $a = b + 1$. Then, by \eqref{equation_IV_2_1}, 
		\[
			3b + c = 2, \quad 2b(b-1) - (b-1)^2 - (c+1)^2 = 0, 
		\]
		whose solution is $(a,b,c) = (2, 1, -1)$, that is, $\mathrm{PD}(Z_0) = 2x + y  - E_1$ (and so $[Z_0] \cdot [Z_0] = 2$). Then the adjunction formula 
		\[
			[Z_0] \cdot [Z_0] + \sum (2 - 2g_i) = 4 
		\]
		implies that 
		there exists a sphere component, say $Z_0^1$, of $Z_0$. If we denote by $\mathrm{PD}(Z_0^1) = \alpha x + \beta y + \gamma E_1$, it satisfies 
		\[
			 2\alpha\beta - \beta^2 - \gamma^2 + 2 = [Z_0^1] \cdot [Z_0^1] + 2 = \langle c_1(TM_0), [Z_0^1] \rangle = 2\alpha + \beta + \gamma.
		\]
		Also, since $([Z_0] - [Z_0^1]) \cdot [Z_0^1] = 0$, 
		\[
			\left( (2 - \alpha)x + (1 - \beta)y - (1+\gamma)E_1) \right) \cdot (\alpha x + \beta y + \gamma E_1) = -2\alpha\beta + \alpha + \beta + \gamma + \beta^2 + \gamma^2 = 0.
		\]
		Combining those two equations above, we get $\alpha = 2$ and 
		\[
			\beta^2 + \gamma^2 - 3\beta + \gamma + 2 = 0 \quad \Leftrightarrow \quad (\beta - \frac{3}{2})^2 + (\gamma + \frac{1}{2})^2 - \frac{1}{2} = 0.
		\]
		Therefore, $(\beta, \gamma) = (2, 0), (2, -1), (1, 0), (1, -1)$. In any case, $\mathrm{Vol}(Z_0^1) \geq 4$ which is impossible unless $Z_0^1 = Z_0$. 
		This implies that $Z_0$ is connected and is a sphere. See {\bf (IV-2-1.1)}. 
		
		 \vs{0.3cm}
			
		\noindent
		{\bf (IV-2-2) : $m = 1, \mathrm{Vol}(Z_0) = 3, (b_{\min}, b_{\max}) = (-1,0)$} \vs{0.3cm}				
		
		\noindent
		By Lemma \ref{lemma_Euler_extremum}, it follows that 
		\begin{equation}\label{equation_IV_2_2}
			\mathrm{Vol}(Z_0) = 2a+b+c= 3, \quad \langle e(P_2^-)^2, [M_{2-\epsilon}] \rangle = 0 ~\text{so that $\langle e(P_0^+)^2, [M_0] \rangle = -1$}
		\end{equation}
		where the latter equation is equivalent to 
		\[
			2(a-1)(b-1) - (b-1)^2 - (c+1)^2 = -1.
		\]
		\vs{0.1cm}
		
		\noindent
		{\bf Case (1) :} $\mathrm{PD}(C) = x - E_1$. \vs{0.1cm}
		
		\noindent
		By \eqref{equation_vanishing_IV_2}, we have $b+c = 1$ so that $a = 1$ and $(b-1)^2 + (c+1)^2 = 1$ (and so $(b,c) = (1, 0)$ or $(2, -1)$). \vs{0.1cm}
		\begin{itemize}
			\item If $(a,b,c) = (1,1,0)$, then $\mathrm{PD}(Z_0) = x + y$ and $[Z_0] \cdot [Z_0] = 1$ so that there exists at least one sphere component, denote by $Z_0^1$,
			 in $Z_0$. 

				Suppose that $Z_0$ is not connected. Then $\mathrm{Vol}(Z_0^1) = 1$ or $2$. If $\mathrm{Vol}(Z_0^1) = 1$, then $[Z_0^1] \cdot [Z_0^1] = -1$
				by the adjunction formula, and hence $\mathrm{PD}(Z_0^1) = E_1, y, x - E_1$ by Lemma \eqref{lemma_list_exceptional}.
				In either case, it follows that 
				\[
					[Z_0^1] \cdot ([Z_0] - [Z_0^1]) \neq 0
				\]
				which leads to a contradiction. So, $\mathrm{Vol}(Z_0^1) \neq 1$.
				
				On the other hand, if $\mathrm{Vol}(Z_0^1) = 2$, then $[Z_0^1] \cdot [Z_0^1] = 0$ by the adjunction formula. If we let 
				$\mathrm{PD}(Z_0^1) = \alpha x + \beta y + \gamma E_1$, then 
				\begin{itemize}
					\item $2\alpha\beta - \beta^2 - \gamma^2 = 0$, \quad ($\because ~[Z_0^1]\cdot [Z_0^1] = 0$), 
					\item $\alpha - 2\alpha\beta + \beta^2 + \gamma^2 = 0$, \quad ($\because ~[Z_0^1] \cdot ([Z_0] - [Z_0^1]) = 0$), 
					\item $2\alpha + \beta + \gamma = 2$ \quad ($\because ~\mathrm{Vol}(Z_0^1) = 2$)
				\end{itemize}
				whose (real) solution does not exist. Thus $Z_0$ is connected and we have $Z_0 \cong S^2$. See Table \ref{table_IV_2}: {\bf (IV-2-2.1)}.\vs{0.2cm}
			\item If $(a,b,c) = (1, 2, -1)$, i.e., $\mathrm{PD}(Z_0) = x + 2y - E_1$, then we have $[Z_0] \cdot [Z_0] = -1$ and there are at least two 
			sphere components $Z_0^1$ and $Z_0^2$ in $Z_0$ by the adjunction formula. Since $\mathrm{Vol}(Z_0^1) + \mathrm{Vol}(Z_0^2) \leq 3$,
			we may assume that $\mathrm{PD}(Z_0^1) = 1$ (so that $[Z_0^1] \cdot [Z_0^1] = -1$).
			Then we obtain $\mathrm{PD}(Z_0^1) = y$ by the fact that $([Z_0] - [Z_0^1]) \cdot [Z_0^1] = 0$ and Lemma \ref{lemma_list_exceptional}. So, 
			\[
				Z_0^1 \cong S^2 ~(\mathrm{PD}(Z_0^1) = y) \quad \text{and} \quad Z_0^2 \cong S^2 ~(\mathrm{PD}(Z_0^2) = x + y - E_1)
			\]
			See Table \ref{table_IV_2}: {\bf (IV-2-2.2)}. (Note that $\mathrm{Vol}(Z_0^2) \neq 1$ otherwise $\mathrm{PD}(Z_0^2)$ also should be $y$ which contradicts 
			that $[Z_0^1] \cdot [Z_0^2] = 0$.)
		\end{itemize}
		
		\vs{0.3cm}
		
		\noindent
		{\bf Case (2) :} $\mathrm{PD}(C) = E_1$.  \vs{0.1cm}

		\noindent
		Since $c=-2$ by \eqref{equation_vanishing_IV_2}, we have 
		\[
			2a + b = 5 \quad \text{and} \quad 2(a-1)(b-1) - (b-1)^2 = 0
		\]
		where it has a unique integral solution $(a,b,c) = (2,1,-2)$. However, since
		\[
			[\omega_1] \cdot y = (2x + 2y) \cdot y = 0, 
		\]
		the exceptional divisor representing $y$ vanishes at $M_1$, i.e., two simultaneous blow-downs occur at $M_1$. 
		Thus no such manifold exists.
		
%
		\vs{0.3cm}
		
		\noindent
		{\bf Case (3) :} $\mathrm{PD}(C) = y$.  \vs{0.1cm}
		
		\noindent
		Now we have $a = b+1$ and so 
		\[
			3b+c = 1 \quad \text{and} \quad 2b(b-1) - (b-1)^2 - (c+1)^2 = -1
		\]
		by \eqref{equation_IV_2_2}. This has a unique integral solution $(a,b,c) = (2,1,-2)$. This case is exactly 
		the same as in {\bf Case (2)} above and we have $[\omega_1] \cdot E_1 = 0$. 
		Then two simultaneous blow-downs occur at $M_1$ which is impossible. 
		Therefore there is no such manifold.

		 \vs{0.3cm}

		\noindent
		{\bf (IV-2-3) : $m = 1, \mathrm{Vol}(Z_0) = 2, (b_{\min}, b_{\max}) = (-1,1)$} \vs{0.3cm}				

		\noindent
		In  this case, we have 
		\begin{equation}\label{equation_IV_2_3}
			\mathrm{Vol}(Z_0) = 2a+b+c= 2, \quad \langle e(P_2^-)^2, [M_{2-\epsilon}] \rangle = -1 ~\text{so that $\langle e(P_0^+)^2, [M_0] \rangle = -2$}
		\end{equation}
		where the latter one is 
		\[
			2(a-1)(b-1) - (b-1)^2 - (c+1)^2 = -2.
		\]
		\vs{0.1cm}
		
		\noindent
		{\bf Case (1) :} $\mathrm{PD}(C) = x - E_1$. \vs{0.1cm}
		
		\noindent
		Using $b+c = 1$ by \eqref{equation_vanishing_IV_2}, we have $a = \frac{1}{2}$. Thus no such manifold exists.
		
		\vs{0.3cm}
		
		\noindent
		{\bf Case (2) :} $\mathrm{PD}(C) = E_1$.  \vs{0.1cm}

		\noindent
		Substituting $c = -2$, we have 
		\[
			2a + b = 4, \quad 2(a-1)(b-1) - (b-1)^2 = -1
		\]
		and therefore the only possible solution is $(a,b) = (1,2)$, i.e., $\mathrm{PD}(Z_0) = x + 2y - 2E_1$. However, the adjunction formula \eqref{equation_adjunction}
		implies that 
		\[
			[Z_0] \cdot [Z_0] + \sum (2 - 2g_i) = -4 + \sum (2 - 2g_i) = 2, 
		\]
		i.e., there are three sphere components $Z_0^1, Z_0^2, Z_0^3$ and hence $\mathrm{Vol}(Z_0) \geq 3$ which leads to a contradiction.
		So, no such manifold exists.

		\vs{0.3cm}
		
		\noindent
		{\bf Case (3) :} $\mathrm{PD}(C) = y$.  \vs{0.1cm}
		
		\noindent
		In this case, $a = b+1$ so that 
		\[
			3b + c = 0, \quad 2b(b-1) - (b-1)^2 - (c+1)^2 = -2
		\]
		and it has a unique solution $(a,b,c) = (1,0,0)$. If $Z_0$ is not connected, then the adjunction formula implies that $Z_0$ consists of two spheres $Z_0^1$ and $Z_0^2$
		each of which has symplectic area $1$ (so that it is an exceptional sphere). On the other hand, by the fact that $[Z_0^1] \cdot [Z_0^2] = 0$ and Lemma \ref{lemma_list_exceptional}
		imply that the dual classes of $Z_0^1$ and $Z_0^2$ are 
		$y$ and $E_1$, respectively. Then it follows that $\mathrm{PD}(Z_0) = x \neq \mathrm{PD}(Z_0^1) + \mathrm{PD}(Z_0^2)$. So, $Z_0$ is connected and 
		\[
			Z_0 \cong S^2, \quad \mathrm{PD}(Z_0) = x.
		\]
		See Table \ref{table_IV_2}: {\bf (IV-2-3)}.
		 \vs{0.3cm}

		\noindent
		{\bf (IV-2-4) : $m = 1, \mathrm{Vol}(Z_0) = 1, (b_{\min}, b_{\max}) = (-1,2)$} \vs{0.3cm}				

		\noindent
		As $\mathrm{Vol}(Z_0) = 1$, $Z_0$ is connected. Also, 
		\begin{equation}\label{equation_IV_2_4}
			\mathrm{Vol}(Z_0) = 2a+b+c= 1, \quad \langle e(P_2^-)^2, [M_{2-\epsilon}] \rangle = -2 ~\text{so that $\langle e(P_0^+)^2, [M_0] \rangle = -3$}
		\end{equation}
		i.e., 
		\[
			2(a-1)(b-1) - (b-1)^2 - (c+1)^2 = -3.
		\]
		\vs{0.1cm}
		
		\noindent
		{\bf Case (1) :} $\mathrm{PD}(C) = x - E_1$. \vs{0.1cm}
		
		\noindent
		We have $b+c = 1$ so that $(a,b,c) = (0, 2,-1)$ or $(0,0,1)$. If $(a,b,c) = (0, 2,-1)$, then $\mathrm{PD}(Z_0) = 2y - E_1$ and $[Z_0] \cdot [Z_0] = -5$. 
		This is impossible by the adjunction formula since $Z_0$ is connected. So, no such manifold exists. On the other hand, if $(a,b,c) = (0,0,1)$, i.e., $\mathrm{PD}(Z_0) = E_1$, 
		then we have 
		\[
			Z_0 \cong S^2, \quad \mathrm{PD}(Z_0) = E_1. 
		\]
		See Table \ref{table_IV_2}: {\bf (IV-2-4)}.
		
		\vs{0.3cm}
		
		\noindent
		{\bf Case (2) :} $\mathrm{PD}(C) = E_1$.  \vs{0.1cm}

		\noindent
		Now, we have $c = -2$ and \eqref{equation_IV_2_4} implies that 
		\[
			2a + b = 3, \quad 2(a-1)(b-1) - (b-1)^2 = -2
		\]		
		which has no integral solution. Thus there is no such manifold.
		\vs{0.3cm}
		
		\noindent
		{\bf Case (3) :} $\mathrm{PD}(C) = y$.  \vs{0.1cm}
		
		\noindent
		It follows that $a = b+1$, and we obtain
		\[
			3b + c = -1, \quad 2b(b-1) - (b-1)^2 - (c+1)^2 = -3
		\]
		where no integral solution exists. Thus there is no such manifold.
		 \vs{0.3cm}
		
		\noindent
		{\bf (IV-2-5) : $m = 1, \mathrm{Vol}(Z_0) = 2, (b_{\min}, b_{\max}) = (0,0)$} \vs{0.3cm}				

		\noindent
		Since $b_{\min} = 0$, we have $M_{-2 + \epsilon} \cong S^2 \times S^2$ and so $e(P)_{-2}^+ = -y$ and $c_1(TM_0) = 2x + 2y - E_1$, see \eqref{equation_bmin_IV_2}.  
		Also, Lemma \ref{lemma_Euler_extremum} implies that 
		\begin{equation}\label{equation_IV_2_5}
			\mathrm{Vol}(Z_0) = 2a+2b+c= 2, \quad \langle e(P_2^-)^2, [M_{2-\epsilon}] \rangle = 0 ~\text{so that $\langle e(P_0^+)^2, [M_0] \rangle = -1$}
		\end{equation}
		where the latter equation can be re-written by 
		\[
			2a(b-1) - (c+1)^2 = -1.
		\]
		Note that if $Z_0$ is connected, then $[Z_0] \cdot [Z_0] = 0$ by the adjunction formula.
		Also, if $Z_0$ is disconnected with two components $Z_0^1$ and $Z_0^2$ such that $\mathrm{Vol}(Z_0^1) = \mathrm{Vol}(Z_0^2) = 1$, then 
		the adjunction formula implies that $[Z_0^1] \cdot [Z_0^1] = [Z_0^2] \cdot [Z_0^2] = -1$. In particular, $[Z_0] \cdot [Z_0] = -2$.
		
		Recall that a possible dual class of the cycle $C$ vanishing at the reduced space $M_1$ is $x - E_1$, $E_1$, or $y - E_1$ by Lemma \ref{lemma_list_exceptional}.
		\vs{0.1cm}
		
		\noindent
		{\bf Case (1) :} $\mathrm{PD}(C) = x - E_1$. \vs{0.1cm}
		
		\noindent
		By \eqref{equation_vanishing_IV_2}, we have $b+c = 1$ so that 
		\[
			2a -c  = 0, \quad -2ac - (c+1)^2 = -1
		\]
		where it has a unique integral solution $(a,b,c) = (0,1,0)$. However, in this case, a cycle representing $y - E_1$ is also vanishing at $M_1$. In other words, 
		two blow-downs occur at $M_1$. So, no such manifold exists.
		\vs{0.3cm}
		
		\noindent
		{\bf Case (2) :} $\mathrm{PD}(C) = E_1$.

		\noindent
		In this case, $c = -2$ so that 
		\[
			a+b = 2, \quad 2a(b-1) = 0
		\]
		where the solution is $(a,b,c) = (0,2,-2)$ or $(1,1,-2)$. If $(a,b,c) = (0,2,-2)$, then $[Z_0] \cdot [Z_0] = -4$ so that there are at least three spheres in $Z_0$ by the adjunction formula, 
		which is impossible since 
		$\mathrm{Vol}(Z_0) = 2$. Thus there is no such manifold. 
		
		If $(a,b,c) = (1,1,-2)$, then $[Z_0] \cdot [Z_0] = -2$ and so $Z_0$ consists of two spheres, say $Z_0^1$ and $Z_0^2$, each of which has self-intersection number $-1$ by the 
		adjunction formula. 
		By Lemma \ref{lemma_list_exceptional}, we get 
		\[
			Z_0^1 \cong Z_0^2 \cong S^2, \quad \mathrm{PD}(Z_0^1) = x - E_1, \quad \mathrm{PD}(Z_0^2) = y - E_1.
		\]
		See Table \ref{table_IV_2}: {\bf (IV-2-5)}.
		\vs{0.3cm}
		
		\noindent
		{\bf Case (3) :} $\mathrm{PD}(C) = y - E_1$.
		
		\noindent
		From \eqref{equation_vanishing_IV_2}, we have $a + c = 0$ and so 
		\[
			a + 2b = 2, \quad 2a(b-1) - (1-a)^2 = -1
		\]
		and it has the unique solution $(a,b,c) = (0, 1, 0)$. Similar to {\bf Case (1)}, a cycle representing $x - E_1$ also vanishes at $M_1$ so that two blow-downs occur simultaneously 
		at $M_1$. Therefore there is no such manifold.
		 \vs{0.3cm}
		
		\noindent
		{\bf (IV-2-6) : $m = 1, \mathrm{Vol}(Z_0) = 1, (b_{\min}, b_{\max}) = (0,1)$} \vs{0.3cm}	

		\noindent
		Note that $Z_0$ is connected and the condition $b_{\min} = 0$ implies that $e(P)_{-2}^+ = -y$ by Lemma \ref{lemma_Euler_extremum}. 
		Moreover, $\mathrm{Vol}(Z_0) = 1$ and $b_{\max} = 1$ implies that 
		\begin{equation}\label{equation_IV_2_5}
			\mathrm{Vol}(Z_0) = 2a+2b+c= 1, \quad \langle e(P_2^-)^2, [M_{2-\epsilon}] \rangle = -1 ~\text{so that $\langle e(P_0^+)^2, [M_0] \rangle = -2$}
		\end{equation}
		where the latter one is equivalent to 
		\[
			2a(b-1) - (c+1)^2 = -2.
		\]
		\vs{0.1cm}
		
		\noindent
		{\bf Case (1) :} $\mathrm{PD}(C) = x - E_1$. \vs{0.1cm}
		
		\noindent
		Since $b+c = 1$, we have 
		\[
			2a + b = 0, \quad 2a(-2a-1) - (2 + 2a)^2 = -2
		\]
		so that $(a,b,c) = (-1,2,-1)$. That is, $\mathrm{PD}(Z_0) = -x + 2y - E_1$ and so $[Z_0] \cdot [Z_0] = -5$. This contradicts the fact that $Z_0$ is conencted by the adjunction formula.
		So, there is no such manifold.
		\vs{0.3cm}
		
		\noindent
		{\bf Case (2) :} $\mathrm{PD}(C) = E_1$.

		\noindent
		We have $c = -2$ by \eqref{equation_vanishing_IV_2} which implies that $2a + 2b = 3$. Thus no such manifold exists.
		\vs{0.3cm}
		
		\noindent
		{\bf Case (3) :} $\mathrm{PD}(C) = y - E_1$.
		
		\noindent
		In this case, we have $a + c = 0$ so that 
		\[
			a + 2b = 1, \quad 2a(b-1) - (1-a)^2 = -2.
		\]							
		It has a unique solution $(a,b,c) = (1,0,-1)$, i.e., 
		\[
			Z_0 \cong S^2, \quad \mathrm{PD}(Z_0) = x - E_1.
		\] 
		See Table \ref{table_IV_2}: {\bf (IV-2-6)}.
	\end{proof}

	\begin{example}[Fano variety of type {\bf (IV-2)}]\label{example_IV_2} In this example, we describe Fano varieties of type {\bf (IV-2)} listed in Theorem \ref{theorem_IV_2}.
		
		\begin{itemize}
	           	 \item {\bf (IV-2-1.1)} \cite[20th in Section 12.4]{IP}  : Recall that a smooth quadric in $\p^4$, isomorphic to a coadjoint orbit of $\mathrm{SO}(5)$, admits a 
	           	 maximal torus $T^2$ action whose moment map image is given on the left of Figure \ref{figure_IV_2_1_1} (see also \cite[Example 6.4]{Cho}). Let $M$ be 
	           	 the blow-up of the smooth quadric along two disjoint $T^2$-invariant spheres with the induced $T^2$-action. Then the corresponding moment map can be 
	           	 described as on the right of Figure \ref{figure_IV_1_1_1}. 
	           	 
	           	 	\begin{figure}[h]
	           	 		\scalebox{1}{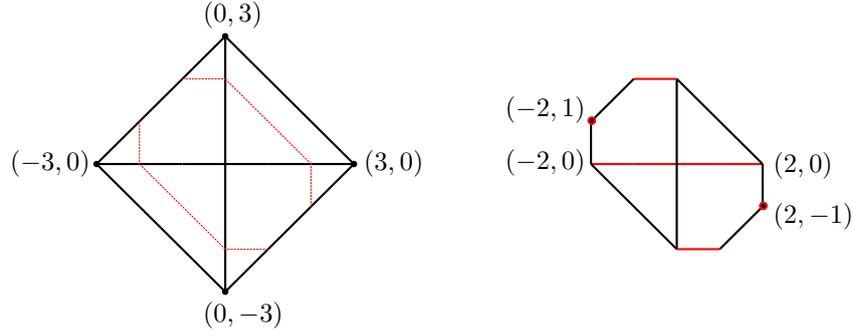}
	           	 		\caption{\label{figure_IV_2_1_1} Blow up of the smooth quadric along two disjoint lines}
	           	 	\end{figure}
	           	 	
			Now, we take the $S^1$-subgroup of $T^2$ generated by $\xi = (0,1) \in \frak{t}$. Then the fixed point set consists of 
	            				\begin{itemize}
	            					\item $Z_{-2} = S^2$ with  $\mu(Z_{-2}) = \overline{(0,-2) ~(1,-2)}$ and $\mathrm{Vol}(Z_{-2}) = 1$,
		            				\item $Z_{-1} = \mathrm{pt}$ with $\mu(Z_{-1}) = (2,-1)$,
		            				\item $Z_{0} = S^2$ with $\mu(Z_{0}) = \overline{(-2,0) ~(2,0)}$ and $\mathrm{Vol}(Z_{0}) = 4$,
	           	 				\item $Z_1 = \mathrm{pt}$ with $\mu(Z_1) = (-2,1)$,
	            					\item $Z_2 = S^2$ with $\mu(Z_2) = \overline{(-1,2) ~(0,2)})$ and $\mathrm{Vol}(Z_{2}) = 1$.
	            				\end{itemize}
			\vs{0.5cm}
			
	           	 \item {\bf (IV-2-1.2)} \cite[8th in Section 12.5]{IP} : Consider $X = \p^1 \times \p^1 \times \p^1$ equipped with $T^2$-action defined by 
	           	 \[
	           	 	(t_1, t_2) \cdot ([x_0 : x_1], [y_0 : y_1], [z_0 : z_1]) := ([t_1x_0 : x_1], [t_2y_0 : y_1], [t_2z_0 : z_1])
	           	 \]
	           	 with respect to the normalized monotone K\"{a}hler form on $X$, the moment map image is give in the middle of Figure \ref{figure_IV_2_1_2}. 
	           	 (Note that the red double line in the middle  indicates the image of the upper-left and lower-right red edges in the first of Figure \ref{figure_IV_2_1_2}.)
	           	 
	           	 Let $C$ be the $T$-invariant sphere given by 
	           	 \[
	           	 	C = \{ ([1:0], [y_0 : y_1],  [y_0 : y_1]) ~|~ [y_0:y_1] \in \p^1\}
			\]
	           	whose moment map image is indicated by the blue line in Figure \ref{figure_IV_2_1_2}. Then, let $M$ be the $T^2$-equivariant blow-up of $X$ whose moment map 
	           	is described in the third of Figure \ref{figure_IV_2_1_2}. The fixed point set consists of 
	            				\begin{itemize}
	            					\item $Z_{-2} = S^2$ with  $\mu(Z_{-2}) = \overline{(1,-2) ~(2,-2)}$ and $\mathrm{Vol}(Z_{-2}) = 1$,
		            				\item $Z_{-1} = \mathrm{pt}$ with $\mu(Z_{-1}) = (0,-1)$,
		            				\item $Z_{0} = S^2 ~\dot \cup~ S^2$ with $\mu(Z_{0}^1) =\mu(Z_{0}^2) = \overline{(0,0) ~(2,0)}$ and $\mathrm{Vol}(Z_{0}^1) = \mathrm{Vol}(Z_{0}^2) = 2$,
	           	 				\item $Z_1 = \mathrm{pt}$ with $\mu(Z_1) = (0,1)$,
	            					\item $Z_2 = S^2$ with $\mu(Z_2) = \overline{(1,2) ~(2,2)})$ and $\mathrm{Vol}(Z_{2}) = 1$.
	            				\end{itemize}
	           	 	\vs{0.3cm}
	           	 
	           	 	\begin{figure}[h]
	           	 		\scalebox{1}{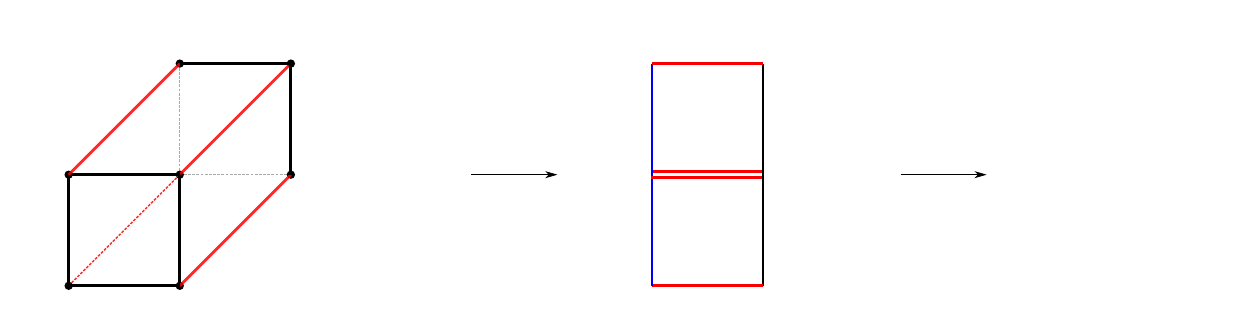}
	           	 		\caption{\label{figure_IV_2_1_2} Blow up of $\p^1 \times \p^1 \times \p^1$ along $C$}
	           	 	\end{figure}	           	 
	           	 
	           	 \item {\bf (IV-2-2.1)} \cite[24th in Section 12.4]{IP} : Consider the complete flag variety $\mcal{F}(3) \cong U(3) / T^3$ together with the induced $T^2$-action whose moment map 
	           	 image is given in the first of Figure \ref{figure_IV_2_2_1}. (See also Example \ref{example_II_1}.) Let $C$ be a $T$-invariant sphere (for instance, take a sphere whose moment map
	           	 image is $\overline{(0,0) ~(0,2)}$ as in Figure \ref{figure_IV_2_2_1}). Let $M$ be the $T^2$-equivariant blow-up of $\mcal{F}(3)$ along $C$. Then the moment map image 
	           	 for the induced $T^2$-action on $M$ can be depicted as in the second in Figure \ref{figure_IV_2_2_1}. 
		           The fixed point set consists of 
                           				\begin{itemize}
	            					\item $Z_{-2} = S^2$ with  $\mu(Z_{-2}) = \overline{(1,0) ~(2,0)}$ and $\mathrm{Vol}(Z_{-2}) = 1$,
		            				\item $Z_{-1} = \mathrm{pt}$ with $\mu(Z_{-1}) = (1,1)$,
		            				\item $Z_{0} = S^2$ with $\mu(Z_{0}) = \overline{(1,2) ~(4,2)}$ and $\mathrm{Vol}(Z_{0}) = 3$,
	           	 				\item $Z_1 = \mathrm{pt}$ with $\mu(Z_1) = (1,3)$,
	            					\item $Z_2 = S^2$ with $\mu(Z_2) = \overline{(2,4) ~(4,4)})$ and $\mathrm{Vol}(Z_{2}) = 2$.
	            				\end{itemize}
	           	 	           	 	\vs{0.3cm}
	           	 	           	 	
	           	 	\begin{figure}[H]
	           	 		\scalebox{1}{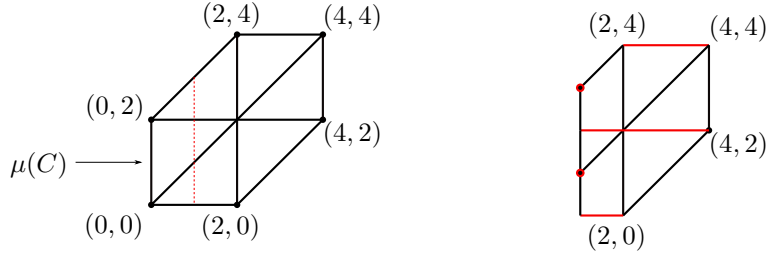}
	           	 		\caption{\label{figure_IV_2_2_1} Blow up of $\mcal{F}(3)$ along $C$}
	           	 	\end{figure}
	            
                          	\item {\bf (IV-2-2.2)} \cite[10th in Section 12.5]{IP} : Consider $\C P^1 \times ~X_2$ with the standard $T^3$-action, where $X_k$ is the $k$-times blow-up of $\p^2$. 
                          	The corresponding moment polytope is given in Figure \ref{figure_IV_2_2_2}. Take a circle subgroup of $T^3$ generated by $\xi = (-1,1,0)$. Then one can easily check
                          	that the $S^1$-action is semifree and the fixed point set consists of 
                           				\begin{itemize}
	            					\item $Z_{-2} = S^2$ with  $\mu(Z_{-2}) = \overline{(2,0,0) ~(2,0,1)}$ and $\mathrm{Vol}(Z_{-2}) = 1$,
		            				\item $Z_{-1} = \mathrm{pt}$ with $\mu(Z_{-1}) = (1,0,2)$,
		            				\item $Z_{0} = S^2$ with $\mu(Z_{0}) = \overline{(2,2,0) ~(2,2,1)}$ and $\mathrm{Vol}(Z_{0}) = 1$,
	           	 				\item $Z_1 = \mathrm{pt}$ with $\mu(Z_1) = (1,2,2)$,
	            					\item $Z_2 = S^2$ with $\mu(Z_2) = \overline{(0,2,0) ~(0,2,2)})$ and $\mathrm{Vol}(Z_{2}) = 2$.
	            				\end{itemize}
	           	 	           	 	\vs{0.3cm}
                          	
	           	 	\begin{figure}[H]
	           	 		\scalebox{1}{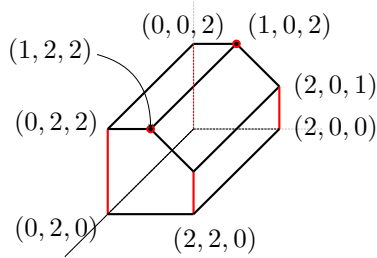}
	           	 		\caption{\label{figure_IV_2_2_2} $S^2 \times X_2$}
	           	 	\end{figure}
                          	                          	
	           	\item {\bf (IV-2-3)} \cite[26th in Section 12.4]{IP} : Consider $\p^3$ with the standard $T^3$-action and let $M$ be the $T^3$-equivariant blow-up of 
	           	$\p^3$ along a disjoint union of a fixed point and a $T^3$-invariant sphere. Then the moment map image of $M$ is described in Figure \ref{figure_IV_2_3}. 
	           	If we take a circle subgroup of $T^3$ generated by $\xi = (0,-1,-1)$, then the $S^1$-action becomes semifree and the fixed point set is give by 
                           				\begin{itemize}
	            					\item $Z_{-2} = S^2$ with  $\mu(Z_{-2}) = \overline{(0,2,2) ~(0,3,1)}$ and $\mathrm{Vol}(Z_{-2}) = 1$,
		            				\item $Z_{-1} = \mathrm{pt}$ with $\mu(Z_{-1}) = (0,3,0)$,
		            				\item $Z_{0} = S^2$ with $\mu(Z_{0}) = \overline{(0,0,2) ~(2,0,2)}$ and $\mathrm{Vol}(Z_{0}) = 2$,
	           	 				\item $Z_1 = \mathrm{pt}$ with $\mu(Z_1) = (3,0,1)$,
	            					\item $Z_2 = S^2$ with $\mu(Z_2) = \overline{(0,0,0) ~(3,0,0)})$ and $\mathrm{Vol}(Z_{2}) = 3$.
	            				\end{itemize}
	           	 	           	 	\vs{0.3cm}
	           	
	           	\textbf{(e.g. the blow-up of $\C P^3$ with center a disjoint union of a point and a line with $c_1^3(M) = 46$)}
	           	(with $\xi = (0,-1,-1)$.)
	           	 	\begin{figure}[H]
	           	 		\scalebox{0.8}{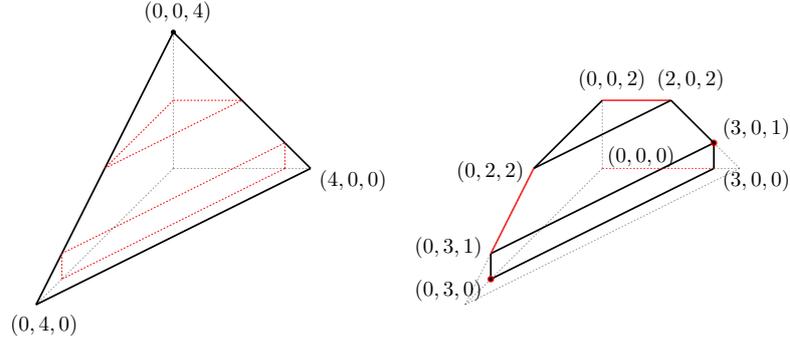}
	           	 		\caption{\label{figure_IV_2_3} Toric blow up of $\p^3$ along a fixed point and a $T^3$-invariant sphere}
	           	 	\end{figure}

	           	\item {\bf (IV-2-4)} \cite[29th in Section 12.4]{IP} : Consider $V_7$, the $T^3$-equivariant blow-up of $\p^3$ at a fixed point. (See also Example \ref{example_II_1}.) 
	           	Take $C$ be any $T^3$-invariant sphere lying on the exceptional divisor of the blow-up $V_7 \rightarrow \p^3$. Then the moment map image is given in 
	           	Figure \ref{figure_IV_2_4}. Take a circle subgroup generated by $\xi = (0,-1,-1)$. 
	           	 	\begin{figure}[H]
	           	 		\scalebox{0.7}{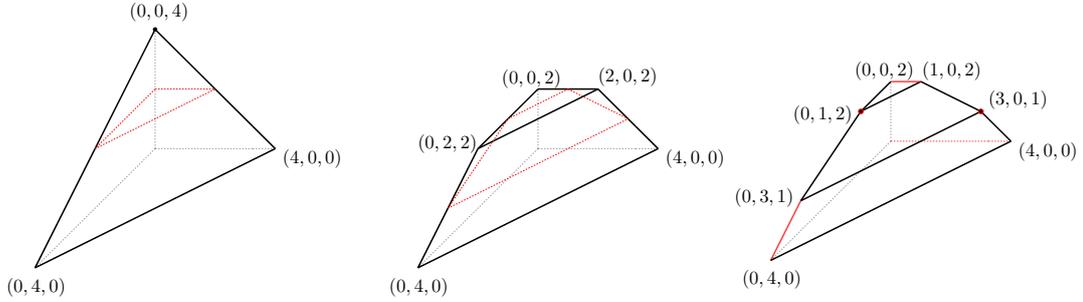}
	           	 		\caption{\label{figure_IV_2_4} Blow up of $V_7$ along a $T$-invariant sphere on the exceptional divisor}
	           	 	\end{figure}
	           	 	\noindent
			The $S^1$-action is semifree and the fixed point set consists of 
                           				\begin{itemize}
	            					\item $Z_{-2} = S^2$ with  $\mu(Z_{-2}) = \overline{(0,4,0) ~(0,3,1)}$ and $\mathrm{Vol}(Z_{-2}) = 1$,
		            				\item $Z_{-1} = \mathrm{pt}$ with $\mu(Z_{-1}) = (0,1,2)$,
		            				\item $Z_{0} = S^2$ with $\mu(Z_{0}) = \overline{(0,0,2) ~(1,0,2)}$ and $\mathrm{Vol}(Z_{0}) = 1$,
	           	 				\item $Z_1 = \mathrm{pt}$ with $\mu(Z_1) = (3,0,1)$,
	            					\item $Z_2 = S^2$ with $\mu(Z_2) = \overline{(0,0,0) ~(4,0,0)})$ and $\mathrm{Vol}(Z_{2}) = 4$.
	            				\end{itemize}
	           	 	           	 	\vs{0.3cm}
	           		           	
	           	\item {\bf (IV-2-5)} \cite[12th in Section 12.5]{IP} : We consider $Y$, the blow-up of $\p^3$ along a $T^3$-invariant line (see Example \ref{example_III}).
	           	Let $C_1$ and $C_2$ be two $T^3$-invariant disjoint lines lying on the exceptional divisor of $Y \rightarrow \p^3$. See Figure \ref{figure_IV_2_5} (a).
	           	Let $M$ be the $T^3$-equivariant blow-up of $Y$ along $C_1$ and $C_2$. Then the moment map image of the induced $T^3$-action is given by 
	           	Figure \ref{figure_IV_2_5}. 
	           	
	           	Take an $S^1$ subgroup of $T^3$ generated by $\xi = (1,0,1)$. One can easily check that the $S^1$-action is semifree and the fixed point set is given by
                           				\begin{itemize}
	            					\item $Z_{-2} = S^2$ with  $\mu(Z_{-2}) = \overline{(0,4,0) ~(0,2,0)}$ and $\mathrm{Vol}(Z_{-2}) = 2$,
		            				\item $Z_{-1} = \mathrm{pt}$ with $\mu(Z_{-1}) = (0,1,1)$,
		            				\item $Z_{0} = S^2 ~\dot \cup ~ S^2$ with 
		            				\[
		            					\mu(Z_{0}^1) = \overline{(0,1,2) ~(0,2,2)},  \quad \mu(Z_{0}^2) = \overline{(1,0,1) ~(2,0,0)}, 
		            					\quad \quad \mathrm{Vol}(Z_{0}^1) = \mathrm{Vol}(Z_{0}^2) = 1,
							\]
	           	 				\item $Z_1 = \mathrm{pt}$ with $\mu(Z_1) = (1,0,2)$,
	            					\item $Z_2 = S^2$ with $\mu(Z_2) = \overline{(2,0,2) ~(4,0,0)})$ and $\mathrm{Vol}(Z_{2}) = 2$.
	            				\end{itemize}
	           	 	           	 	\vs{0.3cm}
	           	
	           	 	\begin{figure}[H]
	           	 		\scalebox{0.7}{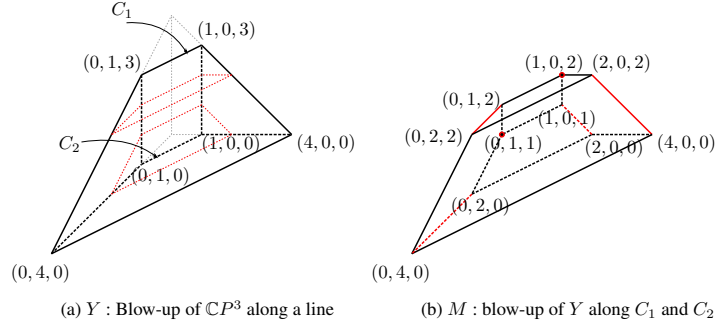}
	           	 		\caption{\label{figure_IV_2_5} Blow up of $Y$ along disjoint $T$-invariant two spheres on the exceptional divisor}
	           	 	\end{figure}

	           	\item {\bf (IV-2-6)} \cite[30th in Section 12.4]{IP} : Consider the $T^3$-equivariant blow-up $V_7$ of $\p^3$ at a fixed point and let $M$ be the blow-up of $V_7$
	           	along a $T^3$-invariant sphere passing through the exceptional divisor of $V_7 \rightarrow \p^3$. Then the moment map image of $M$ with respect to the induced 
	           	action is given by Figure \ref{figure_IV_2_6}. 

	           	 	\begin{figure}[H]
	           	 		\scalebox{0.7}{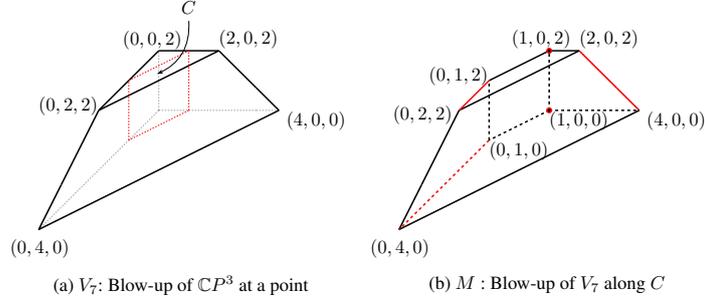}
	           	 		\caption{\label{figure_IV_2_6} Blow up of $V_7$ along a $T$-invariant sphere passing through the exceptional divisor}
	           	 	\end{figure}

			Take a circle subgroup of $T^3$ generated by $\xi = (-1,0,-1)$. Then the action is semifree and the fixed point set 
	           	consists of 
                           				\begin{itemize}
	            					\item $Z_{-2} = S^2$ with  $\mu(Z_{-2}) = \overline{(4,0,0) ~(2,0,2)}$ and $\mathrm{Vol}(Z_{-2}) = 2$,
		            				\item $Z_{-1} = \mathrm{pt}$ with $\mu(Z_{-1}) = (1,0,2)$,
		            				\item $Z_{0} = S^2$ with 
		            				$
		            					\mu(Z_{0}) = \overline{(0,1,2) ~(0,2,2)} 
							$ with $\mathrm{Vol}(Z_{0}^2) = 1$, 
	           	 				\item $Z_1 = \mathrm{pt}$ with $\mu(Z_1) = (1,0,0)$,
	            					\item $Z_2 = S^2$ with $\mu(Z_2) = \overline{(0,1,0) ~(0,4,0)})$ and $\mathrm{Vol}(Z_{2}) = 3$.
	            				\end{itemize}
	           	 	           	 	\vs{0.3cm}
	          \end{itemize}		
	\end{example}

\section{Main Theorem}
\label{secMainTheorem}

	In this section, we prove our main theorem \ref{theorem_main}.
	\begin{theorem}[Theorem \ref{theorem_main}]
		Let $(M,\omega)$ be a six-dimensional closed monotone symplectic manifold equipped with a semifree Hamiltonian 
		circle action. Suppose that the maximal and the minimal fixed component of the action are both 2-dimensional. 
		Then $(M,\omega)$ is $S^1$-equivariantly symplectomorphic to some 
		K\"{a}hler Fano manifold with a certain holomorphic Hamiltonian circle action. 
	\end{theorem}
	
	We list all possible topological fixed point data in Table \ref{table_list}. Notice that our classification implies that 
	any reduced space of $(M,\omega)$ in Theorem \ref{theorem_main} is either $\p^1 \times \p^1$, or $\p^2 \#~k~ \overline{\p^2}$ for $1 \leq k \leq 4$. 
	The following theorems then imply that those spaces 
	are symplectically rigid (in the sense of \cite[Definition 2.13]{McD2} or \cite[Definition 1.4]{G}). (See also Section \ref{secFixedPointData} or \cite[Section 5]{Cho}.)

	\begin{table}[h]
	\begin{adjustbox}{width=1\textwidth}
		\begin{tabular}{|c|c|c|c|c|c|c|c|c|c|}
			\hline
			         & $(M_0, [\omega_0])$           & $e(P_{-2+\epsilon})$ & $Z_{-2}$ & $Z_{-1}$ &  $Z_0$ & $Z_1$ & $Z_2$ & $b_2$ &  $c_1^3$ \\ \hline \hline
	          {\bf (I-1)}      & $(S^2 \times S^2, 2x + 2y)$ & $x-y$                         &  $S^2$    &               &              &            & $S^2$ & $1$ & $64$\\ \hline    
		{\bf (II-1.1)} & $(S^2 \times S^2, 2x + 2y)$ & $-y$                           &$S^2$ & &$Z_0 \cong S^2, ~\mathrm{PD}(Z_0) = x+y$ & & $S^2$ & $2$ &$48$ \\ \hline
		{\bf (II-1.2)} & $(S^2 \times S^2, 2x + 2y)$ & $-y$                           &$S^2$ & & $Z_0 \cong S^2, ~\mathrm{PD}(Z_0) = x$ & & $S^2$ & $2$ & $56$\\ \hline    
		{\bf (II-1.3)} & $(S^2 \times S^2, 2x + 2y)$ & $-y$                           &$S^2$ &  & \makecell{ $Z_0 = Z_0^1 ~\dot \cup ~ Z_0^2$ \\
						$Z_0^1 \cong Z_0^2 \cong S^2$ \\ $\mathrm{PD}(Z_0^1) = \mathrm{PD}(Z_0^2) = y$}   &  & $S^2$ & $3$ & $48$\\ \hline
		{\bf (II-2.1)} & $(E_{S^2}, 3x + 2y)$ & $-x -y$  &$S^2$ & &
				    	\makecell{ $Z_0 = Z_0^1 ~\dot \cup ~ Z_0^2$ \\
					    $Z_0^1 \cong Z_0^2 \cong S^2$ \\ $\mathrm{PD}(Z_0^1) = y$, $\mathrm{PD}(Z_0^2) = x+y$}  & & $S^2$ & $3$ & $48$\\ \hline    
		{\bf (II-2.2)} & $(E_{S^2}, 3x + 2y)$ & $-x-y$  &$S^2$ & & $Z_0 \cong S^2, ~\mathrm{PD}(Z_0) = 2x+2y$ &  & $S^2$ & $2$ &$40$ \\ \hline
		{\bf (III.1)} & \makecell{$(E_{S^2} \# ~\overline{\p^2},$ \\$3x + 2y - E_1)$} & $-y$  &$S^2$ & 
				    	{ pt} & &{pt} & $S^2$ & $2$ & $54$\\ \hline    
		{\bf (III.2)} & \makecell{$(S^2 \times S^2 \# ~2\overline{\p^2},$ \\ $2x + 2y - E_1 - E_2)$} & $-y$  &$S^2$ & 
				    	{2 pts} & &{2 pts}  & $S^2$ & $3$ & $44$\\ \hline    
		{\bf (III.3)} & \makecell{$(E_{S^2} \# ~\overline{\p^2},$ \\ $3x + 2y - E_1)$} & $-x-y$  &$S^2$ & {3 ~pts} & &{3 ~pts} & $S^2$ & $4$ &$34$ \\ \hline
		{\bf (IV-1-1.1)} & \makecell{$(E_{S^2} \# ~2\overline{\p^2},$ \\$3x + 2y - E_1-E_2)$} & $-x-y$  &$S^2$ & 
				    	{ 2 pts} &
				    		\makecell{ $Z_0 = Z_0^1 ~\dot \cup ~ Z_0^2$ \\ $Z_0^1 \cong Z_0^2 \cong S^2$ \\ 
				    		$\mathrm{PD}(Z_0^1) = x+y-E_1 - E_2$ \\ $\mathrm{PD}(Z_0^2) = x - E_1$}
					     & { 2 pts} & $S^2$ & $5$ & $36$\\ \hline    
		{\bf (IV-1-1.2)} & \makecell{$(E_{S^2} \# ~2\overline{\p^2},$ \\$3x + 2y - E_1-E_2)$} & $-x-y$  &$S^2$ & 
				    	{ 2 pts} &
				    		\makecell{ $Z_0 = Z_0^1 ~\dot \cup ~ Z_0^2$ \\ $Z_0^1 \cong Z_0^2 \cong S^2$ \\ $\mathrm{PD}(Z_0^1) = y$ \\ 
				    		$\mathrm{PD}(Z_0^2) = x+y-E_1 - E_2$}
					     & { 2 pts} & $S^2$ & $5$ & $36$\\ \hline    					     
		{\bf (IV-1-1.3)} & \makecell{$(E_{S^2} \# ~2\overline{\p^2},$ \\$3x + 2y - E_1-E_2)$} & $-x-y$  &$S^2$ & 
				    	{ 2 pts} &
				    		\makecell{ $Z_0 \cong S^2$  \\ $\mathrm{PD}(Z_0) = x+y-E_1$}
					     & { 2 pts} & $S^2$ & $4$ & $36$\\ \hline    
		{\bf (IV-1-2)} & \makecell{$(E_{S^2} \# ~2\overline{\p^2},$ \\$3x + 2y - E_1-E_2)$} & $-x-y$  &$S^2$ & 
				    	{ 2 pts} & 
				    		\makecell{ $Z_0 \cong S^2$  \\ $\mathrm{PD}(Z_0) = x - E_1$}
				    	& { 2 pts}  & $S^2$ & $4$ & $40$\\ \hline 
		{\bf (IV-2-1.1)} & \makecell{$(E_{S^2} \# ~\overline{\p^2},$ \\$3x + 2y - E_1)$} & $-x-y$  &$S^2$ & { pt} &
				    		\makecell{ $Z_0 \cong S^2$  \\ $\mathrm{PD}(Z_0) = 2x + y - E_1$}				    
				     &{ pt} & $S^2$ & $3$ &$38$ \\ \hline
		{\bf (IV-2-1.2)} & \makecell{$(E_{S^2} \# ~\overline{\p^2},$ \\$3x + 2y - E_1)$} & $-x-y$  &$S^2$ & { pt} &
				    		\makecell{ $Z_0 = Z_0^1 ~\dot \cup ~ Z_0^2$ \\ $Z_0^1 \cong Z_0^2 \cong S^2$ \\ 
				    		$\mathrm{PD}(Z_0^1) = \mathrm{PD}(Z_0^2) = x + y - E_1$}
				     &{ pt} & $S^2$ & $4$ &$38$ \\ \hline
		{\bf (IV-2-2.1)} & \makecell{$(E_{S^2} \# ~\overline{\p^2},$ \\$3x + 2y - E_1)$} & $-x-y$  &$S^2$ & { pt} & 
				    		\makecell{ $Z_0 \cong S^2$  \\ $\mathrm{PD}(Z_0) = x + y$}						    
				    &{ pt} & $S^2$ & $3$ &$42$ \\ \hline
		{\bf (IV-2-2.2)} & \makecell{$(E_{S^2} \# ~\overline{\p^2},$ \\$3x + 2y - E_1)$} & $-x-y$  &$S^2$ & { pt} & 
				    		\makecell{ $Z_0 = Z_0^1 ~\dot \cup ~ Z_0^2$ \\ $Z_0^1 \cong Z_0^2 \cong S^2$ \\ 
				    		$\mathrm{PD}(Z_0^1) = y$ \\ $\mathrm{PD}(Z_0^2)= x + y - E_1$}				    				    
				    &{ pt} & $S^2$ & $4$ &$42$ \\ \hline				    
		{\bf (IV-2-3)} & \makecell{$(E_{S^2} \# ~\overline{\p^2},$ \\$3x + 2y - E_1)$} & $-x-y$  &$S^2$ & { pt} &
				    		\makecell{ $Z_0 \cong S^2$  \\ $\mathrm{PD}(Z_0) = x$}				    
				     &{ pt} & $S^2$ & $3$ &$46$ \\ \hline
		{\bf (IV-2-4)} & \makecell{$(E_{S^2} \# ~\overline{\p^2},$ \\$3x + 2y - E_1)$} & $-x-y$  &$S^2$ & { pt} &
				    		\makecell{ $Z_0 \cong S^2$ 
				    		 \\ $\mathrm{PD}(Z_0) = E_1$}			    
				     &{ pt} & $S^2$ & $3$ &$50$ \\ \hline
		{\bf (IV-2-5)} & \makecell{$(S^2 \times S^2  \# ~\overline{\p^2},$ \\$2x + 2y - E_1)$} & $-y$  &$S^2$ & { pt} & 
				    		\makecell{ $Z_0 = Z_0^1 \dot \cup Z_0^2$ \\ $Z_0^1 \cong Z_0^2 \cong S^2$ \\ 
				    		$\mathrm{PD}(Z_0^1) = x - E_1$ \\  $\mathrm{PD}(Z_0^2) = y - E_1$  \\ }				    
				    &{ pt} & $S^2$ & $4$ &$46$ \\ \hline
		{\bf (IV-2-6)} & \makecell{$(S^2 \times S^2  \# ~\overline{\p^2},$ \\$2x + 2y - E_1)$} & $-y$  &$S^2$ & { pt} &
				    		\makecell{ $Z_0 \cong S^2$  \\ $\mathrm{PD}(Z_0) = x - E_1$}					    
				     &{ pt} & $S^2$ & $3$ &$50$ \\ \hline
		\end{tabular}
		\end{adjustbox}
		\vs{0.1cm}
		\caption {List of topological fixed point data} \label{table_list} 
	\end{table}

	\begin{theorem}\cite[Theorem 1.2]{McD4}\label{theorem_uniqueness}
		Let $M$ be the blow-up of a rational or a ruled symplectic four manifold. Then any two cohomologous and deformation equivalent\footnote{Two symplectic forms $\omega_0$ and $\omega_1$
		are said to be {\em deformation equivalent} if there exists a family of symplectic forms $\{ \omega_t  ~|~  0 \leq t \leq 1 \}$ connecting $\omega_0$ and $\omega_1$. We also say that 
		$\omega_0$ and $\omega_1$ are {\em isotopic} if such a family can be chosen such that $[\omega_t]$ is a constant path in $H^2(M; \Z)$.}
		symplectic forms on $M$ are isotopic.
	\end{theorem}

	\begin{theorem}\cite[Lemma 4.2]{G}\label{theorem_symplectomorphism_group}
	For any of the following symplectic manifolds, the group of symplectomorphisms  which act trivially on homology is path-connected. 
		\begin{itemize}
			\item $\p^2$ with the Fubini-Study form. \cite[Remark in p.311]{Gr}
			\item $\p^1 \times \p^1$ with any symplectic form. \cite[Theorem 1.1]{AM}
			\item $\p^2 \# ~k~\overline{\p^2}$ with any blow-up symplectic form for $k \leq 4$. \cite[Theorem 1.4]{AM}, \cite{E}, \cite{LaP}, \cite{Pin} \cite{LLW}. 
		\end{itemize}		
	\end{theorem}

	\begin{remark}
		In \cite[Theorem 9.3]{Cho}, the author only mentioned the symplectic rigidity of $X_k = \p^2 \# k \overline{\p^2}$ for $k \leq 3$ since $X_k$ ($k > 3$) does not appear as a reduced space
		when an extremal fixed point set is an isolated point. On the other hand, in our case of Theorem \ref{theorem_main}, $X_4$ appears as a reduced space, see {\bf (III.3)}.  
		Recently, Li-Li-Wu proved the symplectic rigidity of $X_4$ in \cite{LLW} (where it fails from $k=5$, see \cite{Se}).
	\end{remark}
	
	To complete the proof of Theorem \ref{theorem_main}, we only need to show that each TFD determines FD uniquely. (Then the proof follows by Gonzalez theorem \ref{theorem_Gonzalez}
	from the fact that every reduced space is symplectically rigid and the existence of a Fano variety corresponding to each TFD as illustrated from Section 
	\ref{secCaseIMathrmCritMathringHEmptyset} to \ref{secCaseIVMathrmCritMathringH11}.)
	Note that a topological fixed point data only records homology classes of fixed components regarded as embedded submanifolds of reduced spaces. 
	In general, we cannot rule out the possibility that there are many distinct
	fixed point data which have the same topological fixed point data. 
	
	Recall that any non-extremal part of a topological fixed point data in Table \ref{table_list} is one of the forms
	\[
		(M_c, [\omega_c], [Z_c^1], \cdots, [Z_c^{k_c}]), \quad c = -1, 0, 1. 
	\]
	If $c = \pm 1$, then all $Z_c^i$'s are isolated points. In this case, the topological fixed point data determines a fixed point data uniquely, since if 
	\[
		(M_c, \omega_c, p_1, \cdots, p_r) \quad \text{and} \quad (M_c, \omega_c' , q_1, \cdots, q_r), \quad \quad p_i, q_j : \text{points},\quad  [\omega_c] = [\omega_c'],
	\]
	then it follows from the symplectic rigidity of $M_c$ (obtained by Theorem \ref{theorem_uniqueness} and Theorem \ref{theorem_symplectomorphism_group})
	that there exists a symplectomorphism $\phi : (M_c, \omega_c) \rightarrow (M_c, \omega_c')$ sending $p_i$ to $q_i$ for $i=1,\cdots,r$. 
	(See \cite[Proposition 0.3]{ST}.)  
	
	For $c= 0$, 
	we note that every $Z_0^i$ in Table \ref{table_list} is a sphere with self intersection greater than equal to $-1$. 
	Then the following theorems guarantee that any symplectic embedding $Z_0 \hookrightarrow M_0$ in Table \ref{table_list}
	can be identified with an algebraic embedding. 
				
	\begin{theorem}\cite[Proposition 3.2]{LW}\cite[Theorem 6.9]{Z}\label{theorem_Z}
		Any symplectic sphere $S$ with self-intersection $[S]\cdot[S] \geq 0$ in a symplectic four manifold $(M,\omega)$ is symplectically isotopic to an (algebraic) rational curve.
		Any two homologous spheres with self-intersection $-1$ are symplectically isotopic to each other.
	\end{theorem}

	Furthermore, we may apply the following lemma to each reduced space since every rational surface $X$ satisfies $H^1(X, \mcal{O}_X) = 0$.
	
	\begin{lemma}\label{lemma_isotopic}\cite[Lemma 9.6]{Cho}
		Suppose that $X$ is a smooth projective surface with $H^1(X, \mcal{O}_X) = 0$. Let $H_1$ and $H_2$ be two smooth curves of $X$ representing the same homology class.
		Then $H_1$ is symplectically isotopic to $H_2$ with respect to the symplectic form $\omega_X = \omega_{\mathrm{FS}}|_X$ on $X$.
	\end{lemma}

	Now we are ready to prove Theorem \ref{theorem_main}
	
	\begin{proof}[Proof of Theorem \ref{theorem_main}]
		Let $(M,\omega)$ be a six-dimensional closed monotone symplectic manifold with $c_1(TM) = [\omega]$. Also assume that $(M,\omega)$ admits a semifree Hamiltonian circle action 
		with the balanced moment map $H : M \rightarrow \R$. By Table \ref{table_list}, we know that any reduced space is either 
		\[
			\p^1 \times \p^1, \quad \p^2 \# k\overline{\p^2}, \quad \quad k \leq 4
		\]
		and hence is symplecticaly rigid by Theorem \ref{theorem_uniqueness} and Theorem \ref{theorem_symplectomorphism_group}. 
		Moreover, we also know that there exists a smooth Fano 3-fold admitting semifree holomorphic Hamiltonian $S^1$-action whose topological fixed point data 
		equals $\frak{F}_{\mathrm{top}}(M,\omega,H)$. 
		So, it remains to show that $\frak{F}_{\mathrm{top}}(M,\omega,H)$ determines $\frak{F}(M,\omega,H)$ uniquely.
		
		By Theorem \ref{theorem_Z}, we may assume that every $(M_c, \omega_c, Z_c) \in \frak{F}(M,\omega,H)$ is an algebraic tuple, that is, 
		$Z_c$ is a complex (and hence K\"{a}hler) submanifold of $M_c$ for every critical value $c$ of the balanced moment map $H$. 
		Moreover, since any reduced space is birationally equivalent to $\p^2$, we see that $H^1(M_c, \mcal{O}_{M_c}) = 0$ and therefore we may apply Lemma \ref{lemma_isotopic}
		so that $(M_c, \omega_c, Z_c)$ is equivalent to the fixed point data $(X_c, (\omega_X)_c, (Z_X)_c)$ of $X$ at level $c$. 
		This completes the proof.
	\end{proof}

\bibliographystyle{annotation}

\end{document}

%% file: figure_summary.pdf_tex
\begingroup%
  \makeatletter%
  \providecommand\color[2][]{%
    \errmessage{(Inkscape) Color is used for the text in Inkscape, but the package 'color.sty' is not loaded}%
    \renewcommand\color[2][]{}%
  }%
  \providecommand\transparent[1]{%
    \errmessage{(Inkscape) Transparency is used (non-zero) for the text in Inkscape, but the package 'transparent.sty' is not loaded}%
    \renewcommand\transparent[1]{}%
  }%
  \providecommand\rotatebox[2]{#2}%
  \ifx\svgwidth\undefined%
    \setlength{\unitlength}{670.50811768bp}%
    \ifx\svgscale\undefined%
      \relax%
    \else%
      \setlength{\unitlength}{\unitlength * \real{\svgscale}}%
    \fi%
  \else%
    \setlength{\unitlength}{\svgwidth}%
  \fi%
  \global\let\svgwidth\undefined%
  \global\let\svgscale\undefined%
  \makeatother%
  \begin{picture}(1,0.89484373)%
    \put(0,0){\includegraphics[width=\unitlength,page=1]{figure_summary.pdf}}%
    \put(0.08648104,0.70226301){\color[rgb]{0,0,0}\makebox(0,0)[lb]{\smash{{\bf (I-1)}}}}%
    \put(0.23997937,0.70394375){\color[rgb]{0,0,0}\makebox(0,0)[lb]{\smash{{\bf (II-1.1)}}}}%
    \put(0.40105125,0.70394375){\color[rgb]{0,0,0}\makebox(0,0)[lb]{\smash{{\bf (II-1-2)}}}}%
    \put(0.56808875,0.70394375){\color[rgb]{0,0,0}\makebox(0,0)[lb]{\smash{{\bf (II-1.3)}}}}%
    \put(0,0){\includegraphics[width=\unitlength,page=2]{figure_summary.pdf}}%
    \put(0.6993325,0.70394375){\color[rgb]{0,0,0}\makebox(0,0)[lb]{\smash{{\bf (II-2.1)}}}}%
    \put(0.8902325,0.70394375){\color[rgb]{0,0,0}\makebox(0,0)[lb]{\smash{{\bf (II-2.2)}}}}%
    \put(0.97971687,0.90080934){\color[rgb]{0,0,0}\makebox(0,0)[lt]{\begin{minipage}{0.04175938\unitlength}\raggedright \end{minipage}}}%
    \put(0,0){\includegraphics[width=\unitlength,page=3]{figure_summary.pdf}}%
    \put(0.0789075,0.44742188){\color[rgb]{0,0,0}\makebox(0,0)[lb]{\smash{{\bf (III.1)}}}}%
    \put(0.30560125,0.44742188){\color[rgb]{0,0,0}\makebox(0,0)[lb]{\smash{{\bf (III.2)}}}}%
    \put(0.46667312,0.44742188){\color[rgb]{0,0,0}\makebox(0,0)[lb]{\smash{{\bf (III.3)}}}}%
    \put(0,0){\includegraphics[width=\unitlength,page=4]{figure_summary.pdf}}%
    \put(0.62110358,0.44574113){\color[rgb]{0,0,0}\makebox(0,0)[lb]{\smash{{\bf (IV-1-1.1)}}}}%
    \put(0,0){\includegraphics[width=\unitlength,page=5]{figure_summary.pdf}}%
    \put(0.78217546,0.44574113){\color[rgb]{0,0,0}\makebox(0,0)[lb]{\smash{{\bf (IV-1-1.2)}}}}%
    \put(0,0){\includegraphics[width=\unitlength,page=6]{figure_summary.pdf}}%
    \put(0.92006062,0.44742188){\color[rgb]{0,0,0}\makebox(0,0)[lb]{\smash{{\bf (IV-1-1.3)}}}}%
    \put(0,0){\includegraphics[width=\unitlength,page=7]{figure_summary.pdf}}%
    \put(0.07294187,0.238625){\color[rgb]{0,0,0}\makebox(0,0)[lb]{\smash{{\bf (IV-1-2)}}}}%
    \put(0,0){\includegraphics[width=\unitlength,page=8]{figure_summary.pdf}}%
    \put(0.26384187,0.238625){\color[rgb]{0,0,0}\makebox(0,0)[lb]{\smash{{\bf (IV-2-1.1)}}}}%
    \put(0.4129825,0.238625){\color[rgb]{0,0,0}\makebox(0,0)[lb]{\smash{{\bf (IV-2-1.2)}}}}%
    \put(0.34736062,0.34004063){\color[rgb]{0,0,0}\makebox(0,0)[lb]{\smash{}}}%
    \put(0,0){\includegraphics[width=\unitlength,page=9]{figure_summary.pdf}}%
    \put(0.532295,0.238625){\color[rgb]{0,0,0}\makebox(0,0)[lb]{\smash{{\bf (IV-2-2.1)}}}}%
    \put(0,0){\includegraphics[width=\unitlength,page=10]{figure_summary.pdf}}%
    \put(0.6993325,0.238625){\color[rgb]{0,0,0}\makebox(0,0)[lb]{\smash{{\bf (IV-2-2.2)}}}}%
    \put(0,0){\includegraphics[width=\unitlength,page=11]{figure_summary.pdf}}%
    \put(0.87830125,0.238625){\color[rgb]{0,0,0}\makebox(0,0)[lb]{\smash{{\bf (IV-2-3)}}}}%
    \put(0,0){\includegraphics[width=\unitlength,page=12]{figure_summary.pdf}}%
    \put(0.09680437,0.0238625){\color[rgb]{0,0,0}\makebox(0,0)[lb]{\smash{{\bf (IV-2-4)}}}}%
    \put(0,0){\includegraphics[width=\unitlength,page=13]{figure_summary.pdf}}%
    \put(0.3175325,0.0238625){\color[rgb]{0,0,0}\makebox(0,0)[lb]{\smash{{\bf (IV-2-5)}}}}%
    \put(0,0){\includegraphics[width=\unitlength,page=14]{figure_summary.pdf}}%
    \put(0.55019187,0.0238625){\color[rgb]{0,0,0}\makebox(0,0)[lb]{\smash{{\bf (IV-2-6)}}}}%
    \put(0.29963562,-0){\color[rgb]{0,0,0}\makebox(0,0)[lb]{\smash{     }}}%
  \end{picture}%
\endgroup%

%% file: figure_II_1_1.pdf_tex
\begingroup%
  \makeatletter%
  \providecommand\color[2][]{%
    \errmessage{(Inkscape) Color is used for the text in Inkscape, but the package 'color.sty' is not loaded}%
    \renewcommand\color[2][]{}%
  }%
  \providecommand\transparent[1]{%
    \errmessage{(Inkscape) Transparency is used (non-zero) for the text in Inkscape, but the package 'transparent.sty' is not loaded}%
    \renewcommand\transparent[1]{}%
  }%
  \providecommand\rotatebox[2]{#2}%
  \ifx\svgwidth\undefined%
    \setlength{\unitlength}{146.59867554bp}%
    \ifx\svgscale\undefined%
      \relax%
    \else%
      \setlength{\unitlength}{\unitlength * \real{\svgscale}}%
    \fi%
  \else%
    \setlength{\unitlength}{\svgwidth}%
  \fi%
  \global\let\svgwidth\undefined%
  \global\let\svgscale\undefined%
  \makeatother%
  \begin{picture}(1,0.89420293)%
    \put(0,0){\includegraphics[width=\unitlength,page=1]{figure_II_1_1.pdf}}%
    \put(0.77467407,0.37873148){\color[rgb]{0,0,0}\makebox(0,0)[lb]{\smash{$(4,0,0)$}}}%
    \put(0.49851705,0.77616303){\color[rgb]{0,0,0}\makebox(0,0)[lb]{\smash{$(0,0,4)$}}}%
    \put(0.12781681,0.06895253){\color[rgb]{0,0,0}\makebox(0,0)[lb]{\smash{$(0,4,0)$}}}%
    \put(0,0){\includegraphics[width=\unitlength,page=2]{figure_II_1_1.pdf}}%
    \put(-0.00309758,0.00721437){\color[rgb]{0,0,0}\makebox(0,0)[lb]{\smash{$x$}}}%
    \put(0.92653732,0.46720811){\color[rgb]{0,0,0}\makebox(0,0)[lb]{\smash{$y$}}}%
    \put(0.42134236,0.8627075){\color[rgb]{0,0,0}\makebox(0,0)[lb]{\smash{$z$}}}%
  \end{picture}%
\endgroup%

%% file: figure_II_2_1.pdf_tex
\begingroup%
  \makeatletter%
  \providecommand\color[2][]{%
    \errmessage{(Inkscape) Color is used for the text in Inkscape, but the package 'color.sty' is not loaded}%
    \renewcommand\color[2][]{}%
  }%
  \providecommand\transparent[1]{%
    \errmessage{(Inkscape) Transparency is used (non-zero) for the text in Inkscape, but the package 'transparent.sty' is not loaded}%
    \renewcommand\transparent[1]{}%
  }%
  \providecommand\rotatebox[2]{#2}%
  \ifx\svgwidth\undefined%
    \setlength{\unitlength}{407.5470459bp}%
    \ifx\svgscale\undefined%
      \relax%
    \else%
      \setlength{\unitlength}{\unitlength * \real{\svgscale}}%
    \fi%
  \else%
    \setlength{\unitlength}{\svgwidth}%
  \fi%
  \global\let\svgwidth\undefined%
  \global\let\svgscale\undefined%
  \makeatother%
  \begin{picture}(1,0.37029318)%
    \put(0.10980094,0.22272394){\color[rgb]{0,0,0}\makebox(0,0)[lb]{\smash{}}}%
    \put(0.06072685,0.21290913){\color[rgb]{0,0,0}\makebox(0,0)[lb]{\smash{}}}%
    \put(0.02146758,0.27179797){\color[rgb]{0,0,0}\makebox(0,0)[lb]{\smash{}}}%
    \put(0,0){\includegraphics[width=\unitlength,page=1]{figure_II_2_1.pdf}}%
    \put(0.21776393,0.3405017){\color[rgb]{0,0,0}\makebox(0,0)[lb]{\smash{$(4,4)$}}}%
    \put(0.21776393,0.2325387){\color[rgb]{0,0,0}\makebox(0,0)[lb]{\smash{$(4,2)$}}}%
    \put(0.10702339,0.34248461){\color[rgb]{0,0,0}\makebox(0,0)[lb]{\smash{$(2,4)$}}}%
    \put(0.10603193,0.14708012){\color[rgb]{0,0,0}\makebox(0,0)[lb]{\smash{$(2,0)$}}}%
    \put(-0.00074282,0.14757589){\color[rgb]{0,0,0}\makebox(0,0)[lb]{\smash{$(0,0)$}}}%
    \put(0.00024864,0.25702591){\color[rgb]{0,0,0}\makebox(0,0)[lb]{\smash{$(0,2)$}}}%
    \put(0.59272334,0.24709984){\color[rgb]{0,0,0}\makebox(0,0)[lb]{\smash{$(4,0,0)$}}}%
    \put(0.51894933,0.32642465){\color[rgb]{0,0,0}\makebox(0,0)[lb]{\smash{$(2,0,2)$}}}%
    \put(0.41009195,0.32582931){\color[rgb]{0,0,0}\makebox(0,0)[lb]{\smash{$(0,0,2)$}}}%
    \put(0.32076722,0.24929362){\color[rgb]{0,0,0}\makebox(0,0)[lb]{\smash{$(0,2,2)$}}}%
    \put(0.30472074,0.10701465){\color[rgb]{0,0,0}\makebox(0,0)[lb]{\smash{$(0,4,0)$}}}%
    \put(0.4729492,0.26198316){\color[rgb]{0,0,0}\makebox(0,0)[lb]{\smash{$(0,0,0)$}}}%
    \put(0.07709047,0.09234214){\color[rgb]{0,0,0}\makebox(0,0)[lb]{\smash{$\xi = (1,0)$}}}%
    \put(0.45496095,0.09234214){\color[rgb]{0,0,0}\makebox(0,0)[lb]{\smash{$\xi = (1,1,0)$}}}%
    \put(0,0){\includegraphics[width=\unitlength,page=2]{figure_II_2_1.pdf}}%
    \put(0.79921713,0.09136234){\color[rgb]{0,0,0}\makebox(0,0)[lb]{\smash{$\xi = (1,0,1)$}}}%
    \put(0.13479991,0.06320823){\color[rgb]{0,0,0}\makebox(0,0)[lb]{\smash{}}}%
    \put(0,0){\includegraphics[width=\unitlength,page=3]{figure_II_2_1.pdf}}%
    \put(0.95288384,0.3459546){\color[rgb]{0,0,0}\makebox(0,0)[lb]{\smash{$(2,0,2)$}}}%
    \put(0.95337962,0.23303436){\color[rgb]{0,0,0}\makebox(0,0)[lb]{\smash{$(2,0,0)$}}}%
    \put(0.833716,0.14361007){\color[rgb]{0,0,0}\makebox(0,0)[lb]{\smash{$(2,2,0)$}}}%
    \put(0.72852803,0.14410584){\color[rgb]{0,0,0}\makebox(0,0)[lb]{\smash{$(0,2,0)$}}}%
    \put(0.82737123,0.34645031){\color[rgb]{0,0,0}\makebox(0,0)[lb]{\smash{$(0,0,2)$}}}%
    \put(0.72537033,0.26851303){\color[rgb]{0,0,0}\makebox(0,0)[lb]{\smash{$(0,2,2)$}}}%
    \put(0,0){\includegraphics[width=\unitlength,page=4]{figure_II_2_1.pdf}}%
    \put(0.83204549,0.26900874){\color[rgb]{0,0,0}\makebox(0,0)[lb]{\smash{$(2,2,2)$}}}%
    \put(0,0){\includegraphics[width=\unitlength,page=5]{figure_II_2_1.pdf}}%
    \put(0.12281189,0.03123052){\color[rgb]{0,0,0}\makebox(0,0)[lb]{\smash{(a)}}}%
    \put(0.48468984,0.0317263){\color[rgb]{0,0,0}\makebox(0,0)[lb]{\smash{(b)}}}%
    \put(0.84904638,0.02974332){\color[rgb]{0,0,0}\makebox(0,0)[lb]{\smash{(c)}}}%
    \put(0,0){\includegraphics[width=\unitlength,page=6]{figure_II_2_1.pdf}}%
  \end{picture}%
\endgroup%

%% file: figure_II_2_2.pdf_tex
\begingroup%
  \makeatletter%
  \providecommand\color[2][]{%
    \errmessage{(Inkscape) Color is used for the text in Inkscape, but the package 'color.sty' is not loaded}%
    \renewcommand\color[2][]{}%
  }%
  \providecommand\transparent[1]{%
    \errmessage{(Inkscape) Transparency is used (non-zero) for the text in Inkscape, but the package 'transparent.sty' is not loaded}%
    \renewcommand\transparent[1]{}%
  }%
  \providecommand\rotatebox[2]{#2}%
  \ifx\svgwidth\undefined%
    \setlength{\unitlength}{328.07187805bp}%
    \ifx\svgscale\undefined%
      \relax%
    \else%
      \setlength{\unitlength}{\unitlength * \real{\svgscale}}%
    \fi%
  \else%
    \setlength{\unitlength}{\svgwidth}%
  \fi%
  \global\let\svgwidth\undefined%
  \global\let\svgscale\undefined%
  \makeatother%
  \begin{picture}(1,0.49768184)%
    \put(0,0){\includegraphics[width=\unitlength,page=1]{figure_II_2_2.pdf}}%
    \put(0.24908742,0.42230446){\color[rgb]{0,0,0}\makebox(0,0)[lb]{\smash{$(1,0,2)$}}}%
    \put(0.13418132,0.42107286){\color[rgb]{0,0,0}\makebox(0,0)[lb]{\smash{$(0,0,2)$}}}%
    \put(0.20253636,0.32612028){\color[rgb]{0,0,0}\makebox(0,0)[lb]{\smash{$(0,0,0)$}}}%
    \put(0.37101193,0.3125724){\color[rgb]{0,0,0}\makebox(0,0)[lb]{\smash{$(3,0,0)$}}}%
    \put(0.22470252,0.17845553){\color[rgb]{0,0,0}\makebox(0,0)[lb]{\smash{$(3,2,0)$}}}%
    \put(0.03417845,0.1790714){\color[rgb]{0,0,0}\makebox(0,0)[lb]{\smash{$(0,2,0)$}}}%
    \put(-0.00092277,0.32981515){\color[rgb]{0,0,0}\makebox(0,0)[lb]{\smash{$(0,2,2)$}}}%
    \put(0.11669412,0.27968992){\color[rgb]{0,0,0}\makebox(0,0)[lb]{\smash{$(1,2,2)$}}}%
    \put(0,0){\includegraphics[width=\unitlength,page=2]{figure_II_2_2.pdf}}%
    \put(0.14929783,0.04530495){\color[rgb]{0,0,0}\makebox(0,0)[lb]{\smash{(a) ~$\p^1 \times F_1$}}}%
    \put(0,0){\includegraphics[width=\unitlength,page=3]{figure_II_2_2.pdf}}%
    \put(0.76672472,0.12252594){\color[rgb]{0,0,0}\makebox(0,0)[lb]{\smash{$(0,-3)$}}}%
    \put(0.76672472,0.48829937){\color[rgb]{0,0,0}\makebox(0,0)[lb]{\smash{$(0,3)$}}}%
    \put(0.95427831,0.30161163){\color[rgb]{0,0,0}\makebox(0,0)[lb]{\smash{$(3,0)$}}}%
    \put(0.539735,0.30161163){\color[rgb]{0,0,0}\makebox(0,0)[lb]{\smash{$(-3,0)$}}}%
    \put(0,0){\includegraphics[width=\unitlength,page=4]{figure_II_2_2.pdf}}%
    \put(0.63232945,0.04587067){\color[rgb]{0,0,0}\makebox(0,0)[lb]{\smash{(b) ~Blow-up of $Q$ along a conic}}}%
    \put(0.84663193,0.42397898){\color[rgb]{0,0,0}\makebox(0,0)[lb]{\smash{$(1,2)$}}}%
    \put(0.84847937,0.19243383){\color[rgb]{0,0,0}\makebox(0,0)[lb]{\smash{$(1,-2)$}}}%
    \put(0.65942514,0.42274738){\color[rgb]{0,0,0}\makebox(0,0)[lb]{\smash{$(-1,2)$}}}%
    \put(0.63417693,0.19304971){\color[rgb]{0,0,0}\makebox(0,0)[lb]{\smash{$(-1,-2)$}}}%
  \end{picture}%
\endgroup%

%% file: figure_II_3.pdf_tex
\begingroup%
  \makeatletter%
  \providecommand\color[2][]{%
    \errmessage{(Inkscape) Color is used for the text in Inkscape, but the package 'color.sty' is not loaded}%
    \renewcommand\color[2][]{}%
  }%
  \providecommand\transparent[1]{%
    \errmessage{(Inkscape) Transparency is used (non-zero) for the text in Inkscape, but the package 'transparent.sty' is not loaded}%
    \renewcommand\transparent[1]{}%
  }%
  \providecommand\rotatebox[2]{#2}%
  \ifx\svgwidth\undefined%
    \setlength{\unitlength}{364.34863281bp}%
    \ifx\svgscale\undefined%
      \relax%
    \else%
      \setlength{\unitlength}{\unitlength * \real{\svgscale}}%
    \fi%
  \else%
    \setlength{\unitlength}{\svgwidth}%
  \fi%
  \global\let\svgwidth\undefined%
  \global\let\svgscale\undefined%
  \makeatother%
  \begin{picture}(1,0.43846843)%
    \put(0,0){\includegraphics[width=\unitlength,page=1]{figure_II_3.pdf}}%
    \put(0.4171143,0.24268527){\color[rgb]{0,0,0}\makebox(0,0)[lb]{\smash{$(4,0,0)$}}}%
    \put(0.27150992,0.40902623){\color[rgb]{0,0,0}\makebox(0,0)[lb]{\smash{$(1,0,3)$}}}%
    \put(0.27816389,0.24046732){\color[rgb]{0,0,0}\makebox(0,0)[lb]{\smash{$(1,0,0)$}}}%
    \put(0.17480997,0.19433533){\color[rgb]{0,0,0}\makebox(0,0)[lb]{\smash{$(0,1,0)$}}}%
    \put(0.10427728,0.36101066){\color[rgb]{0,0,0}\makebox(0,0)[lb]{\smash{$(0,1,3)$}}}%
    \put(-0.00083089,0.05071309){\color[rgb]{0,0,0}\makebox(0,0)[lb]{\smash{$(0,4,0)$}}}%
    \put(0.08808822,0.00282195){\color[rgb]{0,0,0}\makebox(0,0)[lb]{\smash{(a) ~Blow-up of $\p^3$ along a line }}}%
    \put(0.2452284,0.11970393){\color[rgb]{0,0,0}\makebox(0,0)[lb]{\smash{}}}%
    \put(0,0){\includegraphics[width=\unitlength,page=2]{figure_II_3.pdf}}%
    \put(0.83141321,0.40902623){\color[rgb]{0,0,0}\makebox(0,0)[lb]{\smash{$(1,0,3)$}}}%
    \put(0.81256027,0.27983666){\color[rgb]{0,0,0}\makebox(0,0)[lb]{\smash{$(1,0,0)$}}}%
    \put(0.69978008,0.23647723){\color[rgb]{0,0,0}\makebox(0,0)[lb]{\smash{$(0,1,0)$}}}%
    \put(0.66418051,0.36101066){\color[rgb]{0,0,0}\makebox(0,0)[lb]{\smash{$(0,1,3)$}}}%
    \put(0.5942052,0.00226753){\color[rgb]{0,0,0}\makebox(0,0)[lb]{\smash{(b) ~Blow-up of $\p^3$ along disjoint two lines }}}%
    \put(0,0){\includegraphics[width=\unitlength,page=3]{figure_II_3.pdf}}%
    \put(0.55396778,0.09850855){\color[rgb]{0,0,0}\makebox(0,0)[lb]{\smash{$(0,3,0)$}}}%
    \put(0.55395545,0.15419145){\color[rgb]{0,0,0}\makebox(0,0)[lb]{\smash{$(0,3,1)$}}}%
    \put(0.94785214,0.3063383){\color[rgb]{0,0,0}\makebox(0,0)[lb]{\smash{$(3,0,1)$}}}%
    \put(0.94785214,0.24046732){\color[rgb]{0,0,0}\makebox(0,0)[lb]{\smash{$(3,0,0)$}}}%
    \put(0,0){\includegraphics[width=\unitlength,page=4]{figure_II_3.pdf}}%
  \end{picture}%
\endgroup%

%% file: figure_II_3_blow_up.pdf_tex
\begingroup%
  \makeatletter%
  \providecommand\color[2][]{%
    \errmessage{(Inkscape) Color is used for the text in Inkscape, but the package 'color.sty' is not loaded}%
    \renewcommand\color[2][]{}%
  }%
  \providecommand\transparent[1]{%
    \errmessage{(Inkscape) Transparency is used (non-zero) for the text in Inkscape, but the package 'transparent.sty' is not loaded}%
    \renewcommand\transparent[1]{}%
  }%
  \providecommand\rotatebox[2]{#2}%
  \ifx\svgwidth\undefined%
    \setlength{\unitlength}{355.90463977bp}%
    \ifx\svgscale\undefined%
      \relax%
    \else%
      \setlength{\unitlength}{\unitlength * \real{\svgscale}}%
    \fi%
  \else%
    \setlength{\unitlength}{\svgwidth}%
  \fi%
  \global\let\svgwidth\undefined%
  \global\let\svgscale\undefined%
  \makeatother%
  \begin{picture}(1,0.27647848)%
    \put(0,0){\includegraphics[width=\unitlength,page=1]{figure_II_3_blow_up.pdf}}%
    \put(0.92260042,0.01461066){\color[rgb]{0,0,0}\makebox(0,0)[lb]{\smash{$Z_{-2}$}}}%
    \put(0,0){\includegraphics[width=\unitlength,page=2]{figure_II_3_blow_up.pdf}}%
    \put(0.92260042,0.25062886){\color[rgb]{0,0,0}\makebox(0,0)[lb]{\smash{$Z_2$}}}%
    \put(0,0){\includegraphics[width=\unitlength,page=3]{figure_II_3_blow_up.pdf}}%
    \put(0.94507834,0.1719561){\color[rgb]{0,0,0}\makebox(0,0)[lb]{\smash{$Z_1$}}}%
    \put(0.94507834,0.08204445){\color[rgb]{0,0,0}\makebox(0,0)[lb]{\smash{$Z_{-1}$}}}%
  \end{picture}%
\endgroup%

%% file: figure_II_4_1_1.pdf_tex
\begingroup%
  \makeatletter%
  \providecommand\color[2][]{%
    \errmessage{(Inkscape) Color is used for the text in Inkscape, but the package 'color.sty' is not loaded}%
    \renewcommand\color[2][]{}%
  }%
  \providecommand\transparent[1]{%
    \errmessage{(Inkscape) Transparency is used (non-zero) for the text in Inkscape, but the package 'transparent.sty' is not loaded}%
    \renewcommand\transparent[1]{}%
  }%
  \providecommand\rotatebox[2]{#2}%
  \ifx\svgwidth\undefined%
    \setlength{\unitlength}{351.24747314bp}%
    \ifx\svgscale\undefined%
      \relax%
    \else%
      \setlength{\unitlength}{\unitlength * \real{\svgscale}}%
    \fi%
  \else%
    \setlength{\unitlength}{\svgwidth}%
  \fi%
  \global\let\svgwidth\undefined%
  \global\let\svgscale\undefined%
  \makeatother%
  \begin{picture}(1,0.43927891)%
    \put(0,0){\includegraphics[width=\unitlength,page=1]{figure_II_4_1_1.pdf}}%
    \put(-0.00086188,0.00235211){\color[rgb]{0,0,0}\makebox(0,0)[lb]{\smash{$Y$ : blow-up of $\p^3$ along two disjoint lines}}}%
    \put(0.90035515,0.35534898){\color[rgb]{0,0,0}\makebox(0,0)[lb]{\smash{$(2,0,2)$}}}%
    \put(0.94590708,0.30979705){\color[rgb]{0,0,0}\makebox(0,0)[lb]{\smash{$(3,0,1)$}}}%
    \put(0.94487234,0.23974358){\color[rgb]{0,0,0}\makebox(0,0)[lb]{\smash{$(3,0,0)$}}}%
    \put(0.58781576,0.08065588){\color[rgb]{0,0,0}\makebox(0,0)[lb]{\smash{$(0,3,0)$}}}%
    \put(0.58724055,0.18867122){\color[rgb]{0,0,0}\makebox(0,0)[lb]{\smash{$(0,3,1)$}}}%
    \put(0.61496192,0.26620187){\color[rgb]{0,0,0}\makebox(0,0)[lb]{\smash{$(0,2,2)$}}}%
    \put(0.69157352,0.31842471){\color[rgb]{0,0,0}\makebox(0,0)[lb]{\smash{$(0,1,2)$}}}%
    \put(0.7970569,0.35822489){\color[rgb]{0,0,0}\makebox(0,0)[lb]{\smash{$(1,0,2)$}}}%
    \put(0.8028087,0.27666785){\color[rgb]{0,0,0}\makebox(0,0)[lb]{\smash{$(1,0,1)$}}}%
    \put(0.7128552,0.23054075){\color[rgb]{0,0,0}\makebox(0,0)[lb]{\smash{$(0,1,1)$}}}%
    \put(0.72079498,0.1722221){\color[rgb]{0,0,0}\makebox(0,0)[lb]{\smash{$(0,2,0)$}}}%
    \put(0.82247754,0.22593933){\color[rgb]{0,0,0}\makebox(0,0)[lb]{\smash{$(2,0,0)$}}}%
    \put(0.62788183,0.00332645){\color[rgb]{0,0,0}\makebox(0,0)[lb]{\smash{$M$ : blow-up of $Y$ along $C_1$ and $C_2$}}}%
  \end{picture}%
\endgroup%

%% file: figure_II_4_1_2.pdf_tex
\begingroup%
  \makeatletter%
  \providecommand\color[2][]{%
    \errmessage{(Inkscape) Color is used for the text in Inkscape, but the package 'color.sty' is not loaded}%
    \renewcommand\color[2][]{}%
  }%
  \providecommand\transparent[1]{%
    \errmessage{(Inkscape) Transparency is used (non-zero) for the text in Inkscape, but the package 'transparent.sty' is not loaded}%
    \renewcommand\transparent[1]{}%
  }%
  \providecommand\rotatebox[2]{#2}%
  \ifx\svgwidth\undefined%
    \setlength{\unitlength}{132.33945313bp}%
    \ifx\svgscale\undefined%
      \relax%
    \else%
      \setlength{\unitlength}{\unitlength * \real{\svgscale}}%
    \fi%
  \else%
    \setlength{\unitlength}{\svgwidth}%
  \fi%
  \global\let\svgwidth\undefined%
  \global\let\svgscale\undefined%
  \makeatother%
  \begin{picture}(1,0.9682783)%
    \put(0,0){\includegraphics[width=\unitlength,page=1]{figure_II_4_1_2.pdf}}%
    \put(0.78384297,0.52039916){\color[rgb]{0,0,0}\makebox(0,0)[lb]{\smash{$(2,0,0)$}}}%
    \put(0.78384297,0.77196725){\color[rgb]{0,0,0}\makebox(0,0)[lb]{\smash{$(2,0,2)$}}}%
    \put(0.46387667,0.77349382){\color[rgb]{0,0,0}\makebox(0,0)[lb]{\smash{$(1,0,2)$}}}%
    \put(0.65866882,0.30179499){\color[rgb]{0,0,0}\makebox(0,0)[lb]{\smash{$(2,1,0)$}}}%
    \put(0.3380965,0.15799494){\color[rgb]{0,0,0}\makebox(0,0)[lb]{\smash{$(1,2,0)$}}}%
    \put(0.00098513,0.18819147){\color[rgb]{0,0,0}\makebox(0,0)[lb]{\smash{$(0,2,0)$}}}%
    \put(0.20528153,0.67396506){\color[rgb]{0,0,0}\makebox(0,0)[lb]{\smash{$(0,1,2)$}}}%
    \put(-0.00171567,0.49688619){\color[rgb]{0,0,0}\makebox(0,0)[lb]{\smash{$(0,2,2)$}}}%
  \end{picture}%
\endgroup%

%% file: figure_II_4_1_22.pdf_tex
\begingroup%
  \makeatletter%
  \providecommand\color[2][]{%
    \errmessage{(Inkscape) Color is used for the text in Inkscape, but the package 'color.sty' is not loaded}%
    \renewcommand\color[2][]{}%
  }%
  \providecommand\transparent[1]{%
    \errmessage{(Inkscape) Transparency is used (non-zero) for the text in Inkscape, but the package 'transparent.sty' is not loaded}%
    \renewcommand\transparent[1]{}%
  }%
  \providecommand\rotatebox[2]{#2}%
  \ifx\svgwidth\undefined%
    \setlength{\unitlength}{221.94904785bp}%
    \ifx\svgscale\undefined%
      \relax%
    \else%
      \setlength{\unitlength}{\unitlength * \real{\svgscale}}%
    \fi%
  \else%
    \setlength{\unitlength}{\svgwidth}%
  \fi%
  \global\let\svgwidth\undefined%
  \global\let\svgscale\undefined%
  \makeatother%
  \begin{picture}(1,0.53542639)%
    \put(0,0){\includegraphics[width=\unitlength,page=1]{figure_II_4_1_22.pdf}}%
    \put(0.39986226,0.43491416){\color[rgb]{0,0,0}\makebox(0,0)[lb]{\smash{$(4,4)$}}}%
    \put(0.39986226,0.23667035){\color[rgb]{0,0,0}\makebox(0,0)[lb]{\smash{$(4,2)$}}}%
    \put(0.19651837,0.43855522){\color[rgb]{0,0,0}\makebox(0,0)[lb]{\smash{$(2,4)$}}}%
    \put(0.19469779,0.07974979){\color[rgb]{0,0,0}\makebox(0,0)[lb]{\smash{$(2,0)$}}}%
    \put(-0.00136398,0.08066014){\color[rgb]{0,0,0}\makebox(0,0)[lb]{\smash{$(0,0)$}}}%
    \put(0.00045649,0.28163433){\color[rgb]{0,0,0}\makebox(0,0)[lb]{\smash{$(0,2)$}}}%
    \put(0,0){\includegraphics[width=\unitlength,page=2]{figure_II_4_1_22.pdf}}%
    \put(0.47195088,0.52502493){\color[rgb]{0,0,0}\makebox(0,0)[lb]{\smash{$A$}}}%
    \put(0,0){\includegraphics[width=\unitlength,page=3]{figure_II_4_1_22.pdf}}%
    \put(0.30975149,0.00238257){\color[rgb]{0,0,0}\makebox(0,0)[lb]{\smash{$B$}}}%
    \put(0,0){\includegraphics[width=\unitlength,page=4]{figure_II_4_1_22.pdf}}%
    \put(0.77233038,0.06934965){\color[rgb]{0,0,0}\makebox(0,0)[lb]{\smash{$\xi = (1,0)$}}}%
  \end{picture}%
\endgroup%

%% file: figure_II_4_1_222.pdf_tex
\begingroup%
  \makeatletter%
  \providecommand\color[2][]{%
    \errmessage{(Inkscape) Color is used for the text in Inkscape, but the package 'color.sty' is not loaded}%
    \renewcommand\color[2][]{}%
  }%
  \providecommand\transparent[1]{%
    \errmessage{(Inkscape) Transparency is used (non-zero) for the text in Inkscape, but the package 'transparent.sty' is not loaded}%
    \renewcommand\transparent[1]{}%
  }%
  \providecommand\rotatebox[2]{#2}%
  \ifx\svgwidth\undefined%
    \setlength{\unitlength}{307.13828544bp}%
    \ifx\svgscale\undefined%
      \relax%
    \else%
      \setlength{\unitlength}{\unitlength * \real{\svgscale}}%
    \fi%
  \else%
    \setlength{\unitlength}{\svgwidth}%
  \fi%
  \global\let\svgwidth\undefined%
  \global\let\svgscale\undefined%
  \makeatother%
  \begin{picture}(1,0.46837573)%
    \put(0,0){\includegraphics[width=\unitlength,page=1]{figure_II_4_1_222.pdf}}%
    \put(0.93813861,0.20837519){\color[rgb]{0,0,0}\makebox(0,0)[lb]{\smash{$(3,0,0)$}}}%
    \put(0.93813861,0.32558624){\color[rgb]{0,0,0}\makebox(0,0)[lb]{\smash{$(3,0,1)$}}}%
    \put(0.88604482,0.37767996){\color[rgb]{0,0,0}\makebox(0,0)[lb]{\smash{$(2,0,2)$}}}%
    \put(0.7652804,0.37570657){\color[rgb]{0,0,0}\makebox(0,0)[lb]{\smash{$(1,0,2)$}}}%
    \put(0.52033043,0.05275149){\color[rgb]{0,0,0}\makebox(0,0)[lb]{\smash{$(0,3,0)$}}}%
    \put(0.51862488,0.14983581){\color[rgb]{0,0,0}\makebox(0,0)[lb]{\smash{$(0,3,1)$}}}%
    \put(0.69095419,0.16470022){\color[rgb]{0,0,0}\makebox(0,0)[lb]{\smash{$(0,1,0)$}}}%
    \put(0.64859491,0.32755963){\color[rgb]{0,0,0}\makebox(0,0)[lb]{\smash{$(0,1,2)$}}}%
    \put(0.57032536,0.2540234){\color[rgb]{0,0,0}\makebox(0,0)[lb]{\smash{$(0,2,2)$}}}%
    \put(0.79869514,0.21495305){\color[rgb]{0,0,0}\makebox(0,0)[lb]{\smash{$(1,0,0)$}}}%
    \put(0,0){\includegraphics[width=\unitlength,page=2]{figure_II_4_1_222.pdf}}%
    \put(0.35208338,0.45582066){\color[rgb]{0,0,0}\makebox(0,0)[lb]{\smash{$A$}}}%
  \end{picture}%
\endgroup%

%% file: figure_II_4_2_1_1.pdf_tex
\begingroup%
  \makeatletter%
  \providecommand\color[2][]{%
    \errmessage{(Inkscape) Color is used for the text in Inkscape, but the package 'color.sty' is not loaded}%
    \renewcommand\color[2][]{}%
  }%
  \providecommand\transparent[1]{%
    \errmessage{(Inkscape) Transparency is used (non-zero) for the text in Inkscape, but the package 'transparent.sty' is not loaded}%
    \renewcommand\transparent[1]{}%
  }%
  \providecommand\rotatebox[2]{#2}%
  \ifx\svgwidth\undefined%
    \setlength{\unitlength}{297.07102661bp}%
    \ifx\svgscale\undefined%
      \relax%
    \else%
      \setlength{\unitlength}{\unitlength * \real{\svgscale}}%
    \fi%
  \else%
    \setlength{\unitlength}{\svgwidth}%
  \fi%
  \global\let\svgwidth\undefined%
  \global\let\svgscale\undefined%
  \makeatother%
  \begin{picture}(1,0.39047893)%
    \put(0,0){\includegraphics[width=\unitlength,page=1]{figure_II_4_2_1_1.pdf}}%
    \put(0.24160198,0.38270774){\color[rgb]{0,0,0}\makebox(0,0)[lb]{\smash{$(0,3)$}}}%
    \put(0.44357388,0.19420072){\color[rgb]{0,0,0}\makebox(0,0)[lb]{\smash{$(3,0)$}}}%
    \put(0.24160198,0.00203155){\color[rgb]{0,0,0}\makebox(0,0)[lb]{\smash{$(0,-3)$}}}%
    \put(-0.0007643,0.19420072){\color[rgb]{0,0,0}\makebox(0,0)[lb]{\smash{$(-3,0)$}}}%
    \put(0,0){\includegraphics[width=\unitlength,page=2]{figure_II_4_2_1_1.pdf}}%
    \put(0.95753462,0.19149083){\color[rgb]{0,0,0}\makebox(0,0)[lb]{\smash{$(2,0)$}}}%
    \put(0.62091476,0.19515288){\color[rgb]{0,0,0}\makebox(0,0)[lb]{\smash{$(-2,0)$}}}%
    \put(0.95753462,0.12782892){\color[rgb]{0,0,0}\makebox(0,0)[lb]{\smash{$(2,-1)$}}}%
    \put(0.62091476,0.2624766){\color[rgb]{0,0,0}\makebox(0,0)[lb]{\smash{$(-2,1)$}}}%
  \end{picture}%
\endgroup%

%% file: figure_II_4_2_1_2.pdf_tex
\begingroup%
  \makeatletter%
  \providecommand\color[2][]{%
    \errmessage{(Inkscape) Color is used for the text in Inkscape, but the package 'color.sty' is not loaded}%
    \renewcommand\color[2][]{}%
  }%
  \providecommand\transparent[1]{%
    \errmessage{(Inkscape) Transparency is used (non-zero) for the text in Inkscape, but the package 'transparent.sty' is not loaded}%
    \renewcommand\transparent[1]{}%
  }%
  \providecommand\rotatebox[2]{#2}%
  \ifx\svgwidth\undefined%
    \setlength{\unitlength}{360.17294922bp}%
    \ifx\svgscale\undefined%
      \relax%
    \else%
      \setlength{\unitlength}{\unitlength * \real{\svgscale}}%
    \fi%
  \else%
    \setlength{\unitlength}{\svgwidth}%
  \fi%
  \global\let\svgwidth\undefined%
  \global\let\svgscale\undefined%
  \makeatother%
  \begin{picture}(1,0.26351815)%
    \put(0,0){\includegraphics[width=\unitlength,page=1]{figure_II_4_2_1_2.pdf}}%
    \put(0.35602064,0.08833376){\color[rgb]{0,0,0}\makebox(0,0)[lb]{\smash{projection}}}%
    \put(0,0){\includegraphics[width=\unitlength,page=2]{figure_II_4_2_1_2.pdf}}%
    \put(0.66571615,0.09052183){\color[rgb]{0,0,0}\makebox(0,0)[lb]{\smash{blow-up along $C$}}}%
    \put(0,0){\includegraphics[width=\unitlength,page=3]{figure_II_4_2_1_2.pdf}}%
    \put(0.18816779,0.25710847){\color[rgb]{0,0,0}\makebox(0,0)[lb]{\smash{$\mu(C)$}}}%
    \put(0.2325909,0.10162761){\color[rgb]{0,0,0}\makebox(0,0)[lb]{\smash{$(2,0,0)$}}}%
    \put(-0.00063039,0.0127814){\color[rgb]{0,0,0}\makebox(0,0)[lb]{\smash{$(0,2,0)$}}}%
    \put(0.09932159,0.22379114){\color[rgb]{0,0,0}\makebox(0,0)[lb]{\smash{$(0,0,2)$}}}%
    \put(0,0){\includegraphics[width=\unitlength,page=4]{figure_II_4_2_1_2.pdf}}%
    \put(0.49912951,0.22379114){\color[rgb]{0,0,0}\makebox(0,0)[lb]{\smash{$(0,2)$}}}%
    \put(0.48802374,0.00167563){\color[rgb]{0,0,0}\makebox(0,0)[lb]{\smash{$(0,-2)$}}}%
    \put(0.61018727,0.00167563){\color[rgb]{0,0,0}\makebox(0,0)[lb]{\smash{$(2,-2)$}}}%
    \put(0.61018727,0.22379114){\color[rgb]{0,0,0}\makebox(0,0)[lb]{\smash{$(2,2)$}}}%
  \end{picture}%
\endgroup%

%% file: figure_II_4_2_2_2.pdf_tex
\begingroup%
  \makeatletter%
  \providecommand\color[2][]{%
    \errmessage{(Inkscape) Color is used for the text in Inkscape, but the package 'color.sty' is not loaded}%
    \renewcommand\color[2][]{}%
  }%
  \providecommand\transparent[1]{%
    \errmessage{(Inkscape) Transparency is used (non-zero) for the text in Inkscape, but the package 'transparent.sty' is not loaded}%
    \renewcommand\transparent[1]{}%
  }%
  \providecommand\rotatebox[2]{#2}%
  \ifx\svgwidth\undefined%
    \setlength{\unitlength}{275.38254395bp}%
    \ifx\svgscale\undefined%
      \relax%
    \else%
      \setlength{\unitlength}{\unitlength * \real{\svgscale}}%
    \fi%
  \else%
    \setlength{\unitlength}{\svgwidth}%
  \fi%
  \global\let\svgwidth\undefined%
  \global\let\svgscale\undefined%
  \makeatother%
  \begin{picture}(1,0.37724334)%
    \put(0,0){\includegraphics[width=\unitlength,page=1]{figure_II_4_2_2_2.pdf}}%
    \put(0.42262136,0.36313113){\color[rgb]{0,0,0}\makebox(0,0)[lb]{\smash{$(4,4)$}}}%
    \put(0.42262136,0.20335331){\color[rgb]{0,0,0}\makebox(0,0)[lb]{\smash{$(4,2)$}}}%
    \put(0.25873304,0.36606571){\color[rgb]{0,0,0}\makebox(0,0)[lb]{\smash{$(2,4)$}}}%
    \put(0.25726571,0.07688062){\color[rgb]{0,0,0}\makebox(0,0)[lb]{\smash{$(2,0)$}}}%
    \put(0.09924653,0.07761433){\color[rgb]{0,0,0}\makebox(0,0)[lb]{\smash{$(0,0)$}}}%
    \put(0.10071377,0.23959276){\color[rgb]{0,0,0}\makebox(0,0)[lb]{\smash{$(0,2)$}}}%
    \put(0,0){\includegraphics[width=\unitlength,page=2]{figure_II_4_2_2_2.pdf}}%
    \put(0.94553032,0.34860589){\color[rgb]{0,0,0}\makebox(0,0)[lb]{\smash{$(4,4)$}}}%
    \put(0.94553032,0.18882806){\color[rgb]{0,0,0}\makebox(0,0)[lb]{\smash{$(4,2)$}}}%
    \put(0.781642,0.35154046){\color[rgb]{0,0,0}\makebox(0,0)[lb]{\smash{$(2,4)$}}}%
    \put(0.78017467,0.06235538){\color[rgb]{0,0,0}\makebox(0,0)[lb]{\smash{$(2,0)$}}}%
    \put(0,0){\includegraphics[width=\unitlength,page=3]{figure_II_4_2_2_2.pdf}}%
    \put(-0.00109932,0.16137149){\color[rgb]{0,0,0}\makebox(0,0)[lb]{\smash{$\mu(C)$}}}%
    \put(0.4661971,-0){\color[rgb]{0,0,0}\makebox(0,0)[lb]{\smash{     }}}%
  \end{picture}%
\endgroup%

%% file: figure_II_4_2_2_1.pdf_tex
\begingroup%
  \makeatletter%
  \providecommand\color[2][]{%
    \errmessage{(Inkscape) Color is used for the text in Inkscape, but the package 'color.sty' is not loaded}%
    \renewcommand\color[2][]{}%
  }%
  \providecommand\transparent[1]{%
    \errmessage{(Inkscape) Transparency is used (non-zero) for the text in Inkscape, but the package 'transparent.sty' is not loaded}%
    \renewcommand\transparent[1]{}%
  }%
  \providecommand\rotatebox[2]{#2}%
  \ifx\svgwidth\undefined%
    \setlength{\unitlength}{120.48994141bp}%
    \ifx\svgscale\undefined%
      \relax%
    \else%
      \setlength{\unitlength}{\unitlength * \real{\svgscale}}%
    \fi%
  \else%
    \setlength{\unitlength}{\svgwidth}%
  \fi%
  \global\let\svgwidth\undefined%
  \global\let\svgscale\undefined%
  \makeatother%
  \begin{picture}(1,0.79792072)%
    \put(0,0){\includegraphics[width=\unitlength,page=1]{figure_II_4_2_2_1.pdf}}%
    \put(0.88173287,0.38010037){\color[rgb]{0,0,0}\makebox(0,0)[lb]{\smash{$(2,0,0)$}}}%
    \put(0.88173287,0.51289154){\color[rgb]{0,0,0}\makebox(0,0)[lb]{\smash{$(2,0,1)$}}}%
    \put(0.73888125,0.69832726){\color[rgb]{0,0,0}\makebox(0,0)[lb]{\smash{$(1,0,2)$}}}%
    \put(0.40958282,0.70000397){\color[rgb]{0,0,0}\makebox(0,0)[lb]{\smash{$(0,0,2)$}}}%
    \put(0.50817342,0.03068105){\color[rgb]{0,0,0}\makebox(0,0)[lb]{\smash{$(2,2,0)$}}}%
    \put(0.01724202,0.06589251){\color[rgb]{0,0,0}\makebox(0,0)[lb]{\smash{$(0,2,0)$}}}%
    \put(0.01724202,0.39787043){\color[rgb]{0,0,0}\makebox(0,0)[lb]{\smash{$(0,2,2)$}}}%
    \put(0,0){\includegraphics[width=\unitlength,page=2]{figure_II_4_2_2_1.pdf}}%
    \put(-0.0018844,0.62021134){\color[rgb]{0,0,0}\makebox(0,0)[lb]{\smash{$(1,2,2)$}}}%
    \put(0,0){\includegraphics[width=\unitlength,page=3]{figure_II_4_2_2_1.pdf}}%
  \end{picture}%
\endgroup%

%% file: figure_II_4_2_3.pdf_tex
\begingroup%
  \makeatletter%
  \providecommand\color[2][]{%
    \errmessage{(Inkscape) Color is used for the text in Inkscape, but the package 'color.sty' is not loaded}%
    \renewcommand\color[2][]{}%
  }%
  \providecommand\transparent[1]{%
    \errmessage{(Inkscape) Transparency is used (non-zero) for the text in Inkscape, but the package 'transparent.sty' is not loaded}%
    \renewcommand\transparent[1]{}%
  }%
  \providecommand\rotatebox[2]{#2}%
  \ifx\svgwidth\undefined%
    \setlength{\unitlength}{350.77294617bp}%
    \ifx\svgscale\undefined%
      \relax%
    \else%
      \setlength{\unitlength}{\unitlength * \real{\svgscale}}%
    \fi%
  \else%
    \setlength{\unitlength}{\svgwidth}%
  \fi%
  \global\let\svgwidth\undefined%
  \global\let\svgscale\undefined%
  \makeatother%
  \begin{picture}(1,0.42918047)%
    \put(0,0){\includegraphics[width=\unitlength,page=1]{figure_II_4_2_3.pdf}}%
    \put(0.17709913,0.42040521){\color[rgb]{0,0,0}\makebox(0,0)[lb]{\smash{$(0,0,4)$}}}%
    \put(0,0){\includegraphics[width=\unitlength,page=2]{figure_II_4_2_3.pdf}}%
    \put(0.79288208,0.22654761){\color[rgb]{0,0,0}\makebox(0,0)[lb]{\smash{$(0,0,0)$}}}%
    \put(0,0){\includegraphics[width=\unitlength,page=3]{figure_II_4_2_3.pdf}}%
    \put(0.53670536,0.04812516){\color[rgb]{0,0,0}\makebox(0,0)[lb]{\smash{$(0,3,0)$}}}%
    \put(0.53669255,0.10596312){\color[rgb]{0,0,0}\makebox(0,0)[lb]{\smash{$(0,3,1)$}}}%
    \put(0.94583391,0.26399838){\color[rgb]{0,0,0}\makebox(0,0)[lb]{\smash{$(3,0,1)$}}}%
    \put(0.94583391,0.19557806){\color[rgb]{0,0,0}\makebox(0,0)[lb]{\smash{$(3,0,0)$}}}%
    \put(0,0){\includegraphics[width=\unitlength,page=4]{figure_II_4_2_3.pdf}}%
    \put(0.75414717,0.32998825){\color[rgb]{0,0,0}\makebox(0,0)[lb]{\smash{$(0,0,2)$}}}%
    \put(0.85863851,0.33011486){\color[rgb]{0,0,0}\makebox(0,0)[lb]{\smash{$(2,0,2)$}}}%
    \put(0.59232889,0.20698144){\color[rgb]{0,0,0}\makebox(0,0)[lb]{\smash{$(0,2,2)$}}}%
    \put(0,0){\includegraphics[width=\unitlength,page=5]{figure_II_4_2_3.pdf}}%
    \put(0.40987468,0.19557806){\color[rgb]{0,0,0}\makebox(0,0)[lb]{\smash{$(4,0,0)$}}}%
    \put(-0.00064729,0.00172053){\color[rgb]{0,0,0}\makebox(0,0)[lb]{\smash{$(0,4,0)$}}}%
    \put(0.40987468,-0.00968286){\color[rgb]{0,0,0}\makebox(0,0)[lb]{\smash{}}}%
    \put(0.39847129,-0.04389302){\color[rgb]{0,0,0}\makebox(0,0)[lb]{\smash{}}}%
  \end{picture}%
\endgroup%

%% file: figure_II_4_2_4.pdf_tex
\begingroup%
  \makeatletter%
  \providecommand\color[2][]{%
    \errmessage{(Inkscape) Color is used for the text in Inkscape, but the package 'color.sty' is not loaded}%
    \renewcommand\color[2][]{}%
  }%
  \providecommand\transparent[1]{%
    \errmessage{(Inkscape) Transparency is used (non-zero) for the text in Inkscape, but the package 'transparent.sty' is not loaded}%
    \renewcommand\transparent[1]{}%
  }%
  \providecommand\rotatebox[2]{#2}%
  \ifx\svgwidth\undefined%
    \setlength{\unitlength}{553.26147461bp}%
    \ifx\svgscale\undefined%
      \relax%
    \else%
      \setlength{\unitlength}{\unitlength * \real{\svgscale}}%
    \fi%
  \else%
    \setlength{\unitlength}{\svgwidth}%
  \fi%
  \global\let\svgwidth\undefined%
  \global\let\svgscale\undefined%
  \makeatother%
  \begin{picture}(1,0.27364536)%
    \put(0,0){\includegraphics[width=\unitlength,page=1]{figure_II_4_2_4.pdf}}%
    \put(0.11813602,0.26808176){\color[rgb]{0,0,0}\makebox(0,0)[lb]{\smash{$(0,0,4)$}}}%
    \put(0,0){\includegraphics[width=\unitlength,page=2]{figure_II_4_2_4.pdf}}%
    \put(0.2657179,0.12553925){\color[rgb]{0,0,0}\makebox(0,0)[lb]{\smash{$(4,0,0)$}}}%
    \put(-0.00041039,0.00109083){\color[rgb]{0,0,0}\makebox(0,0)[lb]{\smash{$(0,4,0)$}}}%
    \put(0,0){\includegraphics[width=\unitlength,page=3]{figure_II_4_2_4.pdf}}%
    \put(0.63444047,0.12553925){\color[rgb]{0,0,0}\makebox(0,0)[lb]{\smash{$(4,0,0)$}}}%
    \put(0.3683122,0.00109083){\color[rgb]{0,0,0}\makebox(0,0)[lb]{\smash{$(0,4,0)$}}}%
    \put(0,0){\includegraphics[width=\unitlength,page=4]{figure_II_4_2_4.pdf}}%
    \put(0.47676002,0.20352672){\color[rgb]{0,0,0}\makebox(0,0)[lb]{\smash{$(0,0,2)$}}}%
    \put(0.56937179,0.20506765){\color[rgb]{0,0,0}\makebox(0,0)[lb]{\smash{$(2,0,2)$}}}%
    \put(0.39585528,0.13999896){\color[rgb]{0,0,0}\makebox(0,0)[lb]{\smash{$(0,2,2)$}}}%
    \put(0,0){\includegraphics[width=\unitlength,page=5]{figure_II_4_2_4.pdf}}%
    \put(0.97424364,0.1327691){\color[rgb]{0,0,0}\makebox(0,0)[lb]{\smash{$(4,0,0)$}}}%
    \put(0.70811536,0.00832069){\color[rgb]{0,0,0}\makebox(0,0)[lb]{\smash{$(0,4,0)$}}}%
    \put(0,0){\includegraphics[width=\unitlength,page=6]{figure_II_4_2_4.pdf}}%
    \put(0.81656318,0.21075657){\color[rgb]{0,0,0}\makebox(0,0)[lb]{\smash{$(0,0,2)$}}}%
    \put(0.88163187,0.21075657){\color[rgb]{0,0,0}\makebox(0,0)[lb]{\smash{$(1,0,2)$}}}%
    \put(0.70088551,0.08784904){\color[rgb]{0,0,0}\makebox(0,0)[lb]{\smash{$(0,3,1)$}}}%
    \put(0,0){\includegraphics[width=\unitlength,page=7]{figure_II_4_2_4.pdf}}%
    \put(0.94532422,0.18337809){\color[rgb]{0,0,0}\makebox(0,0)[lb]{\smash{$(3,0,1)$}}}%
    \put(0.75734801,0.16891838){\color[rgb]{0,0,0}\makebox(0,0)[lb]{\smash{$(0,1,2)$}}}%
    \put(0,0){\includegraphics[width=\unitlength,page=8]{figure_II_4_2_4.pdf}}%
    \put(0.62721062,-0.01182794){\color[rgb]{0,0,0}\makebox(0,0)[lb]{\smash{}}}%
  \end{picture}%
\endgroup%

%% file: figure_II_4_2_5.pdf_tex
\begingroup%
  \makeatletter%
  \providecommand\color[2][]{%
    \errmessage{(Inkscape) Color is used for the text in Inkscape, but the package 'color.sty' is not loaded}%
    \renewcommand\color[2][]{}%
  }%
  \providecommand\transparent[1]{%
    \errmessage{(Inkscape) Transparency is used (non-zero) for the text in Inkscape, but the package 'transparent.sty' is not loaded}%
    \renewcommand\transparent[1]{}%
  }%
  \providecommand\rotatebox[2]{#2}%
  \ifx\svgwidth\undefined%
    \setlength{\unitlength}{362.97502441bp}%
    \ifx\svgscale\undefined%
      \relax%
    \else%
      \setlength{\unitlength}{\unitlength * \real{\svgscale}}%
    \fi%
  \else%
    \setlength{\unitlength}{\svgwidth}%
  \fi%
  \global\let\svgwidth\undefined%
  \global\let\svgscale\undefined%
  \makeatother%
  \begin{picture}(1,0.48022194)%
    \put(0,0){\includegraphics[width=\unitlength,page=1]{figure_II_4_2_5.pdf}}%
    \put(0.41869279,0.27772732){\color[rgb]{0,0,0}\makebox(0,0)[lb]{\smash{$(4,0,0)$}}}%
    \put(0.27253739,0.44469777){\color[rgb]{0,0,0}\makebox(0,0)[lb]{\smash{$(1,0,3)$}}}%
    \put(0.27921655,0.27550091){\color[rgb]{0,0,0}\makebox(0,0)[lb]{\smash{$(1,0,0)$}}}%
    \put(0.1754715,0.22919441){\color[rgb]{0,0,0}\makebox(0,0)[lb]{\smash{$(0,1,0)$}}}%
    \put(0.1046719,0.39650049){\color[rgb]{0,0,0}\makebox(0,0)[lb]{\smash{$(0,1,3)$}}}%
    \put(-0.00083404,0.08502866){\color[rgb]{0,0,0}\makebox(0,0)[lb]{\smash{$(0,4,0)$}}}%
    \put(0.07240512,0.02995805){\color[rgb]{0,0,0}\makebox(0,0)[lb]{\smash{(a) $Y$ : ~Blow-up of $\p^3$ along a line }}}%
    \put(0,0){\includegraphics[width=\unitlength,page=2]{figure_II_4_2_5.pdf}}%
    \put(0.9476548,0.27772732){\color[rgb]{0,0,0}\makebox(0,0)[lb]{\smash{$(4,0,0)$}}}%
    \put(0.52812798,0.08502866){\color[rgb]{0,0,0}\makebox(0,0)[lb]{\smash{$(0,4,0)$}}}%
    \put(0.60136713,0.02995805){\color[rgb]{0,0,0}\makebox(0,0)[lb]{\smash{(b) $M$ : blow-up of $Y$ along $C_1$ and $C_2$}}}%
    \put(0,0){\includegraphics[width=\unitlength,page=3]{figure_II_4_2_5.pdf}}%
    \put(0.14697604,0.47386174){\color[rgb]{0,0,0}\makebox(0,0)[lb]{\smash{$C_1$}}}%
    \put(0,0){\includegraphics[width=\unitlength,page=4]{figure_II_4_2_5.pdf}}%
    \put(0.06983575,0.27550105){\color[rgb]{0,0,0}\makebox(0,0)[lb]{\smash{$C_2$}}}%
    \put(0.36737688,-0){\color[rgb]{0,0,0}\makebox(0,0)[lb]{\smash{    }}}%
    \put(0.85225873,0.39672144){\color[rgb]{0,0,0}\makebox(0,0)[lb]{\smash{$(2,0,2)$}}}%
    \put(0.75307835,0.39672144){\color[rgb]{0,0,0}\makebox(0,0)[lb]{\smash{$(1,0,2)$}}}%
    \put(0.63185789,0.34162123){\color[rgb]{0,0,0}\makebox(0,0)[lb]{\smash{$(0,1,2)$}}}%
    \put(0.57675768,0.28652102){\color[rgb]{0,0,0}\makebox(0,0)[lb]{\smash{$(0,2,2)$}}}%
    \put(0.69330162,0.27550105){\color[rgb]{0,0,0}\makebox(0,0)[lb]{\smash{$(0,1,1)$}}}%
    \put(0.7721117,0.31357051){\color[rgb]{0,0,0}\makebox(0,0)[lb]{\smash{$(1,0,1)$}}}%
    \put(0.84123869,0.26993507){\color[rgb]{0,0,0}\makebox(0,0)[lb]{\smash{$(2,0,0)$}}}%
    \put(0.6449758,0.18499196){\color[rgb]{0,0,0}\makebox(0,0)[lb]{\smash{$(0,2,0)$}}}%
    \put(0,0){\includegraphics[width=\unitlength,page=5]{figure_II_4_2_5.pdf}}%
  \end{picture}%
\endgroup%

%% file: figure_II_4_2_6.pdf_tex
\begingroup%
  \makeatletter%
  \providecommand\color[2][]{%
    \errmessage{(Inkscape) Color is used for the text in Inkscape, but the package 'color.sty' is not loaded}%
    \renewcommand\color[2][]{}%
  }%
  \providecommand\transparent[1]{%
    \errmessage{(Inkscape) Transparency is used (non-zero) for the text in Inkscape, but the package 'transparent.sty' is not loaded}%
    \renewcommand\transparent[1]{}%
  }%
  \providecommand\rotatebox[2]{#2}%
  \ifx\svgwidth\undefined%
    \setlength{\unitlength}{353.2614769bp}%
    \ifx\svgscale\undefined%
      \relax%
    \else%
      \setlength{\unitlength}{\unitlength * \real{\svgscale}}%
    \fi%
  \else%
    \setlength{\unitlength}{\svgwidth}%
  \fi%
  \global\let\svgwidth\undefined%
  \global\let\svgscale\undefined%
  \makeatother%
  \begin{picture}(1,0.48210343)%
    \put(0,0){\includegraphics[width=\unitlength,page=1]{figure_II_4_2_6.pdf}}%
    \put(0.41615484,0.30572248){\color[rgb]{0,0,0}\makebox(0,0)[lb]{\smash{$(4,0,0)$}}}%
    \put(-0.00064273,0.11081724){\color[rgb]{0,0,0}\makebox(0,0)[lb]{\smash{$(0,4,0)$}}}%
    \put(0.16920313,0.42786278){\color[rgb]{0,0,0}\makebox(0,0)[lb]{\smash{$(0,0,2)$}}}%
    \put(0.31424733,0.43027611){\color[rgb]{0,0,0}\makebox(0,0)[lb]{\smash{$(2,0,2)$}}}%
    \put(0.04249395,0.3283686){\color[rgb]{0,0,0}\makebox(0,0)[lb]{\smash{$(0,2,2)$}}}%
    \put(0.40483179,0.09058446){\color[rgb]{0,0,0}\makebox(0,0)[lb]{\smash{}}}%
    \put(0,0){\includegraphics[width=\unitlength,page=2]{figure_II_4_2_6.pdf}}%
    \put(0.06446278,0.05420195){\color[rgb]{0,0,0}\makebox(0,0)[lb]{\smash{(a) $V_7$: Blow-up of $\p^3$ at a point}}}%
    \put(0,0){\includegraphics[width=\unitlength,page=3]{figure_II_4_2_6.pdf}}%
    \put(0.95966161,0.30572248){\color[rgb]{0,0,0}\makebox(0,0)[lb]{\smash{$(4,0,0)$}}}%
    \put(0.54286403,0.11081724){\color[rgb]{0,0,0}\makebox(0,0)[lb]{\smash{$(0,4,0)$}}}%
    \put(0.85775409,0.43027611){\color[rgb]{0,0,0}\makebox(0,0)[lb]{\smash{$(2,0,2)$}}}%
    \put(0.5768332,0.3146322){\color[rgb]{0,0,0}\makebox(0,0)[lb]{\smash{$(0,2,2)$}}}%
    \put(0,0){\includegraphics[width=\unitlength,page=4]{figure_II_4_2_6.pdf}}%
    \put(0.63129294,0.37366083){\color[rgb]{0,0,0}\makebox(0,0)[lb]{\smash{$(0,1,2)$}}}%
    \put(0.75584657,0.43027611){\color[rgb]{0,0,0}\makebox(0,0)[lb]{\smash{$(1,0,2)$}}}%
    \put(0.81246185,0.30572248){\color[rgb]{0,0,0}\makebox(0,0)[lb]{\smash{$(1,0,0)$}}}%
    \put(0.72187739,0.26043025){\color[rgb]{0,0,0}\makebox(0,0)[lb]{\smash{$(0,1,0)$}}}%
    \put(0,0){\includegraphics[width=\unitlength,page=5]{figure_II_4_2_6.pdf}}%
    \put(0.25763204,0.47556834){\color[rgb]{0,0,0}\makebox(0,0)[lb]{\smash{$C$}}}%
    \put(0.63129294,0.05661529){\color[rgb]{0,0,0}\makebox(0,0)[lb]{\smash{(b) $M$ : Blow-up of $V_7$ along $C$}}}%
    \put(0.3482165,0){\color[rgb]{0,0,0}\makebox(0,0)[lb]{\smash{     }}}%
  \end{picture}%
\endgroup%